\documentclass[final,leqno,onefignum,onetabnum]{siamltex}
\usepackage[letterpaper,top=1in, bottom=1in, left=1in, right=1in]{geometry}

\usepackage{epsfig,epsf,fancybox}
\usepackage{amsmath}
\usepackage{mathrsfs}
\usepackage{amssymb}
\usepackage{amsfonts}
\usepackage{graphicx}
\usepackage{enumitem}
\usepackage{color}
\usepackage[linktocpage,colorlinks,linkcolor=blue,anchorcolor=blue,citecolor=blue,urlcolor=blue,hypertexnames=false]{hyperref}
\usepackage{boxedminipage}
\usepackage{stmaryrd}
\usepackage{multirow}
\usepackage{longtable}
\usepackage{rotating}
\usepackage{booktabs}
\usepackage{cite}
\usepackage{float}
\usepackage[ruled,vlined,linesnumbered]{algorithm2e}
\usepackage{algpseudocode}

\usepackage{mathtools}
\usepackage[normalem]{ulem} % to use \sout

\newcommand{\bm}[1]{\boldsymbol{#1}}

\newcommand{\va}{{\mathbf{a}}}
\newcommand{\vb}{{\mathbf{b}}}

\newcommand{\vd}{{\mathbf{d}}}
\newcommand{\ve}{{\mathbf{e}}}

\newcommand{\vu}{{\mathbf{u}}}
\newcommand{\vv}{{\mathbf{v}}}
\newcommand{\vw}{{\mathbf{w}}}
\newcommand{\vx}{{\mathbf{x}}}
\newcommand{\vy}{{\mathbf{y}}}
\newcommand{\vz}{{\mathbf{z}}}

\newcommand{\vA}{{\mathbf{A}}}

\newcommand{\vH}{{\mathbf{H}}}
\newcommand{\vI}{{\mathbf{I}}}

\newcommand{\vQ}{{\mathbf{Q}}}

\newcommand{\vW}{{\mathbf{W}}}

\newcommand{\vlam}{{\bm{\lambda}}}

\newcommand{\vxi}{{\bm{\xi}}}
\newcommand{\vmu}{{\bm{\mu}}}
\newcommand{\vSigma}{{\bm{\Sigma}}}

\newcommand{\cD}{{\mathcal{D}}}

\newcommand{\cL}{{\mathcal{L}}}

\newcommand{\cN}{{\mathcal{N}}}

\newcommand{\vareps}{\varepsilon}

\newcommand{\RR}{\mathbb{R}} % real
\newcommand{\vzero}{\mathbf{0}} % 0 vector
\newcommand{\vone}{{\mathbf{1}}} % 1 vector

\newcommand{\dist}{\mathrm{dist}}    % distance
\newcommand{\prox}{{\mathbf{prox}}} % proximal map
\newcommand{\dom}{{\mathrm{dom}}} % domain
\newcommand{\st}{\mbox{ s.t. }}

\DeclareMathOperator*{\argmin}{arg\,min} % argmin
\DeclareMathOperator*{\argmax}{arg\,max} % argmax

\newcommand{\bc}{\begin{center}}
\newcommand{\ec}{\end{center}}

\newcommand{\bdm}{\begin{displaymath}}
\newcommand{\edm}{\end{displaymath}}

\newcommand{\beq}{\begin{equation}}
\newcommand{\eeq}{\end{equation}}

\newcommand{\bfl}{\begin{flushleft}}
\newcommand{\efl}{\end{flushleft}}

\newcommand{\bt}{\begin{tabbing}}
\newcommand{\et}{\end{tabbing}}

\newcommand{\beqn}{\begin{eqnarray}}
\newcommand{\eeqn}{\end{eqnarray}}

\newcommand{\beqs}{\begin{align*}} % no equation numbers
\newcommand{\eeqs}{\end{align*}}  % no equation numbers

\newtheorem{assumption}{Assumption}
\newtheorem{remark}{Remark}
\newtheorem{setting}{Setting}

\numberwithin{equation}{section}

\begin{document}

\title{Inexact accelerated proximal gradient method with line search and reduced complexity for affine-constrained and bilinear saddle-point structured convex problems}

\author{Qihang Lin\thanks{\url{qihang-lin@uiowa.edu}, Department of Business Analytics, University of Iowa, Iowa City, IA 52242} \and Yangyang Xu\thanks{\url{xuy21@rpi.edu}, Department of Mathematical Sciences, Rensselaer Polytechnic Institute, Troy, NY 12180}}

%\institute{Yangyang Xu \at Department of Mathematical Sciences, Rensselaer Polytechnic Institute, Troy, NY 12180\\ \email{xuy21@rpi.edu}}

\date{\today}

\maketitle

\begin{abstract}
The goal of this paper is to reduce the total complexity of gradient-based methods for two classes of problems: affine-constrained composite convex optimization and bilinear saddle-point structured non-smooth convex optimization. Our technique is based on a double-loop inexact accelerated proximal gradient (APG) method for minimizing the summation of a non-smooth but proximable convex function and two smooth convex functions with different smoothness constants and computational costs. Compared to the standard APG method, the inexact APG method can reduce the total computation cost if  one smooth component has higher computational cost but a smaller smoothness constant than the other. With this property, the inexact APG method can be applied to approximately solve the subproblems of a proximal augmented Lagrangian method for affine-constrained composite convex optimization and the smooth approximation for bilinear saddle-point structured non-smooth convex optimization, where the smooth function with a smaller smoothness constant has significantly higher computational cost. Thus it can reduce total complexity for finding an approximately optimal/stationary solution. This technique is similar to the gradient sliding technique in literature \cite{lan2016gradient,lan2016accelerated}. The difference is that our inexact APG method can efficiently stop the inner loop by using a computable condition based on a measure of stationarity violation,  while the gradient sliding methods need to pre-specify the number of iterations for the inner loop. Numerical experiments demonstrate significantly higher efficiency of our methods over an optimal primal-dual first-order method in \cite{hamedani2021primal} and the gradient sliding methods.

\vspace{0.2cm}
	
\noindent {\bf Keywords:} first-order method, convex optimization, constrained optimization, saddle-point non-smooth optimization 
%\vspace{0.3cm}
	
%\noindent {\bf Mathematics Subject Classification:} 

\end{abstract}

\section{Introduction}\label{sec:intro}
In this paper, we consider two classes of convex optimization problems: affine-constrained composite convex optimization:
\vspace{-0.2cm}
\begin{equation}\label{eq:aff-prob-intro}
	\min_{\vx\in\RR^n} ~f(\vx) + r(\vx), \st \vA\vx = \vb,
	\vspace{-0.2cm}
\end{equation}
and bilinear saddle-point structured non-smooth optimization:
\vspace{-0.2cm}
\begin{equation}\label{eq:sp-prob-intro}
	\min_{\vx\in\RR^n} \left\{f(\vx) + r(\vx) + \max_{\vy\in \RR^m} \big[\langle \vy, \vA\vx\rangle - \phi(\vy) \big] \right\}.
	\vspace{-0.2cm}
\end{equation}
In both problems, $f$ is smooth convex while $r$ and $\phi$ are closed convex and admit easy proximal mappings. We further assume that $\phi$ has a bounded domain, $\nabla f$ is $L_f$-Lipschitz continuous, and $f$ is $\mu$-strongly convex with $\mu\geqslant 0$. For simplicity of introduction, we only consider equality constraints in \eqref{eq:aff-prob-intro} in this section, but we will consider both equality and inequality constraints in the main body of the paper as shown in \eqref{eq:aff-prob}.

Different from most of the existing works that target at an $\vareps$-optimal solution, we aim at finding $\vareps$-stationary solutions (defined later) of \eqref{eq:aff-prob-intro} and \eqref{eq:sp-prob-intro}, which can be verified in practice more easily than the former. Moreover, we consider gradient-based methods for solving the two classes of problems. The considered methods only need to evaluate  $(f, \nabla f)$ and $(\vA(\cdot), \vA^\top(\cdot))$ and use the proximal mappings of $r$ and $\phi$. We are interested in the \emph{oracle complexity} of the studied methods, which is defined as the numbers of queries that the methods make to $(f, \nabla f)$ and $(\vA(\cdot), \vA^\top(\cdot))$, denoted by $Q_f$ and $Q_{\vA}$ respectively, until an $\vareps$-stationary point is found. Additionally, we focus on a practical scenario where the cost of evaluating $(f, \nabla f)$ is significantly higher than $(\vA(\cdot), \vA^\top(\cdot))$ and the proximal mappings defined by $r$ and $\phi$. This scenario arises from many applications in statistics and machine leanring, e.g., linearly constrained LASSO problems~\cite{gaines2018algorithms, james2019penalized}, where evaluating  $(f, \nabla f)$ requires processing a large amount of data while evaluating $(\vA(\cdot), \vA^\top(\cdot))$ does not involve any data and can be relatively easy. 

\subsection{Composite subproblems/approximation} Although \eqref{eq:aff-prob-intro} and \eqref{eq:sp-prob-intro} have different formulations, they can be both solved by the numerical procedures that involve solving \emph{composite optimization} problems with the following structure 
\vspace{-0.2cm}
\begin{equation}\label{eq:cvx-composite}
	F^*=\min_{\vx\in\mathbb{R}^n}\left\{F(\vx)=g(\vx)+H(\vx)\right\} \text{ with }H(\vx)=h(\vx)+r(\vx),
	\vspace{-0.2cm}
\end{equation}
where $g$ and $h$ are smooth convex,  $g$ is $\mu$-strongly convex with $\mu\geqslant 0$, $\nabla g$ is $L_g$-Lipschitz continuous, $\nabla h$ is $L_h$-Lipschitz continuous, and $r$ is as in \eqref{eq:aff-prob-intro} and \eqref{eq:sp-prob-intro}. Hence, reducing the complexity of solving~\eqref{eq:cvx-composite} will lead to more efficient methods for \eqref{eq:aff-prob-intro} and \eqref{eq:sp-prob-intro}. Next, we discuss the relevance of~\eqref{eq:cvx-composite} to \eqref{eq:aff-prob-intro} and \eqref{eq:sp-prob-intro}.

%The present of \eqref{eq:cvx-composite} is an consequence of the optimization approaches we use to solve \eqref{eq:aff-prob-intro} and \eqref{eq:sp-prob-intro}. 
%In this paper, 
We consider solving \eqref{eq:aff-prob-intro} by an \emph{inexact proximal augmented Lagrangian method}  (iPALM), which performs the following update in the $k$th main iteration 
\vspace{-0.2cm}
\begin{equation}\label{eq:iter-iPALM-intro}
	\vx^{(k+1)}\approx \argmin_{\vx\in\mathbb{R}^n} \textstyle f(\vx) + r(\vx)+ \langle\vlam^{(k)}, \vA\vx-\vb\rangle + \frac{\beta_k}{2}\|\vA\vx-\vb\|^2+\frac{\rho_k}{2}\|\vx-\vx^{(k)}\|^2.
	\vspace{-0.2cm}
\end{equation}
Here $\vx^{(k)}$ is the main iterate, $\vlam^{(k)}$ is the Lagrangian multiplier, $\beta_k>0$ is a penalty parameter and $\rho_k>0$ is a proximal parameter. It is easy to see that the problem in \eqref{eq:iter-iPALM-intro} is an instance of \eqref{eq:cvx-composite} with %by setting 
\vspace{-0.1cm}
\begin{equation}\label{eq:aff-instance-intro}
\textstyle	g(\vx)=f(\vx)+\frac{\rho_k}{2}\|\vx-\vx^{(k)}\|^2\quad\text{   and   }\quad h(\vx)= \langle\vlam^{(k)}, \vA\vx-\vb\rangle + \frac{\beta_k}{2}\|\vA\vx-\vb\|^2
\vspace{-0.1cm}
\end{equation}
and the smoothness constants are $L_g=L_f+\rho_k$ and $L_h=\beta_k\|\vA\|^2$. 

For \eqref{eq:sp-prob-intro}, we consider to use the smoothing technique by Nesterov~\cite{nesterov2005smooth}, which approximates \eqref{eq:sp-prob-intro} by the smooth convex optimization problem 
\vspace{-0.2cm}
\begin{equation}\label{eq:sp-prob-intro-smooth}
	\min_{\vx\in\RR^n}\left\{ f(\vx) + r(\vx) +\max_{\vy\in \RR^m} \big[\textstyle \langle \vy, \vA\vx\rangle - \phi(\vy) - \frac{\rho}{2}\|\vy-\vy^{(0)}\|^2 \big]\right\}
	\vspace{-0.2cm}
\end{equation}
and solves \eqref{eq:sp-prob-intro-smooth} using a smooth optimization method. Here, $\rho>0$ is a smoothing parameter, and $\vy^{(0)}\in\dom(\phi)$. Again, we can view \eqref{eq:sp-prob-intro-smooth} as an instance of \eqref{eq:cvx-composite} with 
\vspace{-0.2cm}
\begin{equation}\label{eq:sp-instance-intro}
	g(\vx)=f(\vx)\quad\text{  and  }\quad h(\vx)= \max_{\vy\in \RR^m} \big[\textstyle \langle \vy, \vA\vx\rangle - \phi(\vy) - \frac{\rho}{2}\|\vy-\vy^{(0)}\|^2 \big]
	\vspace{-0.2cm}
\end{equation}	
and the smoothness constants $L_g=L_f$ and $L_h=\|\vA\|^2/\rho$. 

\subsection{Contributions}
Based on these close connections between \eqref{eq:aff-prob-intro}, \eqref{eq:sp-prob-intro} and  \eqref{eq:cvx-composite}, our main contribution is to show that, when the cost of evaluating $(f, \nabla f)$ is significantly higher than $(\vA(\cdot), \vA^\top(\cdot))$, the complexity of solving \eqref{eq:aff-prob-intro} and \eqref{eq:sp-prob-intro} known in the literature can be further reduced if we solve \eqref{eq:iter-iPALM-intro} and \eqref{eq:sp-prob-intro-smooth} using an \emph{inexact accelerated proximal gradient} (iAPG) method, which is a generic method for \eqref{eq:cvx-composite} and requires significantly fewer queries to  $(f, \nabla f)$ than to $(\vA(\cdot), \vA^\top(\cdot))$.

Our iAPG method is a double-loop variant of the  \emph{accelerated proximal gradient} (APG) method~\cite{nesterov2003introductory,nesterov2005smooth,tseng2008accelerated,nesterov2013gradient,beck2009fast}. When applied to \eqref{eq:cvx-composite}, the APG method processes the smooth component, i.e., $g+h$, as a whole and solves \eqref{eq:cvx-composite} by performing the following proximal gradient update in the $k$th iteration
\vspace{-0.2cm}
\begin{equation}\label{eq:prox-r}
\vx^{(k+1)}	=\argmin\limits_{\vx\in\mathbb{R}^n} \textstyle \left\langle\nabla g(\vy^{(k)})+\nabla h(\vy^{(k)}),\vx-\vy^{(k)}\right\rangle+\frac{1}{2\eta_k}\|\vx-\vy^{(k)}\|^2+ r(\vx),
\vspace{-0.2cm}
\end{equation}
where $\vy^{(k)}\in\mathbb{R}^n$ is an auxiliary iterate and $\eta_k>0$ is a step length parameter. By the assumption made on $r$, \eqref{eq:prox-r} can be solved easily, e.g., in a closed form. We denote the numbers of queries to $(g, \nabla g)$ and $(h, \nabla h)$ by $Q_g$ and $Q_h$, respectively. When $\mu>0$, it is known (see e.g., \cite{nesterov2003introductory}) that the APG method finds an $\varepsilon$-optimal solution for \eqref{eq:cvx-composite} with oracle complexity
$
Q_g=Q_h=O\left(\sqrt{\frac{L_g+L_h}{\mu}}\log\left(\frac{1}{\varepsilon}\right)\right).
$
However, according to the instantizations in \eqref{eq:aff-instance-intro} and \eqref{eq:sp-instance-intro},
evaluating $(g, \nabla g)$ has significantly higher complexity than $(h, \nabla h)$ in both instances since the former requires evaluating $(f, \nabla f)$ while the latter only requires evaluating $(\vA(\cdot), \vA^\top(\cdot))$. Given that, a potential strategy (see, e.g.,~\cite{lan2016gradient,lan2016accelerated}) to reduce the overall complexity for solving \eqref{eq:cvx-composite}, and thus, for solving \eqref{eq:aff-prob-intro} and \eqref{eq:sp-prob-intro} is to query $(g, \nabla g)$ and $(h, \nabla h)$ in different frequencies so as to reduce $Q_g$, even if doing so may slightly increase $Q_h$.

To implement this strategy, one technique is to separate $g$ and $h$ by solving the following \emph{proximal mapping} subproblem in the $k$th iteration %of the APG method
\vspace{-0.2cm}
\begin{equation}\label{eq:prox-hr}
\vx^{(k+1)}	=\argmin\limits_{\vx\in\mathbb{R}^n} \textstyle \left\langle\nabla g(\vy^{(k)}),\vx-\vy^{(k)}\right\rangle+\frac{1}{2\eta_k}\|\vx-\vy^{(k)}\|^2+ h(\vx)+r(\vx).
\vspace{-0.2cm}
\end{equation}
Unlike \eqref{eq:prox-r}, \eqref{eq:prox-hr} typically cannot be solved easily. A practical solution is to use another optimization algorithm to solve \eqref{eq:prox-hr} inexactly to certain precision. This requires a double-loop implementation. %and this approach is known as the iAPG method (see e.g., \cite{jiang2012inexact,schmidt2011convergence}). 
Note that \eqref{eq:prox-hr} is itself an instance of \eqref{eq:cvx-composite} and thus can be solved inexactly by the APG method in oracle complexity with logarithmic dependency on the precision thanks to the smoothness of $h$,  the simplicity of $r$, and the strong convexity of $\frac{1}{2\eta_k}\|\vx-\vy^{(k)}\|^2$. By choosing appropriate precision for solving \eqref{eq:prox-hr} in each iteration, we show that, when $\mu>0$, our iAPG method can find an $\vareps$-stationary solution %and an $\vareps$-optimal solution 
of~\eqref{eq:cvx-composite} with oracle complexity\footnote{Here and in the rest of the paper, $\tilde O$ suppresses some logarithmic terms. } 
\vspace{-0.1cm}
\begin{equation}
\label{eq:complexity-comp-opt}
\textstyle Q_g=O\left(\sqrt{\frac{L_g}{\mu}}\log\left(\frac{1}{\varepsilon}\right)\right)\text{ and }Q_h=\tilde O\left(\sqrt{\frac{L_g+L_h}{\mu}}\log\left(\frac{1}{\varepsilon}\right)\right).
\vspace{-0.1cm}
\end{equation}
Since the evaluation of $(g, \nabla g)$ is more costly than  $(h, \nabla h)$, the iAPG method can have a lower overall complexity than the APG method  when $L_h$ is significantly larger than $L_g$.

According to \eqref{eq:aff-instance-intro}, the iAPG method has lower complexity than the APG method for solving \eqref{eq:iter-iPALM-intro} when $\beta_k$ is much larger than $\rho_k$, which is exactly the case in the iPALM. As a consequence, we show that the iPALM, in which \eqref{eq:iter-iPALM-intro} is solved by the iAPG method, finds an $\vareps$-stationary point of \eqref{eq:aff-prob-intro} with oracle complexity\footnote{The factor $\log^2\left(\frac{1}{\vareps}\right)$ in $Q_f$ can be reduced to $\log\frac{1}{\vareps}$ if $\beta_0=\Theta(\frac{1}{\vareps})$ and $\rho_0=\Theta(\vareps)$; see Remark~\ref{rm:affine-cstr}.}
\vspace{-0.3cm}	
\begin{equation}
\label{eq:complexity-aff-prob}
\textstyle Q_f=O\left(\sqrt{\frac{L_f}{\mu}} \log^2\left(\frac{1}{\vareps}\right)\right)\text{ and }Q_{\vA}=\tilde O\left(\sqrt{\frac{L_f}{\mu}}\log\left(\frac{1}{\vareps}\right) + \frac{\|\vA\|}{\sqrt{\mu\vareps}}\right).
\vspace{-0.2cm}
\end{equation}
Without the affine constraint $\vA\vx = \vb$, it is shown by~\cite{nemirovskij1983problem,nesterov2003introductory} that any gradient-based method has to evaluate  $(f, \nabla f)$ at least $\Omega\left(\sqrt{\frac{L_f}{\mu}} \log\left(\frac{1}{\vareps}\right)\right)$ times to find an $\vareps$-optimal point of  \eqref{eq:aff-prob-intro}. With $\vA\vx = \vb$, it is shown by 	\cite{ouyang2021lower} that any gradient-based method needs to evaluate $(\vA(\cdot), \vA^\top(\cdot))$ at least $O(\frac{\|\vA\|}{\sqrt{\mu\vareps}})$ times. In either case, the complexity of the iPALM matches the corresponding lower bound up to logarithmic factors.

Similarly, according to \eqref{eq:sp-instance-intro}, the iAPG method has lower complexity than the APG method
when $\rho$ is small, which is true for the smoothing method. In fact, to obtain an $\vareps$-optimal point of \eqref{eq:sp-prob-intro} by solving \eqref{eq:sp-prob-intro-smooth}, one needs to set $\rho=\Theta(\vareps)$. In this case, we show that, when $\mu>0$, the smooothing method where \eqref{eq:sp-prob-intro-smooth} is solved by the iAPG method finds an $\vareps$-stationary point of \eqref{eq:sp-prob-intro} with the same oracle complexity as \eqref{eq:complexity-aff-prob}. This complexity matches the lower bound \cite{ouyang2021lower} up to logarithmic factors. 

%In the follows, 
\noindent\textbf{Summary of contributions.}~We summarize our contributions that are mentioned above.
\begin{itemize}[leftmargin=*]
	\item[$\bullet$] We give an iAPG method for solving \eqref{eq:cvx-composite}. %Our iAPG method %we studied 
It is a double-loop method where the inner iterations are terminated using a computable stopping criterion based on the stationarity measure of the solution. This is different from existing double-loop approaches, e.g., \cite{lan2016accelerated, schmidt2011convergence}, which require a pre-determined total number of iterations that often depends on some unknown parameters of the problems.
	\item We show the oracle complexity of the proposed iAPG method, which is given in \eqref{eq:complexity-comp-opt}. When evaluating $(g, \nabla g)$ has significantly higher complexity than $(h, \nabla h)$ but $L_g$ is much smaller than $L_h$, the iAPG method is superior over the APG method for solving \eqref{eq:cvx-composite}.  This scenario arises in the subproblems solved during the iPALM for \eqref{eq:aff-prob-intro} and the smooth approximation of \eqref{eq:sp-prob-intro}.
	\item[$\bullet$] Applying the iAPG method to the subproblems solved during the iPALM for \eqref{eq:aff-prob-intro}, we derive the oracle complexity in \eqref{eq:complexity-aff-prob} of the iPALM for finding an $\varepsilon$-stationary solution. This complexity is better than existing ones, e.g., \cite{xu2017accelerated, hamedani2021primal}, when $(f, \nabla f)$ is significantly more expensive than $(\vA(\cdot), \vA^\top(\cdot))$. The complexity result in \cite{lan2021graph} is similar to ours\footnote{The complexity in \cite{lan2021graph} is lower than that in \eqref{eq:complexity-aff-prob} by a logarithmic factor. However, \cite{lan2021graph} targets an $\vareps$-optimal solution.}. However, the inner loop of the method in \cite{lan2021graph} requires a pre-determined number of iterations, and this often yields poor practical performance; see the experimental results in section~\ref{sec:numerical}. Additionally, we show that the iAPG method, in combination with the smoothing techinique by Nesterov, can find  an $\varepsilon$-stationary solution of \eqref{eq:sp-prob-intro} with complexity given in \eqref{eq:complexity-aff-prob} that is also better than existing ones. %to \eqref{eq:aff-prob-intro}.
\end{itemize}
 
%Line search. Finding a stationary point. Parameter free

\subsection{Notation} 
We use lower-case bold letters $\vx,\vy, \vb, \ldots$ for vectors, $\vone$ for an all-one vector/matrix of appropriate size, and upper-case bold letters $\vA, \vW,\ldots$ for matrices. $\vx\odot\vy$ denotes the component-wise product of $\vx$ and $\vy$.
For any number sequence $\{a_i\}_{i\geqslant 0}$, we define $\sum_{i=k_1}^{k_2}a_i=0$ and $\prod_{i=k_1}^{k_2}a_i=1$ if $k_1>k_2$. %Given a proper function $r:\mathbb{R}^d\rightarrow\mathbb{R}\cup\{+\infty\}$, 
The proximal mapping of a proper function $r$ is defined as 
$
\prox_{r}\big(\vz\big):=\argmin_{\vx}\left\{\frac{1}{2}\|\vx-\vz\|^2 +  r(\vx)\right\}.
$

\section{Literature review}

The APG method~\cite{nesterov2003introductory,nesterov2005smooth,tseng2008accelerated,nesterov2013gradient,beck2009fast} is an optimal gradient-based method for the composite optimization when there is no constraint and the non-smooth component in the objective function allows for an easy proximal mapping. However, the APG method cannot be directly applied to \eqref{eq:aff-prob-intro} due to the affine constraints %and cannot be directly applied 
or to \eqref{eq:sp-prob-intro} due to the sophisticated non-smooth term. The iAPG methods by~\cite{jiang2012inexact,schmidt2011convergence} are double-loop implementations of the APG method where the proximal gradient subproblem in each iteration is solved inexactly by another optimization algorithm. The iAPG method we studied in this paper is similar to~\cite{jiang2012inexact,schmidt2011convergence}. However, our method includes a line search scheme. %for the smoothness constant. 
Moreover, the inner loop in our method is terminated based on a computable stationarity measure while \cite{schmidt2011convergence} requires a pre-determined total number of inner iterations that often depends on some unknown parameters of the problems. 

%The inexact APG method by~\cite{jiang2012inexact} is developed for a convex nonlinear semidefinite program with affine constraints. They directly apply the APG method to \eqref{eq:aff-instance-intro} where the proximal mapping subproblem in each main iteration is a quadratic semidefinite program. An inexact semismooth Newton-CG method is then used to solve the subproblem up to the desired precision. 

The augmented Lagrangian method (ALM)~\cite{hestenes1969multiplier,rockafellar1976augmented,powell1978fast} and its modern variants \cite{he2010acceleration,huang2013accelerated,kang2015inexact,lan2016iteration,kang2013accelerated,patrascu2017adaptive,xu2017accelerated,he2021indefinite,sabach2020faster,xu2020first,xu2021iteration,xu2021first,he2021convergence,he2021fast,he2021inertial,bot2021fast} can be applied to \eqref{eq:aff-prob-intro}.  However, the methods in \cite{he2010acceleration,kang2013accelerated} require exactly solving the ALM subproblem, i.e., \eqref{eq:iter-iPALM-intro} with $\rho_k=0$, which is not practical for many applications. Inexact ALMs are studied by~\cite{lan2016iteration,patrascu2017adaptive,xu2021iteration} where ALM subproblems are solved inexactly by the APG method. When $\mu=0$, these methods have oracle complexity $Q_f=Q_{\vA}=O(\frac{1}{\vareps})$ and,  when $\mu>0$, the method by~\cite{xu2021iteration} has oracle complexity $Q_f=Q_{\vA}=O(\frac{1}{\sqrt{\vareps}})$. An accelerated linearized ALM is studied by \cite{xu2017accelerated} where $f$ in \eqref{eq:aff-prob-intro} is linearized in the ALM subproblem. If the augmented term is also linearized so that the subproblem can be solved exactly, the method by~\cite{xu2017accelerated} has the same oracle complexity as~\cite{xu2021iteration} in both the cases when $\mu=0$ and when $\mu>0$. If the augmented term is not linearized, the methods by~\cite{xu2017accelerated,bot2021fast,he2021convergence,he2021inertial} only need $O(\frac{1}{\sqrt{\vareps}})$ iterations even when $\mu=0$, but the ALM subproblem becomes challenging to solve exactly. The linearized ALM method is analyzed in a unified framework together with other variants of the ALM by \cite{sabach2020faster} and is generalized for nonlinear constraints by~\cite{xu2021first}. The same complexity as~\cite{xu2021iteration} is achieved in~\cite{sabach2020faster,xu2021first}. A cutting-plane based ALM is proposed by~\cite{xu2020first} which can find an $\vareps$-KKT point for \eqref{eq:aff-prob-intro} with oracle complexity $Q_f=Q_{\vA}=\tilde O(\frac{m}{\sqrt{\vareps}})$ when $\mu=0$ and $Q_f=Q_{\vA}=\tilde O(m\log(\frac{1}{\vareps}))$ when $\mu>0$, where $m$ is the number of constraints. Hence, its complexity is better than ours only when $m = o(\vareps^{-\frac{1}{2}})$.  ALM-type methods based on dynamical systems~\cite{he2021convergenceinertial,he2021perturbed,boct2021improved} are developed in \cite{he2021convergence,he2021fast} which find an $\vareps$-solution within $o(\frac{1}{\sqrt{\vareps}})$ iterations if each ALM subproblem is solved exactly or inexactly with controllable errors. However, their convergence result is asymptotic, and the total oracle complexity is not given. 

%In a series of works \cite{he2021convergence,he2021fast,he2021inertial,bot2021fast}, the authors focus on the case when $\mu=0$ and develop the ALM methods based on dynamical systems~\cite{he2021convergenceinertial,he2021perturbed,boct2021improved}. They show in \cite{bot2021fast,he2021convergence,he2021inertial} that their methods find an $\vareps$-optimal solution within $O(\frac{1}{\sqrt{\vareps}})$ iterations if $f$ in \eqref{eq:aff-prob-intro} is linearized and the subproblem is solved exactly, which is not practical due to the augmented term. They also show in \cite{he2021convergence,he2021fast} that, if the subproblem in each iteration is solved exactly or inexactly with controllable errors, their methods find an $\vareps$-solution within $o(\frac{1}{\sqrt{\vareps}})$ iterations, but their convergence analysis is only asymptotic and the total oracle complexity is unknown. 
%However, our complexity analysis include the case when $\mu>0$ while theirs do not. 

The (linearized) Bregman methods~\cite{yin2008bregman,yin2010analysis} and their accelerated variants~\cite{huang2013accelerated,kang2013accelerated} are equivalent to gradient-based methods applied to the Lagrangian dual problem of \eqref{eq:aff-prob-intro}. Similar techniques are explored in~\cite{gorbunov2019optimal,dvinskikh2021decentralized}. %also.
 However, these methods require easy evaluation of the proximal mapping of $f$, which limits their applications. For  \eqref{eq:aff-prob-intro} with a strongly convex but not necessarily smooth objective,  a dual $\vareps$-optimal solution can be found by an accelerated Uzawa method~\cite{tao2017accelerated} or an inexact ALM method~\cite{kang2015inexact} within $O(\frac{1}{\sqrt{\vareps}})$ main iterations. However, the method in \cite{tao2017accelerated} requires solving a Lagrangian subproblem exactly and is thus impractical for general $f$. Although the method by~\cite{kang2015inexact} only needs to solve ALM subproblems inexactly, the authors only analyze the total number of main iterations but not the overall oracle complexity. 

Penalty methods~\cite{gorbunov2019optimal,dvinskikh2021decentralized,lan2013iteration,li2017convergence} are also classical approaches for \eqref{eq:aff-prob-intro}, where the affine constraints are moved to the objective function through a penalty term and the unconstrained penalty problem is then solved by another optimization algorithm like the APG method.  The primal method in~\cite{gorbunov2019optimal,dvinskikh2021decentralized} requires $r=0$ and $\vA$ is positive semidefinite while the dual method in~\cite{gorbunov2019optimal,dvinskikh2021decentralized} requires an easy evaluation of the convex conjugate function of $f$, which limits the applications. When $\mu=0$, \cite{lan2013iteration} shows that, if the penalty parameter is large enough, the penalty method finds an $(\vareps,\vareps)$-primal-dual solution of \eqref{eq:aff-prob-intro} (see Def.~1 in \cite{lan2013iteration}) with oracle complexity $Q_f=Q_{\vA}=O(\frac{1}{\vareps})$. The penalty method by \cite{li2017convergence} solves a sequence of unconstrained penalty problems with increasing penalty parameters and only performs one iteration of the APG method on each penalty problem. It has oracle complexity  $Q_f=Q_{\vA}=O(\frac{1}{\vareps})$ when $\mu=0$ and  $Q_f=Q_{\vA}=O(\frac{1}{\sqrt{\vareps}})$ when $\mu>0$. The complexity of the penalty methods in \cite{lan2013iteration,li2017convergence} are higher than ours in both cases.

Using Lagrangian multipliers, a constrained optimization problem can be formulated as a min-max problem to which the primal-dual methods~\cite{tran2014constrained,tran2014primal,tran2018smooth,yurtsever2015universal,wei2018solving}, mostly based on smoothing technique~\cite{nesterov2005smooth}, can be applied. However, the methods by~\cite{tran2014constrained,wei2018solving,tran2018smooth} require a closed-form solution of the proximal mapping of $f$ while the method by~\cite{yurtsever2015universal} requires a closed-form solution of the convex conjugate function of $f$, and thus they have  limited applications. The authors of~\cite{tran2014primal} extend the algorithm and analysis in \cite{tran2014constrained} by allowing the proximal mapping of $f$ to be evaluated inexactly. However, they do not include the oracle complexity for inexactly evaluating the proximal mapping in their complexity analysis.

%Primal-dual methods~\cite{tran2014constrained,tran2014primal,tran2018smooth} based on smoothing technique~\cite{nesterov2005smooth} can be also applied to \eqref{eq:aff-instance-intro}  and obtain the accelerated rates by 

%Sabach and Teboulle\cite{sabach2020faster} also presented a Lagrangian-based algorithm for \eqref{eq:aff-instance-intro}.

%Conditional gradient is studied by \cite{lan2021conditional} for a problem more general than \eqref{eq:aff-instance-intro}.

Smoothing techniques~\cite{nesterov2005smooth,beck2012smoothing,allen2016optimal} are a class of effective approaches for solving the structured non-smooth problem~\eqref{eq:sp-prob-intro}. These methods construct close approximation of \eqref{eq:sp-prob-intro} by one or a sequence of smooth problems, which are then solved by smooth optimization methods such as the APG method.  When $\mu=0$, the methods by \cite{nesterov2005smooth,beck2012smoothing,allen2016optimal} find an $\vareps$-optimal solution with complexity $Q_f=Q_{\vA}=O(\frac{\|A\|}{\vareps}+\sqrt{\frac{L_f}{\vareps}})$. When $\mu>0$, the adapative smoothing method by \cite{allen2016optimal} finds an $\vareps$-optimal solution with complexity $Q_f=Q_{\vA}=O(\sqrt{\frac{L_f}{\mu}}\log(\frac{1}{\vareps})+\frac{\|A\|}{\sqrt{\mu\vareps}})$, which is higher than our complexity given in \eqref{eq:complexity-aff-prob} when the query to $(f,\nabla f)$ is significantly more costly than $(\vA(\cdot), \vA^\top(\cdot))$.

In the literature, \eqref{eq:sp-prob-intro} has also been studied as a bilinear saddle point problem~\cite{bredies2016accelerated,chambolle2011first,chen2014optimal,he2016accelerated,nesterov2005excessive,zhao2019accelerated,zhao2019optimal}. The methods in \cite{bredies2016accelerated,chambolle2011first,nesterov2005excessive} require a closed form  of the proximal mapping of $f+r$ and thus may not be applicable to \eqref{eq:sp-prob-intro}. When $\mu=0$, the methods by \cite{he2016accelerated,chen2014optimal,zhao2019accelerated,zhao2019optimal} find an $\vareps$-saddle-point (see Def.~3.1 in \cite{he2016accelerated}) or an $\vareps$-optimal solution with the same oracle complexity as the smoothing methods mentioned above. When $\mu>0$, the method by~\cite{zhao2019optimal} finds an $\vareps$-optimal solution with the same oracle complexity as the smoothing method~\cite{allen2016optimal}. Problem~\eqref{eq:sp-prob-intro} has also been studied as a variational inequality~\cite{nemirovski2004prox,chen2017accelerated,tseng2008accelerated}. In particular, when $\mu=0$, the mirror-prox methods in~\cite{nemirovski2004prox,tseng2008accelerated} find an $\vareps$-optimal solution of~\eqref{eq:sp-prob-intro} with complexity $Q_f=Q_{\vA}=O(\frac{L_f+\|A\|}{\vareps})$, which is later reduced to  $Q_f=Q_{\vA}=O(\sqrt{\frac{L_f}{\vareps}}+\frac{\|A\|}{\vareps})$ by \cite{chen2017accelerated}. 

For all the methods we discussed above for solving \eqref{eq:aff-prob-intro} and \eqref{eq:sp-prob-intro}, the oracle complexity is essentially the number of iterations the algorithms perform to find the desired solution. Since all of those methods always evaluate both $(f, \nabla f)$ and $(\vA(\cdot), \vA^\top(\cdot))$ in each iteration, $Q_f$ and $Q_{\vA}$ are the same for them. 
When the evaluation cost of $(f, \nabla f)$ is significantly higher than that of $(\vA(\cdot), \vA^\top(\cdot))$, it will be more efficient to query $(f, \nabla f)$ less frequently than  $(\vA(\cdot), \vA^\top(\cdot))$ without compromising the solution quality. This actually can be achieved using the gradient sliding techniques~\cite{lan2016gradient,lan2016accelerated,lan2016conditional,lan2020communication,ouyang2021universal}, which compute the gradient of one (more expensive) component of the objective function once in each outer iteration and process the remaining components in each inner iteration. The iAPG method in this paper utilizes a similar double-loop technique to differentiate the frequencies of evaluating $(f, \nabla f)$ and $(\vA(\cdot), \vA^\top(\cdot))$ and thus reduce $Q_f$.  Although the idea behind the iAPG method is similar to the gradient sliding techniques, such a technique has not be studied for problem \eqref{eq:aff-prob-intro}  under an iPALM framework. Although \eqref{eq:sp-prob-intro} has been studied by~\cite{lan2016gradient,lan2016accelerated}, we consider the case of $\mu>0$, which is not covered in~\cite{lan2016gradient} and for which \cite{lan2016accelerated} needs to apply the sliding method for convex cases in multiple stages. Moreover, in the existing works on the gradient sliding techniques, the inner loop must run for a pre-determined number of iterations which depends on some unknown parameters of the problems. On the contrary, we terminate our inner loop based on a computable stationarity measure, which makes our method more efficient in practice as we demonstrate in Section~\ref{sec:numerical}.

\section{Inexact Accelerated Proximal Gradient Method with Line Search}

In this section, we consider %the composite optimization problem 
\eqref{eq:cvx-composite} where $g$ is $\mu$-strongly convex with $\mu\geqslant 0$ and $L_g$-smooth (i.e. $\nabla g$ is $L_g$-Lipschitz continuous), $h$ are convex and $L_h$-smooth,  %$g$ is , , $\nabla h$ is $L_h$-Lipschitz continuous, 
and $r$ is closed convex and allows easy computation of $\prox_{\eta r}\big(\vz\big)$ for any $\vz\in\mathbb{R}^n$ and $\eta>0$. We assume that $(g,\nabla g)$ is significantly more costly to evaluate than $(h,\nabla h)$ and $L_g$ is significantly smaller than $L_h$. To have a low overall complexity, we propose an iAPG method that calls $(g,\nabla g)$ less frequently than $(h,\nabla h)$. It is
%
%According to Section~\ref{sec:intro}, one effective approach to reduce the complexity of solving \eqref{eq:aff-prob-intro} and \eqref{eq:sp-prob-intro} is to reduce the oracle complexity for solving \eqref{eq:cvx-composite} when $(g,\nabla g)$ is more costly to evaluate than $(h,\nabla h)$. This can be done using the iAPG method for \eqref{eq:cvx-composite} 
given in Algorithm~\ref{alg:iAPG}. This algorithm is a modification of the APG method in \cite[Algorithm~2.2.19]{nesterov2003introductory} by including a line search procedure (in Algorithm~\ref{alg:accellinesearch}) for the step length parameter $\eta_k$ and solving the following proximal mapping subproblem inexactly
%\small
\vspace{-0.1cm}
\begin{equation}\label{eq:def-Phi}
	\textstyle\vx^{(k+1)}\approx\vx^{(k+1)}_*:=\argmin\limits_{\vx\in\mathbb{R}^n}\left\{\Phi(\vx; \vy^{(k)}, \eta_k):=\left\langle\nabla g(\vy^{(k)}),\vx-\vy^{(k)}\right\rangle+\frac{1}{2\eta_k}\|\vx-\vy^{(k)}\|^2+h(\vx) + r(\vx)\right\}.
	\vspace{-0.2cm}
\end{equation}
%More specifically, 
Compared to the APG method that requires $\vx^{(k+1)}=\vx^{(k+1)}_*$, %to be the optimal solution of the minimization problem in Theorem~\ref{thm:main_cvx_rate} while 
our iAPG method only needs $\vx^{(k+1)}$ to be an $\vareps_k$-stationary point, namely, a point satisfying the inequality \eqref{eq:xkprecision}. Our line search procedure %for $\eta_k$ 
follows that in %was first introduced by~\cite{beck2009fast} and extended by~
\cite{lin2015adaptive} for the APG method %in \cite{nesterov2003introductory}  
on solving strongly convex problems.

%It will be shown later 
It can be shown that $\vx^{(k+1)}$ produced by the iAPG method is an $\vareps$-optimal solution of \eqref{eq:cvx-composite} if $k$ is large enough and $\vareps_k$ is small enough. However, %in many applications, one needs to find 
we are more interested in finding an $\vareps$-stationary solution of \eqref{eq:cvx-composite}. %If so, one just 
For this purpose, we just need to perform a proximal gradient step from  $\vx^{(k+1)}$ using a step length $\tilde\eta_k$ that is potentially different from $\eta_k$ and can also be searched by the standard scheme as in~\cite{beck2009fast}. We present the optional procedure to generate an $\vareps$-stationary solution from the iAPG method in Algorithm~\ref{alg:seekstationary}.

\begin{algorithm}[h]
{\small
	\caption{$\mathrm{iAPG}(g, h, r, \vx^{(0)}, \eta_{-1}, \gamma_0, \mu, \underline{L},(\vareps_k)_{k\geqslant 0})$  for solving \eqref{eq:cvx-composite}} 
	\label{alg:iAPG}
\DontPrintSemicolon	
		%{\bfseries Parameters:}   $\underline{L}$ satisfying $\underline{L}\geqslant \mu$.\;
		{\bfseries Inputs:} $\vz^{(0)}=\vx^{(0)}\in\dom(H)$, $\eta_{-1}\leqslant \frac{1}{\mu}$,   
		$\gamma_0\in\left\{\begin{array}{ll} (0,1/\eta_{-1}]&\text{ if }\mu=0 \\ \text{[}\mu,1/\eta_{-1}] &\text{ if }\mu>0\end{array}\right.$, $\mu\leqslant \underline{L}\in (0,L_g]$, and $\vareps_k\geqslant 0, \forall\, k$\;
		\textbf{Optional:} choose $\tilde\eta_0 \leqslant \frac{1}{\underline{L}}$ and $\vareps\geqslant 0$\;
		\For{$k = 0,1,\ldots,$}{
		Generate $(\vx^{(k+1)},\vy^{(k)},\gamma_{k+1},\eta_k,\alpha_k) = \text{InexactAccelLineSearch}(\vx^{(k)},\vz^{(k)},\gamma_k, \eta_{k-1},\mu,\underline{L},\vareps_k)$\;
		$\vz^{(k+1)} =\vx^{(k)}+\frac{1}{\alpha_k}\left(\vx^{(k+1)}-\vx^{(k)}\right)$\;
		\textbf{Optional:} $(\widetilde\vx^{(k+1)},\tilde\eta_{k+1}) = \text{SeekStationary}(\vx^{(k+1)}, \tilde\eta_k)$\;
		\textbf{Optional:} Return  $\widetilde\vx^{(k+1)}$ if $\dist\big(\vzero, \partial F(\widetilde\vx^{(k+1)})\big)\leqslant \vareps$
		}
}		
\end{algorithm}

\begin{algorithm}[h]
{\small\caption{$(\vx^{(k+1)},\vy^{(k)},\gamma_{k+1},\eta_k,\alpha_k) = \text{InexactAccelLineSearch}(\vx^{(k)},\vz^{(k)},\gamma_k, \eta_{k-1},\mu,\underline{L},\vareps_k)$}
	\label{alg:accellinesearch}
	\DontPrintSemicolon	
	{\bfseries Parameters:} $\gamma_{\mathrm{dec}}\in(0,1)$ and $\gamma_{\mathrm{inc}}\in[1,+\infty)$\;
		$\eta_k\leftarrow \min\big\{\frac{1}{\gamma_{\mathrm{dec}} \underline{L}}, \ \gamma_{\mathrm{inc}}\eta_{k-1}\big\}$\;
	\Repeat{$g(\vx^{(k+1)})\leqslant  g(\vy^{(k)}) + \langle\nabla g(\vy^{(k)}),\vx^{(k+1)}-\vy^{(k)}\rangle+\frac{1}{2\eta_k}\|\vx^{(k+1)}-\vy^{(k)}\|^2$}{
		$\eta_k\leftarrow \gamma_{\mathrm{dec}}\eta_k$; find $\alpha_k>0$ and $\gamma_{k+1}$ that satisfy
		$\gamma_{k+1}=\frac{\alpha_k^2}{\eta_k} = \left(1-\alpha_k\right)\gamma_k + \alpha_k \mu$\; 
		Let $\vy^{(k)} =  \frac{1}{\alpha_k \gamma_k+\gamma_{k+1}} \left(\alpha_k \gamma_k \vz^{(k)} + \gamma_{k+1} \vx^{(k)}\right)$\;
		Find $\vx^{(k+1)}$ such that 
		\vspace{-0.1cm}
		\begin{eqnarray}
			\label{eq:xkprecision}
		\textstyle	\dist\left(\vzero, \nabla g(\vy^{(k)}) + \frac{1}{\eta_k}(\vx^{(k+1)}-\vy^{(k)}) + \partial H(\vx^{(k+1)})\right) \leqslant \vareps_k 
		\vspace{-0.1cm}
		\end{eqnarray}
	}
	{\bfseries Return:} $(\vx^{(k+1)},\vy^{(k)},\gamma_{k+1},\eta_k,\alpha_k)$
}
\end{algorithm}

\begin{algorithm}[h]
{\small
\caption{$(\widetilde\vx, \tilde\eta) = \text{SeekStationary}(\vx, \eta)$}
\label{alg:seekstationary}
\DontPrintSemicolon	
$\tilde\eta\leftarrow\eta/\gamma_{\mathrm{dec}}$ with the same $\gamma_{\mathrm{dec}}\in(0,1)$ as that in Algorithm~\ref{alg:accellinesearch}\;
\Repeat{$g(\widetilde\vx)+h(\widetilde\vx)\leqslant  g(\vx)+h(\vx) +\left\langle\nabla g(\vx)+\nabla h(\vx),\widetilde\vx-\vx\right\rangle+\frac{1}{2\tilde\eta}\|\widetilde\vx-\vx\|^2$}{
$\tilde\eta \gets \gamma_{\mathrm{dec}}\tilde\eta$ and let $\widetilde\vx = \prox_{\tilde\eta r}\big(\vx - \tilde\eta(\nabla g(\vx)+\nabla h(\vx))\big)$.
}
{\bfseries Return:} $\tilde\eta $.
}
\end{algorithm}

\subsection{Convergence analysis for iAPG}

In this subsection, we analyze the convergence rate of the proposed iAPG. The analysis also applies to APG by setting $\vareps_k=0$. The technical lemmas below are needed. 

\begin{lemma}\label{lemma:cond-alpha-gamma}
Let $(\eta_k,\tilde\eta_k, \alpha_k,\gamma_k)$ be generated from Algorithm~\ref{alg:iAPG}. It holds that, for any $k\geqslant 0$,
\vspace{-0.2cm}
\begin{equation}\label{eq:cond-alpha-gamma}
\textstyle  \frac{\gamma_{\mathrm{dec}}}{L_g}<\eta_k\leqslant\frac{1}{\underline{L}}, ~\frac{\gamma_{\mathrm{dec}}}{L_g+L_h}<\tilde\eta_k\leqslant\frac{1}{\underline{L}},~\alpha_k\leqslant 1 ~\text{and}~\gamma_k \geqslant  \mu.
\vspace{-0.3cm}
 \end{equation} 	
\end{lemma}

\begin{proof}
From Lines~2 and 4 of Algorithm~\ref{alg:accellinesearch}, we must have  $\eta_k \leqslant \frac{1}{\underline{L}}$ for $k\geqslant 0$ in Algorithm~\ref{alg:iAPG}. In addition, the condition in Line~7 of Algorithm~\ref{alg:accellinesearch} will hold and Algorithm~\ref{alg:accellinesearch} will stop if $\eta_k\leqslant \frac{1}{L_g}$. Given Line~4 of Algorithm~\ref{alg:accellinesearch}, we must have $\eta_k > \frac{\gamma_{\mathrm{dec}}}{L_g}$ in Algorithm~\ref{alg:iAPG}. By the same arguments, we can prove $\frac{\gamma_{\mathrm{dec}}}{L_g+L_h}<\tilde\eta_k\leqslant\frac{1}{\underline{L}}$. %using the similar argument based on Algorithm~\ref{alg:seekstationary}.

Now solving $\alpha_k$ from equation $\frac{\alpha_k^2}{\eta_k} = \left(1-\alpha_k\right)\gamma_k + \alpha_k \mu$ in Line 4 of Algorithm~\ref{alg:accellinesearch}, we have
\vspace{-0.1cm}
	\begin{equation}\label{eq:alpha-value}
	\textstyle\alpha_k=\frac{-(\gamma_k-\mu)+\sqrt{(\gamma_k-\mu)^2+4\gamma_k/\eta_k}}{2/\eta_k}=\frac{2\gamma_k}{(\gamma_k-\mu)+\sqrt{(\gamma_k-\mu)^2+4\gamma_k/\eta_k}}.
\vspace{-0.1cm}	
	\end{equation} 
Since $\mu\leqslant \underline{L}\leqslant 1/\eta_k$, we have $(\gamma_k-\mu)^2+4\gamma_k/\eta_k \geqslant   (\gamma_k+\mu)^2$. Thus it follows from \eqref{eq:alpha-value} that $\alpha_k \leqslant 1, \forall\, k\geqslant  0$. Notice if $\gamma_k\geqslant  \mu$, then $\gamma_{k+1}=\left(1-\alpha_k\right)\gamma_k + \alpha_k\mu \geqslant  \mu$. Since $\gamma_0\geqslant  \mu$, we have $\gamma_k\geqslant  \mu, \forall\,k\geqslant 0$ by induction. %Therefore, we obtain the desired result.
\end{proof}

\begin{lemma}\label{lem:boundsearchsteps}
In any iteration of Algorithm~\ref{alg:iAPG}, Algorithm~\ref{alg:accellinesearch} and Algorithm~\ref{alg:seekstationary} will perform at most $\lceil \log_{\gamma_{\mathrm{dec}}}\frac{\underline{L}}{L_g}\rceil$ and $\lceil \log_{\gamma_{\mathrm{dec}}}\frac{\underline{L}}{L_g+L_h}\rceil$ iterations, respectively. Moreover, if  Algorithm~\ref{alg:iAPG} runs for $t$ iterations, Algorithm~\ref{alg:accellinesearch} and Algorithm~\ref{alg:seekstationary} will perform at most
%\begin{eqnarray*} %\label{eq:complexity-adapapg-all}
$\left(\frac{\log\gamma_{\mathrm{inc}}}{\log\gamma^{-1}_{\mathrm{dec}}}\right)(t-1)+\frac{1}{\log\gamma^{-1}_{\mathrm{dec}}} \min\left\{\log\left(\frac{L_g}{\gamma^2_{\mathrm{dec}}\underline{L}}\right), \ \log\left(\frac{\gamma_{\mathrm{inc}}\eta_{-1}L_g}{\gamma_{\mathrm{dec}}}\right)\right\}$ %=O(t_{k})
%\end{eqnarray*}
and 
%\begin{eqnarray*} %\label{eq:complexity-adapapg-all}
$t+1+\frac{1}{\log\gamma^{-1}_{\mathrm{dec}}}\log\left(\frac{L_g+L_h}{\underline{L}}\right)$
%\end{eqnarray*}
 iterations in total, respectively.
\end{lemma}
\begin{proof}
The first conclusion is from \eqref{eq:cond-alpha-gamma}. The second one can be proved similarly as Lemma 6 in \cite{nesterov2013gradient}. 
\end{proof}

\begin{lemma}\label{lem:prod-alpha}
Let $\tau= \frac{L_g}{\gamma_{\mathrm{dec}} \underline{L} }, \ \kappa = \frac{L_g}{\gamma_\mathrm{dec}\mu}$ and $\alpha_k$ be generated by Algorithm~\ref{alg:iAPG}. It holds that,  for $k\geqslant 0$, 
%\begin{subequations}\label{eq:rate-alpha-all}
%\prod_{j=0}^k(1-\alpha_j) \leqslant 
%\left\{
\vspace{-0.2cm}
\begin{align}
%\frac{1-\alpha_0}{\alpha_0^2}\frac{4\tau^2}{(k+2)^2} , 
\textstyle \frac{1}{(k+1/\alpha_0)\sqrt{\tau}}\leqslant \alpha_k\leqslant \frac{2\sqrt\tau}{k+2}, \text{ if }\mu = 0;\quad %\label{eq:rate-alpha}\\[0.1cm]
%\textstyle 
\sqrt{\frac{1}{\kappa}} \leqslant\alpha_k,  \text{ if }\mu>0, \label{eq:rate-alpha-mu}
\vspace{-0.2cm}
\end{align}
%\right.
%\end{subequations}
%where
%\begin{equation}\label{eq:def-tau-kappa}
%\textstyle \tau= \frac{L_g}{\gamma_{\mathrm{dec}} \underline{L} }, \quad \kappa = \frac{L_g}{\gamma_\mathrm{dec}\mu}.
%\end{equation}
\end{lemma}

\vspace{-0.2cm}
\begin{proof}
%Let the values of $\gamma_k$, $\gamma_{k+1}$, $\eta_k$ and $\alpha_k$ be the ones after Line 10 in iteration $k$. 
When $\mu=0$, Line 4 of Algorithm~\ref{alg:accellinesearch} implies $\frac{\alpha_k^2}{\eta_k}=(1-\alpha_k)\gamma_k=(1-\alpha_k)\frac{\alpha_{k-1}^2}{\eta_{k-1}}$, and thus $\frac{1}{\alpha_k^2}-\frac{1}{\alpha_k} = \frac{1}{\alpha_{k-1}^2}\frac{\eta_{k-1}}{\eta_k}$ for $k\geqslant 1$ in Algorithm~\ref{alg:iAPG}. By the quadratic formula, it follows that 
\vspace{-0.1cm}
\begin{equation}\label{eq:rec-alpha}
\textstyle \frac{1}{\alpha_k} = \frac{1+\sqrt{1+\frac{4}{\alpha_{k-1}^2}\frac{\eta_{k-1}}{\eta_k}}}{2} \geqslant  \frac{1}{2}+\frac{1}{\alpha_{k-1}}\sqrt{\frac{\eta_{k-1}}{\eta_k}}.
\vspace{-0.1cm}
\end{equation}
Recursively applying \eqref{eq:rec-alpha} gives
\vspace{-0.1cm}
\begin{equation}\label{eq:rec-alpha-2}
\textstyle \frac{1}{\alpha_k} \geqslant  \frac{1}{2}\left(1+\sqrt{\frac{\eta_{k-1}}{\eta_k}}\right)+\frac{1}{\alpha_{k-2}}\sqrt{\frac{\eta_{k-2}}{\eta_k}}\geqslant \ldots\geqslant  \frac{1}{2}\sum_{j=1}^k\sqrt{\frac{\eta_{j}}{\eta_k}}+\frac{1}{\alpha_0}\sqrt{\frac{\eta_{0}}{\eta_k}}.
\vspace{-0.1cm}
\end{equation}
According to \eqref{eq:cond-alpha-gamma}, it holds that $\frac{\eta_j}{\eta_k} > \frac{1}{\tau}$ for $k$ and $j\geqslant 0$, and thus \eqref{eq:rec-alpha-2} implies 
%\begin{equation}\label{eq:rec-alpha-3}
$\textstyle \frac{1}{\alpha_k}\geqslant  \big(\frac{k}{2} + \frac{1}{\alpha_0}\big)\sqrt{\frac{1}{\tau}}\geqslant  \frac{k+2}{2} \sqrt{\frac{1}{\tau}},$ 
%\end{equation} 
where the second inequality is because $\alpha_0 \leqslant 1$ by Lemma~\ref{lemma:cond-alpha-gamma}. In addition, from the equality in \eqref{eq:rec-alpha} and the fact that $\sqrt{a+b}\leqslant \sqrt{a}+\sqrt{b}$, we have
%\vspace{-0.1cm}
%\begin{equation}\label{eq:inc-alpha}
$\textstyle \frac{1}{\alpha_k}  \leqslant 1+\frac{1}{\alpha_{k-1}}\sqrt{\frac{\eta_{k-1}}{\eta_k}}$. 
%\vspace{-0.1cm}
%\end{equation}
Recursively applying this inequality %\eqref{eq:inc-alpha} 
gives
\vspace{-0.1cm}
\begin{equation}\label{eq:inc-alpha-2}
	\textstyle \frac{1}{\alpha_k} \leqslant  \left(1+\sqrt{\frac{\eta_{k-1}}{\eta_k}}\right)+\frac{1}{\alpha_{k-2}}\sqrt{\frac{\eta_{k-2}}{\eta_k}}\leqslant \ldots\leqslant  \sum_{j=1}^k\sqrt{\frac{\eta_{j}}{\eta_k}}+\frac{1}{\alpha_0}\sqrt{\frac{\eta_{0}}{\eta_k}}.
	\vspace{-0.1cm}
\end{equation}
By \eqref{eq:cond-alpha-gamma} again, we have $\frac{\eta_j}{\eta_k}<{\tau}$ for all $k$ and $j\geqslant 0$,  and thus \eqref{eq:inc-alpha-2} implies 
%\begin{equation}\label{eq:inc-alpha-3}
$	\textstyle \frac{1}{\alpha_k}\leqslant  \big(k + \frac{1}{\alpha_0}\big)\sqrt{\tau}.$ 
%\end{equation} 
%We have proved the desired result in \eqref{eq:rate-alpha-mu}.
%Now noticing $1-\alpha_k = \frac{\alpha_{k}^2}{\alpha_{k-1}^2}\frac{\eta_{k-1}}{\eta_k}, \forall\, k\geqslant 1$, we have 
%\begin{equation}\label{eq:prod-alpha}
%\prod_{j=0}^k(1-\alpha_j)= (1-\alpha_0)\frac{\alpha_k^2}{\alpha_{0}^2} \frac{\eta_{0}}{\eta_k}.
%\end{equation} By \eqref{eq:rec-alpha-3} and $\frac{\eta_{0}}{\eta_k} < \tau$,
%we obtain the desired result for the case of $\mu=0$.

When $\mu>0$, we have from Lemma~\ref{lemma:cond-alpha-gamma} that $\gamma_{k+1}\geqslant \mu$. Hence, from \eqref{eq:cond-alpha-gamma} and the updating equation of $\gamma_{k+1}$, it follows that $\alpha_k = \sqrt{\eta_k\gamma_{k+1}} \geqslant  \sqrt{\frac{1}{\kappa}}$. Therefore, %$1-\alpha_k \leqslant 1 - \sqrt{\frac{1}{\kappa}}$ for all $k\geqslant 0$, and 
we obtain the desired results. %for the case of $\mu>0$.
\end{proof}

%\begin{lemma}\label{lem:error-dec}
%Algorithm~\ref{alg:iAPG} guarantees that 
%	\begin{align*}
%		H(\vx^{(k+1)}) \leqslant &~ \textstyle H(\widehat\vx^{(k)}) + \left\langle \ve^k - \nabla g(\vy^{(k)}) , \vx^{(k+1)} - \widehat\vx^{(k)}\right\rangle \\
%		&~ \textstyle - \frac{1}{2\eta_k}\left(\|\vx^{(k+1)}-\vy^{(k)}\|^2 + \|\vx^{(k+1)} - \widehat\vx^{(k)}\|^2 - \|\widehat\vx^{(k)} -\vy^{(k)}\|^2\right), 
%	\end{align*}
%for any $k\geqslant 0$, any $\ve^k\in \nabla g(\vy^{(k)})+\frac{1}{\eta_k}(\vx^{(k+1)}-\vy^{(k)})+\partial H(\vx^{(k+1)})$ and any $\widehat\vx^{(k)}\in\dom(H)$.
%\end{lemma}
%
%\begin{proof}
%Since $H$ is convex and $\ve^k - \nabla g(\vy^{(k)})-\frac{1}{\eta_k}(\vx^{(k+1)}-\vy^{(k)})\in\partial H(\vx^{(k+1)})$, we have
%	\begin{align*}
%		H(\vx^{(k+1)}) \leqslant &~ \textstyle H(\widehat\vx^{(k)}) + \left\langle \ve^k - \nabla g(\vy^{(k)})-\frac{1}{\eta_k}(\vx^{(k+1)}-\vy^{(k)}), \vx^{(k+1)} - \widehat\vx^{(k)}\right\rangle. %for any $\widehat\vx^{(k)}\in\dom(H)$. 
%	\end{align*}
%The desired inequality is then obtained using the fact that $\langle \vu, \vv\rangle = \frac{1}{2}\big(\|\vu\|^2 + \|\vv\|^2 - \|\vu-\vv\|^2\big)$.
%\end{proof}

Next, we establish the relationship between two consecutive iterates in Algorithm~\ref{alg:iAPG}.
\begin{proposition}
Let $\{(\vx^{(k)},\vz^{(k)},\alpha_k,\gamma_k)\}$ be generated by Algorithm~\ref{alg:iAPG}. It holds that,  for $k\geqslant 0$, 
\vspace{-0.1cm}
\begin{equation}
\label{eq:main}
\textstyle F(\vx^{(k+1)})-F^*+\frac{\gamma_{k+1}}{2}\|\vx^*-\vz^{(k+1)}\|^2
\leqslant   (1-\alpha_k)\left[F(\vx^{(k)})-F^*+\frac{\gamma_{k}}{2}\|\vx^*-\vz^{(k)}\|^2\right] + \vareps_k \alpha_k\|\vx^*-\vz^{(k+1)}\|.
\vspace{-0.1cm}
\end{equation}
\end{proposition}
\vspace{-0.2cm}

\begin{proof}	
	%\end{eqnarray}
Let $\vx^*$ be an optimal solution of \eqref{eq:cvx-composite} and define 
	%\begin{eqnarray}\label{xk}
$\widehat\vx^{(k)} =\alpha_k\vx^*+(1-\alpha_k)\vx^{(k)} \in\dom(H)$. 
	%\end{eqnarray}
Then
	\begin{eqnarray}
		\label{xk1}
	\widehat\vx^{(k)} - \vy^{(k)} &=&\alpha_k(\vx^*-\vy^{(k)})+(1-\alpha_k)(\vx^{(k)}-\vy^{(k)}).
	\end{eqnarray}
By the updating equation of $\vy^{(k)}$ in Line~5 of Algorithm~\ref{alg:accellinesearch}, it holds that
	%\begin{eqnarray} \label{z-y}
	$\vz^{(k)}-\vy^{(k)} = -\frac{\gamma_{k+1}}{\alpha_k\gamma_k}\left(\vx^{(k)}-\vy^{(k)}\right)$. This together with	
	%Applying $\vz^{(k)}-\vy^{(k)} = -\frac{\gamma_{k+1}}{\alpha_k\gamma_k}\left(\vx^{(k)}-\vy^{(k)}\right)$ to 
	\eqref{xk1} gives
	\vspace{-0.1cm}
	\begin{eqnarray}
	\nonumber
	\widehat\vx^{(k)} - \vy^{(k)} &=& \textstyle \alpha_k (\vx^*-\vy^{(k)})- \frac{\alpha_k(1-\alpha_k)\gamma_{k}}{\gamma_{k+1}}(\vz^{(k)}-\vy^{(k)})%\\\nonumber&=&
	=\alpha_k \left[\vx^*-\vy^{(k)}- \frac{(1-\alpha_k)\gamma_{k}}{\gamma_{k+1}}(\vz^{(k)}-\vy^{(k)})\right]\\\label{xk2}
	&=&\textstyle \alpha_k \left[\vx^*- \frac{(1-\alpha_k)\gamma_{k}}{\gamma_{k+1}}\vz^{(k)}-\frac{\alpha_k\mu}{\gamma_{k+1}}\vy^{(k)}\right],
	\vspace{-0.1cm}
	\end{eqnarray}
	where the last equality follows from the updating equation of $\gamma_{k+1}$.
	
	According to \eqref{eq:xkprecision}, there exists $\ve^{(k)}\in\mathbb{R}^n$ such that $\|\ve^{(k)}\|\leqslant \vareps_k$ and 
	$\ve^{(k)}-\nabla g(\vy^{(k)})-\frac{1}{\eta_k}(\vx^{(k+1)}-\vy^{(k)})\in\partial H(\vx^{(k+1)})$. By the convexity of $H$, we have
	\vspace{-0.1cm}
	\begin{align*}
		H(\vx^{(k+1)}) \leqslant &~ \textstyle H(\widehat\vx^{(k)}) + \big\langle \ve^{(k)} - \nabla g(\vy^{(k)})-\frac{1}{\eta_k}(\vx^{(k+1)}-\vy^{(k)}), \vx^{(k+1)} - \widehat\vx^{(k)}\big\rangle, 
		\vspace{-0.1cm}
	\end{align*} 
which, by the fact that $\langle \vu, \vv\rangle = \frac{1}{2}\big(\|\vu\|^2 + \|\vv\|^2 - \|\vu-\vv\|^2\big)$,  implies 
\vspace{-0.1cm}
\begin{align*}
	H(\vx^{(k+1)}) \leqslant &~ \textstyle H(\widehat\vx^{(k)}) + \left\langle \ve^{(k)} - \nabla g(\vy^{(k)}) , \vx^{(k+1)} - \widehat\vx^{(k)}\right\rangle \\
	&~ \textstyle - \frac{1}{2\eta_k}\left(\|\vx^{(k+1)}-\vy^{(k)}\|^2 + \|\vx^{(k+1)} - \widehat\vx^{(k)}\|^2 - \|\widehat\vx^{(k)} -\vy^{(k)}\|^2\right), \\
	\leqslant &~ \textstyle H(\widehat\vx^{(k)}) + \left\langle   \nabla g(\vy^{(k)}) ,  \widehat\vx^{(k)}-\vx^{(k+1)} \right\rangle +\vareps_k \|\vx^{(k+1)} - \widehat\vx^{(k)}\|\\
	&~ \textstyle - \frac{1}{2\eta_k}\left(\|\vx^{(k+1)}-\vy^{(k)}\|^2 + \|\vx^{(k+1)} - \widehat\vx^{(k)}\|^2 - \|\widehat\vx^{(k)} -\vy^{(k)}\|^2\right).
\vspace{-0.1cm}
\end{align*}	
	From the inequality above and the stopping condition of Algorithm~\ref{alg:accellinesearch}, we have
	\begin{eqnarray*}
		F(\vx^{(k+1)})
		&\leqslant & \textstyle g(\vy^{(k)})+\big\langle\nabla g(\vy^{(k)}), \vx^{(k+1)}-\vy^{(k)}\big\rangle
		+\frac{1}{2\eta_k}\big\|\vx^{(k+1)}-\vy^{(k)}\big\|^2+H(\vx^{(k+1)}) \\
		&\leqslant & \textstyle g(\vy^{(k)})+ \big\langle\nabla g(\vy^{(k)}),\widehat\vx^{(k)}-\vy^{(k)}\big\rangle
		 +\frac{1}{2\eta_k}\big\|\widehat\vx^{(k)}-\vy^{(k)}\big\|^2 +H(\widehat\vx^{(k)} )
		 -\frac{1}{2\eta_k}\big\|\widehat\vx^{(k)}-\vx^{(k+1)}\big\|^2 \\ 
		 & &+ \vareps_k \|\vx^{(k+1)} - \widehat\vx^{(k)}\|.
	\end{eqnarray*}
Applying \eqref{xk1} to the above inequality, we have
\vspace{-0.1cm}
\begin{eqnarray*}
	F(\vx^{(k+1)})
	&\leqslant & g(\vy^{(k)})+ \big\langle\nabla g(\vy^{(k)}),\alpha_k(\vx^*-\vy^{(k)})+(1-\alpha_k)(\vx^{(k)}-\vy^{(k)})\big\rangle \\
	&& \textstyle +\frac{1}{2\eta_k}\big\|\widehat\vx^{(k)}-\vy^{(k)}\big\|^2 +H(\alpha_k\vx^* +(1-\alpha_k)\vx^{(k)})-\frac{1}{2\eta_k}\big\|\widehat\vx^{(k)}-\vx^{(k+1)}\big\|^2 + \vareps_k \|\vx^{(k+1)} - \widehat\vx^{(k)}\|.
	\vspace{-0.1cm}
	\end{eqnarray*}
By the fact that $\alpha_k\in (0,1]$ from Lemma~\ref{lemma:cond-alpha-gamma}, \eqref{xk2} and the convexity of $H$, we have
\vspace{-0.1cm}
	\begin{eqnarray}\label{eq:main0-0}
	F(\vx^{(k+1)})
	&\leqslant & (1-\alpha_k)\big[g(\vy^{(k)})+ \langle\nabla g(\vy^{(k)}),\vx^{(k)}-\vy^{(k)}\rangle+H(\vx^{(k)})\big] \cr
	&& + \alpha_k\big[g(\vy^{(k)})+ \left\langle\nabla g(\vy^{(k)}),\vx^*-\vy^{(k)}\right\rangle+H(\vx^*)\big] \\
	&& \textstyle  +\frac{\alpha_k^2}{2\eta_k}\big\| \vx^*- \frac{(1-\alpha_k)\gamma_{k}}{\gamma_{k+1}}\vz^{(k)}-\frac{\alpha_k\mu}{\gamma_{k+1}}\vy^{(k)}\big\|^2 -\frac{1}{2\eta_k}\big\|\widehat\vx^{(k)}-\vx^{(k+1)}\big\|^2 + \vareps_k \|\vx^{(k+1)} - \widehat\vx^{(k)}\|.\nonumber
	\vspace{-0.1cm}
\end{eqnarray}
Since $\gamma_{k+1}=\alpha^2_k/\eta_k=\left(1-\alpha_k\right)\gamma_k + \alpha_k \mu$,  we have from the convexity of $\|\cdot\|^2$ that
\vspace{-0.1cm}
	\begin{eqnarray*}
\textstyle 	\frac{\alpha_k^2}{2\eta_k}\big\| \vx^*- \frac{(1-\alpha_k)\gamma_{k}}{\gamma_{k+1}}\vz^{(k)}-\frac{\alpha_k\mu}{\gamma_{k+1}}\vy^{(k)}\big\|^2
	&=& \textstyle \frac{\gamma_{k+1}}{2}\big\|\vx^*- \frac{(1-\alpha_k)\gamma_{k}}{\gamma_{k+1}}\vz^{(k)}-\frac{\alpha_k\mu}{\gamma_{k+1}}\vy^{(k)}\big\|^2 \\
	&\leqslant & \textstyle \frac{(1-\alpha_k)\gamma_k}{2}\|\vx^*-\vz^{(k)}\|^2+\frac{\alpha_k\mu}{2}\|\vx^*-\vy^{(k)}\|^2,
	\vspace{-0.2cm}
	\end{eqnarray*}
which, together with \eqref{eq:main0-0} and the $\mu$-strong convexity of $g$, implies
\vspace{-0.2cm}
	\begin{eqnarray}
	\label{eq:main0}
	F(\vx^{(k+1)})
	&\leqslant &\textstyle (1-\alpha_k)\big[g(\vy^{(k)})+ \langle\nabla g(\vy^{(k)}),\vx^{(k)}-\vy^{(k)}\rangle+H(\vx^{(k)})+\frac{\gamma_k}{2}\|\vx^*-\vz^{(k)}\|^2\big] \nonumber\\
	&& \textstyle + \alpha_k\big[g(\vy^{(k)})+ \langle\nabla g(\vy^{(k)}),\vx^*-\vy^{(k)}\rangle+H(\vx^*)+\frac{\mu}{2}\|\vx^*-\vy^{(k)}\|^2\big]\nonumber\\
	&& \textstyle  -\frac{1}{2\eta_k}\big\|\widehat\vx^{(k)}-\vx^{(k+1)}\big\|^2 + \vareps_k \|\vx^{(k+1)} - \widehat\vx^{(k)}\| \nonumber\\
	&\leqslant & \textstyle (1-\alpha_k)\big[ F(\vx^{(k)})+ \frac{\gamma_k}{2}\|\vx^*-\vz^{(k)}\|^2\big]+ \alpha_kF(\vx^*)-\frac{1}{2\eta_k}\big\|\widehat\vx^{(k)}-\vx^{(k+1)}\big\|^2%\nonumber\\&& 
	+ \vareps_k \|\vx^{(k+1)} - \widehat\vx^{(k)}\|.
	\vspace{-0.2cm}
	\end{eqnarray}
By the definitions of $\vz^{(k+1)}$ and $\widehat\vx^{(k)}$, it holds that
\vspace{-0.1cm}
\begin{equation}\label{eq:diff-term-hatk}
\|\widehat\vx^{(k)}-\vx^{(k+1)}\|^2=\|\alpha_k\vx^*+(1-\alpha_k)\vx^{(k)}-\vx^{(k+1)}\|^2=\alpha_k^2\|\vx^*-\vz^{(k+1)}\|^2.
\vspace{-0.1cm}
\end{equation}
Apply the equality in \eqref{eq:diff-term-hatk} to \eqref{eq:main0} and use $\gamma_{k+1}=\alpha^2_k/\eta_k$ to obtain the desired inequality.
\end{proof}

We apply \eqref{eq:main} recursively to derive the convergence rate of Algorithm~\ref{alg:iAPG} for the case of $\mu>0$ as follows. The case of $\mu=0$ will be analyzed in section~\ref{sec:rate-compl-cvx}. 
\begin{theorem}
	\label{thm:main_corollary}
Suppose $\mu>0$. For any $c\in [0,1)$,
Algorithm~\ref{alg:iAPG} guarantees that, for $k \geqslant 0$, 
\vspace{-0.2cm}
\begin{eqnarray}\label{eq:rate-psi-k}
	\psi_{k+1} &\leqslant & \textstyle \prod_{j=0}^k (1-c\alpha_j)\left(\psi_0 + \frac{\sqrt\kappa}{2(1-c)^2\underline{L}}\sum_{t=0}^{k}\frac{\vareps_t^2}{\prod_{j=0}^{t-1} (1-c\alpha_j)}\right),
	\vspace{-0.2cm}
\end{eqnarray}
where $\psi_k := \textstyle F(\vx^{(k)})-F^*+(1-(1-c)\alpha_{k})\frac{\gamma_{k}}{2}\|\vx^*-\vz^{(k)}\|^2,\forall\, k\geqslant 0,$
and $\kappa$ is defined in Lemma~\ref{lem:prod-alpha}. In addition, when $\mu>0$ and $\vareps_k=0$ for all $k$, Algorithm~\ref{alg:iAPG} guarantees that 
\vspace{-0.1cm}
\begin{equation}\label{eq:rate-scvx-no-noise}
\textstyle 	F(\vx^{(k+1)})-F^*+\frac{\gamma_{k+1}}{2}\|\vx^*-\vz^{(k+1)}\|^2 \leqslant \textstyle \big(1-\sqrt{\frac{1}{\kappa}}\big)^{k+1} \left(F(\vx^{(0)})-F^*+\frac{\gamma_{0}}{2}\|\vx^*-\vz^{(0)}\|^2\right), \forall\, k\geqslant 0.
\vspace{-0.1cm}
\end{equation}
\end{theorem}
\vspace{-0.2cm}
\begin{proof}
%Below we analyze the convergence rate for the inexact case with $\vareps_k>0$. 
By the Young's inequality, we have that for any $c\in[0,1)$, 
\vspace{-0.2cm}
$$\textstyle \vareps_k\alpha_k\|\vx^* - \vz^{(k+1)}\| \leqslant \frac{(1-c)\alpha_{k+1}\alpha_k^2}{2\eta_k}\|\vx^* - \vz^{(k+1)}\|^2+\frac{\eta_k}{2(1-c)\alpha_{k+1}}\vareps_k^2.
\vspace{-0.2cm}$$ 
Recall $\gamma_{k+1}=\frac{\alpha_k^2}{\eta_k}$. Hence, we have from \eqref{eq:main} that
\vspace{-0.2cm}
\begin{eqnarray*}
	&&\textstyle F(\vx^{(k+1)})-F^*+\frac{\gamma_{k+1}}{2}\|\vx^*-\vz^{(k+1)}\|^2\cr
	&\leqslant  & \textstyle (1-\alpha_k)\left[F(\vx^{(k)})-F^*+\frac{\gamma_{k}}{2}\|\vx^*-\vz^{(k)}\|^2\right] + \frac{(1-c)\alpha_{k+1}\gamma_{k+1}}{2}\|\vx^* - \vz^{(k+1)}\|^2+\frac{\eta_k}{2(1-c)\alpha_{k+1}}\vareps_k^2,
\vspace{-0.2cm}	
\end{eqnarray*}
which, after rearranging terms, is reduced to
\vspace{-0.2cm}
\begin{eqnarray}
	\label{eq:main-scvx}
	&&\textstyle F(\vx^{(k+1)})-F^*+\big(1-(1-c)\alpha_{k+1}\big)\frac{\gamma_{k+1}}{2}\|\vx^*-\vz^{(k+1)}\|^2\cr
	&\leqslant  & \textstyle (1-\alpha_k)\left[F(\vx^{(k)})-F^*+\frac{\gamma_{k}}{2}\|\vx^*-\vz^{(k)}\|^2\right] +\frac{\eta_k}{2(1-c)\alpha_{k+1}}\vareps_k^2.
	\vspace{-0.2cm}
\end{eqnarray}
 Then it follows from \eqref{eq:main-scvx}, the definition of $\psi_k$, and $F(\vx^{(k)}) - F^*\geqslant 0$ that
 \vspace{-0.2cm}
\begin{equation}\label{eq:rate-psi}
\textstyle 	\psi_{k+1} \leqslant \frac{1-\alpha_k}{1-(1-c)\alpha_k} \psi_k +\frac{\eta_k}{2(1-c)\alpha_{k+1}}\vareps_k^2 \leqslant (1-c\alpha_k) \psi_k +\frac{\sqrt\kappa}{2(1-c)\underline{L}}\vareps_k^2,
\vspace{-0.2cm}
\end{equation}
where we have used \eqref{eq:cond-alpha-gamma} and \eqref{eq:rate-alpha-mu}. %the above inequality indicates
%\begin{equation}
%\psi_{k+1} \leqslant  (1-\frac{1}{2\sqrt\kappa}) \psi_k +\frac{\sqrt\kappa}{\underline{L}}\vareps_k^2.
%\end{equation}
Recursively applying the inequality in \eqref{eq:rate-psi} gives
\vspace{-0.1cm}
\begin{eqnarray*}
	\psi_{k+1} &\leqslant &\textstyle \prod_{j=0}^k (1-c\alpha_j)\psi_0 + \frac{\sqrt\kappa}{2(1-c)\underline{L}}\sum_{t=0}^{k}\Big(\prod_{j=t+1}^k (1-c\alpha_j) \Big)\vareps_t^2 \cr
	&=& \textstyle  \prod_{j=0}^k (1-c\alpha_j)\left(\psi_0 + \frac{\sqrt\kappa}{2(1-c)\underline{L}}\sum_{t=0}^{k}\frac{\vareps_t^2}{\prod_{j=0}^t (1-c\alpha_j)}\right),
	\vspace{-0.1cm}
\end{eqnarray*}
which implies \eqref{eq:rate-psi-k} because $\alpha_j \leqslant  1$ for all $j\geqslant 0$.	

When $\vareps_k=0$ and $\mu>0$, 
\eqref{eq:rate-scvx-no-noise} can be easily derived by recursively applying \eqref{eq:main}  and using \eqref{eq:rate-alpha-mu}.
\end{proof}

The result in \eqref{eq:rate-psi-k} is similar to Propositions~2 and 4 in~\cite{schmidt2011convergence} but takes a different form. It will be later used to derive the oracle complexity of the iAPG method. The result in \eqref{eq:rate-scvx-no-noise} is exactly the convergence property of the APG method~\cite{nesterov2003introductory} for a strongly convex case. Although \eqref{eq:rate-scvx-no-noise} is not new, we still present it here because we need it later to analyze the complexity to obtain $\vx^{(k+1)}$ in Line~6 of Algorithm~\ref{alg:accellinesearch}.

\subsection{Complexity of APG for finding an $\vareps$-stationary point of \eqref{eq:cvx-composite}}
The oracle complexity of Algorithm~\ref{alg:iAPG} must include the complexity for finding $\vx^{(k+1)}$ satisfying \eqref{eq:xkprecision} in each iteration of Algorithm~\ref{alg:accellinesearch}. Such an $\vx^{(k+1)}$ can be found by approximately solving \eqref{eq:def-Phi}, which is an instance of \eqref{eq:cvx-composite} with the $g$, $h$ and $r$ components being $\Phi(\,\cdot\,; \vy^{(k)}, \eta_k)-r(\cdot)$, $0$ and  $r(\cdot)$, respectively. The assumption on $r$ allows us to apply the exact APG method, i.e., Algorithm~\ref{alg:iAPG}  with $\vareps_k=0, \forall\, k\geqslant 0$ to \eqref{eq:def-Phi} in order to find $\vx^{(k+1)}$. The convergence of the objective gap by the exact APG method is characterized by \eqref{eq:rate-scvx-no-noise}. However, \eqref{eq:xkprecision} requires $\vx^{(k+1)}$ to be an $\vareps_k$-stationary solution of \eqref{eq:cvx-composite} instead of an $\vareps_k$-optimal solution. Hence, we first establish the complexity for the exact APG method to find an $\vareps$-stationary solution of \eqref{eq:cvx-composite}. The analysis is standard in literature and included for the sake of completeness.

\begin{lemma}\label{lem:bd-dist-grad}
	Let $C_L = \frac{L_g+L_h}{\sqrt{\underline{L}}} + \sqrt{\frac{L_g+L_h}{\gamma_{\mathrm{dec}}}} $, where $\underline{L}$ and $\gamma_{\mathrm{dec}}$ are those in Algorithms~\ref{alg:iAPG} and~\ref{alg:accellinesearch}.  Suppose the optional steps are invoked in Algorithm~\ref{alg:iAPG}. It holds that, for any $k\geqslant 0$,
	\vspace{-0.1cm}
	\begin{equation}\label{ineq:stationary0}
		\textstyle \dist\big(\vzero, \partial F(\widetilde\vx^{(k+1)})\big) \leqslant  C_L \sqrt{2\big(F(\vx^{(k+1)}) - F^*\big)}.
	\end{equation}
	\vspace{-0.1cm}
\end{lemma}
\vspace{-0.2cm}
\begin{proof}
	When the stopping condition of Algorithm~\ref{alg:seekstationary} holds, we have (cf. \cite[Lemma 2.1]{xu2013block}) $F(\vx) - F(\widetilde\vx) \geqslant  \frac{1}{2\tilde\eta}\| \vx-\widetilde\vx\|^2$, and thus
	\vspace{-0.2cm}
	\begin{equation}\label{eq:bd-diff-x-xtilde-k}
		\|\widetilde\vx^{(k+1)}- \vx^{(k+1)}\|\leqslant \sqrt{2\tilde\eta_{k+1}\big(F(\vx^{(k+1)}) - F(\widetilde\vx^{(k+1)})\big)}.
		\vspace{-0.2cm}
	\end{equation} 
	Also, from the update of $\widetilde\vx$, we have $\vzero\in \nabla (g + h)(\vx) + \frac{1}{\eta}(\widetilde\vx- \vx)+\partial r(\widetilde\vx)$, and thus $\dist\big(\vzero, \partial F(\widetilde\vx)\big) \leqslant \|\nabla (g+h)(\widetilde\vx)- \nabla (g+h)(\vx) + \frac{1}{\eta}(\widetilde\vx- \vx)\|\leqslant (L_g + L_h + \frac{1}{\eta})\|\widetilde\vx- \vx\|$. Therefore, for $\widetilde\vx^{(k+1)}$ obtained in Algorithm~\ref{alg:iAPG}, it holds
	\vspace{-0.2cm}
	\begin{eqnarray}\label{ineq:stationary}
		\dist\big(\vzero, \partial F(\widetilde\vx^{(k+1)})\big)&\leqslant & \textstyle  (L_g + L_h+\frac{1}{\tilde\eta_{k+1}})\|\widetilde\vx^{(k+1)}- \vx^{(k+1)}\| \cr
		&\overset{\eqref{eq:bd-diff-x-xtilde-k}}\leqslant &  \textstyle (L_g + L_h+\frac{1}{\tilde\eta_{k+1}})\sqrt{2\tilde\eta_{k+1}\big(F(\vx^{(k+1)}) - F(\widetilde\vx^{(k+1)})\big)}.
		\vspace{-0.2cm}
	\end{eqnarray}
	Now use $\frac{\gamma_{\mathrm{dec}}}{L_g+L_h} < \tilde\eta_{k+1}\leqslant \frac{1}{\underline{L}}$ in \eqref{eq:cond-alpha-gamma} and $F(\widetilde\vx^{(k+1)})\geqslant  F^*$. We obtain the desired result.
\end{proof}

By \eqref{eq:rate-scvx-no-noise} and \eqref{ineq:stationary0}, we immediately have the following result.
\begin{theorem}
	\label{thm:convstationary}
	Suppose the optional steps are enabled in Algorithm~\ref{alg:iAPG}. Let  $C_L$ be defined in Lemma~\ref{lem:bd-dist-grad}. When $\mu>0$ and $\vareps_k=0$ for all $k\geqslant 0$, Algorithm~\ref{alg:iAPG} guarantees that 
	\vspace{-0.2cm}
	\begin{equation}\label{eq:comp-special}
		\dist\big(\vzero, \partial F(\widetilde\vx^{(k)})\big) \leqslant \textstyle C_L \sqrt{2\left(F(\vx^{(0)})-F^*+\frac{\gamma_{0}}{2}\|\vx^*-\vz^{(0)}\|^2\right)} \big(1-\sqrt{\frac{1}{\kappa}}\big)^{\frac{k}{2}}, \forall\, k\geqslant 0. 
		\vspace{-0.2cm}
	\end{equation}
	%where
\end{theorem}
\vspace{-0.2cm}

\section{Oracle complexity of iAPG}%for finding an $\vareps$-stationary point of \eqref{eq:cvx-composite}
In this section, we show the oracle complexity of Algorithm~\ref{alg:iAPG} for finding an $\varepsilon$-optimal or $\vareps$-stationary solution of \eqref{eq:cvx-composite} in the convex and strongly convex cases separately. 

\subsection{Complexity for ensuring \eqref{eq:xkprecision}}
Theorem~\ref{thm:convstationary} implies the oracle complexity for finding $\vx^{(k+1)}$ satisfying \eqref{eq:xkprecision} by applying the exact APG method to \eqref{eq:def-Phi}. More specifically, we can compute $\vx^{(k+1)}$ by calling  the iAPG method as
\vspace{-0.1cm}
\begin{eqnarray}\label{eq:calliAPG}
	\vx^{(k+1)}=\mathrm{iAPG}\left(\Phi(\cdot\,; \vy^{(k)}, \eta_k)-r(\cdot), 0, r(\cdot), \vx^{(k)}, \eta_k, \eta_k^{-1}, \eta_k^{-1}, \eta_k^{-1}, (0)_{k\geqslant 0}\right).
	\vspace{-0.1cm}
\end{eqnarray}
Note that here we use $\vx^{(k)}$ as the initial solution for computing $\vx^{(k+1)}$ and the inputs in \eqref{eq:calliAPG} are chosen based on the fact that $\Phi(\cdot\,; \vy^{(k)}, \eta_k)-r(\cdot)$ is $1/\eta_k$-strongly convex. The complexity of finding $\vx^{(k+1)}$ is then given as follows.
%Note that a constant $\underline{L}_{\text{in}}$ different from $\underline{L}$ is used as the lower bound estiamtion for $L_h$.
%we have set the input in \eqref{eq:calliAPG} so that $1/\eta_0=\gamma_0=\mu=\underline{L}_{\text{in}}=1/\eta_k$. 

\begin{proposition}[Complexity for ensuring \eqref{eq:xkprecision}]\label{lem:inner-comp} 
	%Consider the $k$-th iteration of Algorithm~\ref{alg:iAPG}. 
	Let $\vx^{(k+1)}_*$ and $\Phi$ be defined in \eqref{eq:def-Phi}. Suppose Algorithm~\ref{alg:iAPG} is applied to \eqref{eq:def-Phi} with the inputs given in \eqref{eq:calliAPG} and the optional step enabled. Solution $\vx^{(k+1)}$ satisfying \eqref{eq:xkprecision} can be found after at most $T_k$ queries to $(h, \nabla h)$, where
	\vspace{-0.2cm}
	\begin{equation}\label{eq:num-oracle-k}
		T_k = \textstyle O\left(\sqrt{1+\frac{L_h}{\underline{L}}}\log\frac{\sqrt{L_g+L_h+L_h^2/\underline{L}}\sqrt{\Phi(\vx^{(k)}; \vy^{(k)},\eta_k) - \Phi(\vx^{(k+1)}_*; \vy^{(k)},\eta_k) }}{\vareps_k}\right) %+\log\left(1+\frac{L_h}{\underline{L}}\right)
		\vspace{-0.2cm}
	\end{equation}
	with hidden constants in the big-$O$ depending only on $\gamma_{\mathrm{dec}}$ and $\gamma_{\mathrm{inc}}$.
\end{proposition}

\begin{proof}
	Notice that the smoothness constant and the strong convexity parameter of $\Phi(\,\cdot\,; \vy^{(k)},\eta_k) - r(\cdot)$ are $\frac{1}{\eta_k}+L_h$ and $\frac{1}{\eta_k}$, respectively. %Hence, the condition number of $\Phi(\,\cdot\,; \vy^{(k)},\eta_k) - r(\cdot)$ is $1+\eta_k L_h$
From the strong convexity of $\Phi$, it holds %implies 
\vspace{-0.2cm}
	\begin{equation}
	\label{eq:distancebygap}
\textstyle	\frac{1}{2\eta_k}\|\vx^{(k)}-\vx^{(k+1)}_*\|^2\leqslant \Phi(\vx^{(k)}; \vy^{(k)},\eta_k) - \Phi(\vx^{(k+1)}_*; \vy^{(k)},\eta_k).
\vspace{-0.2cm}
\end{equation}
By instantizing Theorem~\ref{thm:convstationary} on \eqref{eq:def-Phi}, Algorithm~\ref{alg:iAPG} with the inputs given in \eqref{eq:calliAPG} must find $\vx^{(k+1)}$ satisfying \eqref{eq:xkprecision} in no more than $t_k$ iterations with
\vspace{-0.2cm}
	\begin{eqnarray}
		\nonumber
		t_k&\leqslant&\textstyle \left\lceil \log\frac{\big(\frac{1/\eta_k+L_h}{\sqrt{1/\eta_k}}+\sqrt{\frac{1/\eta_k+L_h}{\gamma_{\mathrm{dec}}}}\big)\sqrt{2\Phi(\vx^{(k)}; \vy^{(k)},\eta_k) - 2\Phi(\vx^{(k+1)}_*; \vy^{(k)},\eta_k)+\frac{1}{\eta_k}\|\vx^{(k)}-\vx^{(k+1)}_*\|^2 }}{\vareps_k} \frac{2}{ \log\big(1 - \sqrt{\frac{\gamma_{\mathrm{dec}}}{1+\eta_k L_h}} \big)^{-1} } \right\rceil\\\label{eq:upperboundtk}
		&=&\textstyle O\left(\sqrt{1+\frac{L_h}{\underline{L}} }\log\frac{2\big(\sqrt{L_g/\gamma_{\mathrm{dec}}}+L_h/\sqrt{\underline{L}}+\sqrt{\frac{L_g/\gamma_{\mathrm{dec}}+L_h}{\gamma_{\mathrm{dec}}}} \big)\sqrt{\Phi(\vx^{(k)}; \vy^{(k)},\eta_k) - \Phi(\vx^{(k+1)}_*; \vy^{(k)},\eta_k) }}{\vareps_k} \right),
		\vspace{-0.2cm}
	\end{eqnarray}
	where the second equation is because of \eqref{eq:cond-alpha-gamma} and \eqref{eq:distancebygap} and uses the fact $\ln(1-a)^{-1} \geqslant a$ for any $a\in (0,1)$. 
	
	%By Lemma~\ref{lemma:cond-alpha-gamma} and Lines 4 and 6 of  Algorithm~\ref{alg:iAPG}, the repeat-loop in each iteration of Algorithm~\ref{alg:iAPG} will stop after at most $\lceil \log_{\gamma_{\mathrm{dec}}}\frac{\underline{L}}{L_g}\rceil$ iterations. This means $(h, \nabla h)$ will be queried $\lceil \log_{\gamma_{\mathrm{dec}}}\frac{\underline{L}}{L_g}\rceil t_k$ times at most. Together with \eqref{eq:upperboundtk}, this implies the conclusion.
	
	Due to the repeat-loop in Algorithms~\ref{alg:accellinesearch} and~\ref{alg:seekstationary}, $(h, \nabla h)$ may be queried multiple times in each iteration during the call of iAPG given in \eqref{eq:calliAPG}. Fortunately, by instantizing Lemma~\ref{lem:boundsearchsteps} on \eqref{eq:def-Phi} with the input given in \eqref{eq:calliAPG}, the total number of queries of $(h, \nabla h)$ must be no more than 
	\vspace{-0.1cm}
	\begin{eqnarray*} %\label{eq:complexity-adapapg-all}
		T_k&=& \textstyle 2\left(\frac{\log\gamma_{\mathrm{inc}}}{\log\gamma^{-1}_{\mathrm{dec}}}\right)(t_k-1)+\frac{2}{\log\gamma^{-1}_{\mathrm{dec}}} \min\left\{\log\left(\frac{1+\eta_kL_h}{\gamma^2_{\mathrm{dec}}}\right), \ \log\left(\frac{\gamma_{\mathrm{inc}}(1+\eta_kL_h)}{\gamma_{\mathrm{dec}}}\right)\right\}%\\
		%&& \textstyle 
		+2t_k+2+\frac{2}{\log\gamma^{-1}_{\mathrm{dec}}}\log\left(1+\eta_kL_h\right)
		\vspace{-0.1cm}
	\end{eqnarray*}
	which, together with \eqref{eq:upperboundtk} and \eqref{eq:cond-alpha-gamma}, implies the conclusion.
\end{proof}

\subsection{Oracle complexity in the strongly convex case}
In this subsection, we consider the strongly convex case, i.e., $\mu>0$. By Theorem~\ref{thm:main_corollary} and Proposition~\ref{lem:inner-comp}, we can establish the overall complexity result to produce an $\vareps$-optimal solution or an $\vareps$-stationary solution of \eqref{eq:cvx-composite} by specifying the choice of the inexactness parameters $\{\vareps_k\}_{k\geqslant 0}$ and bounding $\Phi(\vx^{(k)}; \vy^{(k)},\eta_k) - \Phi(\vx^{(k+1)}_*; \vy^{(k)},\eta_k)$ from above. To do so, let $\vareps_0>0$ be any constant and define the following quantities
\vspace{-0.1cm}
\begin{eqnarray}\label{eq:choice-eps-k}
\vareps_k &=& \textstyle \frac{\vareps_0}{k+1} \sqrt{\prod_{j=0}^{k-1} (1-c\alpha_j)},~~\forall\,k\geqslant 1,\\\label{eq:D-delta}
S&=& \textstyle \frac{\sqrt\kappa}{2(1-c)^2\underline{L}}\sum_{k=0}^{\infty}\frac{\vareps_k^2}{\prod_{j=0}^{k-1} (1-c\alpha_k)}=\frac{\sqrt\kappa}{2(1-c)^2\underline{L}}\sum_{k=0}^{\infty}\frac{\vareps_0^2}{(k+1)^{2}} < \infty,\\\label{eq:choice-delta-k}
\delta_k &=&\textstyle \sqrt{\prod_{j=0}^{k-1} (1-c\alpha_j)} \sqrt{\frac{2(\psi_0 + S)}{\mu}},~~\forall\, k\geqslant 0,
\vspace{-0.1cm}
\end{eqnarray}
where $c\in [0,1)$ is the same constant as that in Theorem~\ref{thm:main_corollary} and $\kappa$ is defined in Lemma~\ref{lem:prod-alpha}. By \eqref{eq:rate-psi-k}, \eqref{eq:choice-eps-k} and \eqref{eq:D-delta}, we have
\vspace{-0.2cm}
\begin{equation}\label{eq:rate-psi-k-1}
\textstyle	\psi_{k+1}\leqslant \prod_{j=0}^k (1-c\alpha_j) \left(\psi_0 + S\right),~\forall k\geqslant 0.
\vspace{-0.1cm}
\end{equation}
With these notations and properties, $\Phi(\vx^{(k)}; \vy^{(k)},\eta_k) - \Phi(\vx^{(k+1)}_*; \vy^{(k)},\eta_k)$ can be upper bounded as follows.

\begin{lemma}\label{lem:bd-Phi-diff}
Suppose $\mu>0$ and $\{\vareps_k\}_{k\geqslant 0}$ in Algorithm~\ref{alg:iAPG} are given in \eqref{eq:choice-eps-k}. Let $\vx^{(k+1)}_*$ and $\Phi$ be defined by~\eqref{eq:def-Phi} and $\delta_k$ by~\eqref{eq:choice-delta-k} with $c\in[0,1)$. Algorithm~\ref{alg:iAPG} guarantees that
\vspace{-0.2cm}
\begin{eqnarray}
\label{eq:low-bd-Phi-k}
\Phi(\vx^{(k)}; \vy^{(k)},\eta_k) - \Phi(\vx^{(k+1)}_*; \vy^{(k)},\eta_k) 
&\leqslant&   \left\{
\begin{array}{ll}
\frac{1}{2\underline{L}} \dist\left({\textstyle\vzero, \partial F(\vx^{(0)})}\right)^2&\text{ if }k=0,\\
\frac{1}{2\underline{L}} \left( {\textstyle \vareps_{k-1} + \frac{3L_g(\delta_k +\delta_{k-1}) }{ \gamma_{\mathrm{dec}}\sqrt c} }\right)^2& \text{ if }k\geqslant 1.
\end{array}
\right.
\vspace{-0.2cm}
\end{eqnarray}
\end{lemma}
\vspace{-0.1cm}

\begin{proof}
%Suppose $k=0$. 
Notice $\vy^{(0)} = \vx^{(0)}$, and thus $\Phi(\vx^{(0)}; \vy^{(0)},\eta_0) = H(\vx^{(0)})$. By the convexity of $H$, it holds
\vspace{-0.1cm}
$$\Phi(\vx; \vy^{(0)},\eta_0) \geqslant \textstyle \left\langle\nabla g(\vx^{(0)}),\vx-\vx^{(0)}\right\rangle+\frac{1}{2\eta_0}\|\vx-\vx^{(0)}\|^2 + H(\vx^{(0)})+\langle \vxi, \vx-\vx^{(0)}\rangle
\vspace{-0.1cm}$$
for any $\vxi\in\partial H(\vx^{(0)})$. Hence,
\vspace{-0.1cm}
\begin{align}\label{eq:low-bd-Phi-case-k=0}
	\Phi(\vx^{(0)}; \vy^{(0)},\eta_0)-\Phi(\vx^{(1)}_*; \vy^{(0)},\eta_0) \leqslant &~ -\min_\vx \textstyle \big\{\left\langle\nabla g(\vx^{(0)})+\vxi,\vx-\vx^{(0)}\right\rangle+\frac{1}{2\eta_0}\|\vx-\vx^{(0)}\|^2\big\}\nonumber\\
	= & ~ \textstyle \frac{\eta_0}{2} \|\nabla g(\vx^{(0)})+\vxi\|^2 \leqslant \frac{1}{2\underline L} \|\nabla g(\vx^{(0)})+\vxi\|^2
	\vspace{-0.1cm}
\end{align} 
by noticing $\eta_0 \leqslant \frac{1}{\underline L}$. Minimizing the right-hand size of \eqref{eq:low-bd-Phi-case-k=0} over $\vxi\in\partial H(\vx^{(0)})$ gives \eqref{eq:low-bd-Phi-k} for $k=0$.

Suppose $k\geqslant 1$. By the definition of $\psi_k$ in Theorem~\ref{thm:main_corollary} and the $\mu$-strong convexity of $F$, we have 
\vspace{-0.1cm}$$
\textstyle \psi_k\geqslant\frac{\mu}{2}\|\vx^{(k)}-\vx^*\|^2+(1-(1-c)\alpha_{k})\frac{\gamma_{k}}{2}\|\vz^{(k)}-\vx^*\|^2\geqslant 
\frac{\mu}{2}\|\vx^{(k)}-\vx^*\|^2+\frac{c\mu}{2}\|\vz^{(k)}-\vx^*\|^2,\vspace{-0.1cm}
$$
where the second inequality is due to \eqref{eq:cond-alpha-gamma}. This inequality implies 
\vspace{-0.1cm}
\begin{equation}\label{eq:dist-xz-k+1-to-xstar}
\textstyle \max\left\{\|\vx^{(k)}-\vx^*\|,\sqrt{c}\|\vz^{(k)}-\vx^*\|\right\} \leqslant \sqrt{\frac{2\psi_k}{\mu}}\leqslant \sqrt{\prod_{j=0}^{k-1} (1-c\alpha_j)} \sqrt{\frac{2(\psi_0 + S)}{\mu}}=\delta_k, ~\forall\, k\geqslant 0, \vspace{-0.1cm}
\end{equation}
where the second inequality is by \eqref{eq:rate-psi-k-1} and the equality is by \eqref{eq:choice-delta-k}. Since $c\in (0,1)$ and $\vy^{(k)}$ is a convex combination of $\vx^{(k)}$ and $\vz^{(k)}$, it follows from \eqref{eq:dist-xz-k+1-to-xstar} that
\vspace{-0.1cm}
\begin{equation}\label{eq:dist-y-k+1-to-xstar}
\textstyle \|\vy^{(k)}-\vx^*\| \leqslant \frac{\delta_k}{\sqrt c} , ~\forall\, k\geqslant 0.
\vspace{-0.1cm}
\end{equation} 

By \eqref{eq:xkprecision}, it holds that $\dist\big(\vzero, \nabla g(\vy^{(k-1)}) + \frac{1}{\eta_{k-1}}(\vx^{(k)}-\vy^{(k-1)}) + \partial H(\vx^{(k)})\big) \leqslant \vareps_{k-1}$ for $k\geqslant 1$. Hence, by the definition of $\Phi$ in \eqref{eq:def-Phi}, we have
\vspace{-0.1cm}
\begin{align}
\dist\big(\vzero, \partial \Phi(\vx^{(k)}; \vy^{(k)},\eta_k)\big) \leqslant & ~ \textstyle \vareps_{k-1} + \|\nabla g(\vy^{(k)}) - \nabla g(\vy^{(k-1)})\| + \frac{1}{\eta_{k-1}}\|\vx^{(k)}-\vy^{(k-1)}\| + \frac{1}{\eta_{k}}\|\vx^{(k)}-\vy^{(k)}\| \nonumber\\
\leqslant & ~ \textstyle \vareps_{k-1} + L_g \|\vy^{(k)} - \vy^{(k-1)}\| + \frac{L_g}{\gamma_{\mathrm{dec}}} \|\vx^{(k)}-\vy^{(k-1)}\| + \frac{L_g}{\gamma_{\mathrm{dec}}}\|\vx^{(k)}-\vy^{(k)}\|\label{eq:low-bd-dist-Phi-scvx-pre}\\
\leqslant & ~ \textstyle \vareps_{k-1} + \frac{L_g(\delta_k +\delta_{k-1}) }{\sqrt c} + \frac{L_g}{\gamma_{\mathrm{dec}}} \big(\delta_k +  \frac{\delta_{k-1}}{\sqrt c} \big) + \frac{L_g}{\gamma_{\mathrm{dec}}} \big(\delta_k +  \frac{\delta_k}{\sqrt c} \big), \label{eq:low-bd-dist-Phi-scvx}
\vspace{-0.1cm}
\end{align}
where  the second inequality is by \eqref{eq:cond-alpha-gamma} and third inequality follows from \eqref{eq:dist-xz-k+1-to-xstar}, \eqref{eq:dist-y-k+1-to-xstar}, and the triangle inequality.  In addition, from the strong convexity of $\Phi(\,\cdot\,; \vy, \eta)$, it follows that for  $k\geqslant 1$,
\vspace{-0.1cm}
\begin{align}\label{eq:low-bd-Phi}
 \Phi(\vx^{(k)}; \vy^{(k)},\eta_k) - \Phi(\vx^{(k+1)}_*; \vy^{(k)},\eta_k) 
\leqslant &~{\textstyle \frac{\eta_k}{2} }\dist\big(\vzero, \partial \Phi(\vx^{(k)}; \vy^{(k)},\eta_k)\big)^2, 
\vspace{-0.1cm}
\end{align}
which, together with \eqref{eq:low-bd-dist-Phi-scvx} and the facts that $c<1, \gamma_{\mathrm{dec}} \leqslant 1, \eta_k \leqslant \frac{1}{\underline L}$, and $\delta_k  \leqslant\delta_{k-1}$, gives the result in \eqref{eq:low-bd-Phi-k} for $k\geqslant 1$. This completes the proof. 
\end{proof}

Lemma~\ref{lem:bd-Phi-diff} allows us to simplify the right-hand side of \eqref{eq:num-oracle-k} and obtain the following result. 

\begin{theorem}[Oracle complexity to obtain an $\vareps$-optimal solution]\label{thm:oracle-opt-scvx}
Suppose $\mu>0$ and $\{\vareps_k\}_{k\geqslant 0}$ in Algorithm~\ref{alg:iAPG} are given in \eqref{eq:choice-eps-k}. Also, suppose $\vx^{(k+1)}$ is computed by applying Algorithm~\ref{alg:iAPG} to \eqref{eq:def-Phi} with the inputs given in \eqref{eq:calliAPG} and the optional steps enabled. Then for any $\vareps>0$, Algorithm~\ref{alg:iAPG} with\footnote{We assume $\underline{L} = \Theta(L_g)$ in order to simplify the results. Our algorithm does not actually need to know $L_g$.} $\underline{L} = \Theta(L_g)$ can produce an $\vareps$-optimal solution of \eqref{eq:cvx-composite} by $K_{\mathrm{opt}}^{\mathrm{sc}}$ queries to $(g, \nabla g)$ and $T_{\mathrm{opt}}^{\mathrm{sc}}$ queries to $(h, \nabla h)$, where
\vspace{-0.1cm}
\begin{eqnarray}
&&\textstyle K_{\mathrm{opt}}^{\mathrm{sc}} \leqslant 2\left\lceil \log_{\gamma_{\mathrm{dec}}}\frac{\underline{L}}{L_g}\right\rceil \cdot \left\lceil \frac{\log\frac{\psi_0 + S}{\vareps}} {\log 1/(1-c/\sqrt\kappa)}\right\rceil = O\left(\sqrt\kappa\log\frac{\psi_0 + S}{\vareps}\right), \label{eq:up-bd-K}\\
&&\textstyle T_{\mathrm{opt}}^{\mathrm{sc}}=O\left(T'+\sqrt{\frac{L_g+L_h}{\mu}}\log\frac{\psi_0 + S}{\vareps} \left(\log C_\kappa + \log\log\frac{\psi_0 + S}{\vareps}\right) \right). \label{eq:comp-opt-scvx}
\vspace{-0.1cm}
\end{eqnarray}
In the above, $\kappa$ is defined in Lemma~\ref{lem:prod-alpha}, %$C_L$ is defined in Lemma~\ref{lem:bd-dist-grad}, 
\vspace{-0.1cm}
\begin{equation}\label{eq:def-T0-Ckappa}
\textstyle T'=\sqrt{1+\frac{L_h}{L_g}}\log\frac{\big(1+\frac{L_h}{L_g}\big) \dist\big(\vzero, \partial F(\vx^{(0)})\big)}{\vareps_0},\ C_\kappa = \frac{\sqrt{\kappa }(L_g+L_h)(2-c)}{\vareps_0\gamma_{\mathrm{dec}}\sqrt c (1-c)}  \sqrt{\frac{2(\psi_0 + S)}{\mu}},
\vspace{-0.1cm}
\end{equation}
and the big-Os hide universal constants depending only on $\gamma_{\mathrm{dec}}$,  $\gamma_{\mathrm{inc}}$ and $c$.
\end{theorem}

\begin{proof}
Let $\textstyle K_1 = \min_k \left\{k: \prod_{j=0}^{k-1} (1-c\alpha_j) \left(\psi_0 + S\right)\leqslant \vareps\right\}$. Then by the definition of $\psi_k$ in Theorem~\ref{thm:main_corollary} and \eqref{eq:rate-psi-k-1}, $\vx^{(K_1)}$ is an $\vareps$-optimal solution of \eqref{eq:cvx-composite}. %if 
%\begin{equation}\label{eq:def-K}
% 
%\end{equation}
Also, let $K_1'= \left\lceil\frac{\log\frac{\psi_0 + S}{\vareps}} {\log 1/(1-c/\sqrt\kappa)}\right\rceil$. Since $\alpha_j \geqslant  1/\sqrt\kappa$ by \eqref{eq:rate-alpha-mu}, we have 
$
\prod_{j=0}^{K_1'-1} (1-c\alpha_j)\left(\psi_0 + S\right)\leqslant  (1-c/\sqrt\kappa)^{K_1'}\left(\psi_0 + S\right)\leqslant \varepsilon,
$ 
which means $K_1 \leqslant K_1'$. 

%Due to the repeat-loop in Algorithm~\ref{alg:iAPG},  $\vx^{(k+1)}$ satisfying \eqref{eq:xkprecision} may need to be computed multiple times (for different $\vy^{(k)}$ and $\eta_k$) in the $k$th iteration of Algorithm~\ref{alg:iAPG}. Fortunately, according to Lemma 4 in~\cite{nesterov2013gradient}, the total number of times $\vx^{(k+1)}$ is computed across all $k$'s between $0$ and $K_1$ is no more than 
%the inequality between inequalities 16 and 17 in \cite{lin2015adaptive}, 

By Lemma~\ref{lem:boundsearchsteps}, Algorithm~\ref{alg:accellinesearch} will stop after at most $\lceil \log_{\gamma_{\mathrm{dec}}}\frac{\underline{L}}{L_g}\rceil$ iterations with  $(g, \nabla g)$ queried only twice in each iteration. Hence, until an $\vareps$-optimal solution is found, the total number of queries to  $(g, \nabla g)$ by Algorithm~\ref{alg:iAPG} is at most $2\lceil \log_{\gamma_{\mathrm{dec}}}\frac{\underline{L}}{L_g}  \rceil K_1'  $, which implies \eqref{eq:up-bd-K} when $\underline{L} = \Theta(L_g)$.

In addition, for $k\geqslant 0$, the number of queries to $(h, \nabla h)$ needed to compute $\vx^{(k+1)}$ satisfying \eqref{eq:xkprecision} is at most $T_k$ given in \eqref{eq:num-oracle-k}. Since Algorithm~\ref{alg:accellinesearch} will stop after at most $\lceil \log_{\gamma_{\mathrm{dec}}}\frac{\underline{L}}{L_g}\rceil$ iterations by Lemma~\ref{lem:boundsearchsteps}, the number of queries to $(h, \nabla h)$ in the $k$th iteration of Algorithm~\ref{alg:iAPG} is no more than $\lceil \log_{\gamma_{\mathrm{dec}}}\frac{\underline{L}}{L_g}\rceil T_k$. Hence, applying~\eqref{eq:low-bd-Phi-k} to the right-hand side of \eqref{eq:num-oracle-k}, we can show that the total number of queries to $(h, \nabla h)$ before finding an $\vareps$-optimal solution is at most
\vspace{-0.1cm}
\begin{equation}\label{eq:argu-T-opt}
\begin{aligned}
T_{\mathrm{opt}}^{\mathrm{sc}} =& ~  \textstyle O\left(\left\lceil \log_{\gamma_{\mathrm{dec}}}\frac{\underline{L}}{L_g}\right\rceil \sqrt{1+\frac{L_h}{\underline{L}}}\log\frac{\sqrt{L_g+L_h+L_h^2/\underline L}\sqrt{\frac{1}{2\underline L} } \dist\big(\vzero, \partial F(\vx^{(0)})\big)}{\vareps_0}\right) \\
&~ \textstyle + \sum_{k=1}^{K_1-1} O\left(\left\lceil \log_{\gamma_{\mathrm{dec}}}\frac{\underline{L}}{L_g}\right\rceil \sqrt{1+\frac{L_h}{\underline{L}}}\log\frac{\sqrt{L_g+L_h+L_h^2/\underline L}\sqrt{\frac{1}{2\underline L} } \left(\vareps_{k-1} + \frac{3L_g(\delta_k +\delta_{k-1}) }{ \gamma_{\mathrm{dec}}\sqrt c} \right)}{\vareps_k}\right).
\end{aligned}
\vspace{-0.1cm}
\end{equation}
By the definition of $\delta_k$ in \eqref{eq:choice-delta-k} and the choice of $\vareps_k$ in \eqref{eq:choice-eps-k}, we have
\vspace{-0.1cm}
\begin{align}\label{eq:bd-log-frac}
\textstyle \frac{\vareps_{k-1} + \frac{3L_g(\delta_k +\delta_{k-1}) }{ \gamma_{\mathrm{dec}}\sqrt c}}{\vareps_k} = &~ \textstyle \frac{(k+1)}{k(1-c\alpha_{k-1})}+\frac{3L_g(k+1)}{\vareps_0\gamma_{\mathrm{dec}}\sqrt c}  \sqrt{\frac{2(\psi_0 + S)}{\mu}} + \frac{3L_g(k+1)}{\vareps_0\gamma_{\mathrm{dec}}\sqrt c (1-c\alpha_{k-1})}  \sqrt{\frac{2(\psi_0 + S)}{\mu}}\nonumber\\
\leqslant &~ \textstyle \frac{2}{1-c} + \frac{3L_g(2-c)(k+1)}{\vareps_0\gamma_{\mathrm{dec}}\sqrt c (1-c)}  \sqrt{\frac{2(\psi_0 + S)}{\mu}}
\vspace{-0.1cm}
\end{align}
for any $k\geqslant 1$, where the inequality comes from $\alpha_{k-1}\leqslant 1$. Substituting \eqref{eq:bd-log-frac} into \eqref{eq:argu-T-opt} and using the facts that $K_1 \leqslant   K_1' $ and $\underline{L} = \Theta(L_g)$, we obtain the desired result in \eqref{eq:comp-opt-scvx} and complete the proof.
\end{proof}

%\begin{remark}
%We assume $\underline{L} = \Theta(L_g)$ to simplify the results. Our algorithm does not actually need to know $L_g$.
%\end{remark}

The theorem below establishes the oracle complexity for producing an $\vareps$-stationary point. 
\begin{theorem}[Oracle complexity to obtain an $\vareps$-stationary solution]\label{thm:oracle-crit-scvx}
Suppose all assumptions in Theorem~\ref{thm:oracle-opt-scvx} hold. Then for any $\vareps>0$, Algorithm~\ref{alg:iAPG} with $\underline{L} = \Theta(L_g)$ and the optional steps enabled can produce an $\vareps$-stationary solution of \eqref{eq:cvx-composite} by $K_{\mathrm{crit}}^{\mathrm{sc}}$ queries to $(g, \nabla g)$ and $T_{\mathrm{crit}}^{\mathrm{sc}}$ queries to $(h, \nabla h)$, where
\vspace{-0.1cm}
\begin{eqnarray}
&&\textstyle K_{\mathrm{crit}}^{\mathrm{sc}} = O\left(\left\lceil \log_{\gamma_{\mathrm{dec}}}\frac{\underline{L}}{L_g}\right\rceil \cdot \left\lceil \frac{\log\frac{C_L^2(\psi_0 + S)}{\vareps^2}} {\log 1/(1-c/\sqrt\kappa)}\right\rceil\right) = O\left(\sqrt\kappa\log\frac{C_L^2(\psi_0 + S)}{\vareps^2}\right), \label{eq:up-bd-K-crit}\\
&&\textstyle T_{\mathrm{crit}}^{\mathrm{sc}}=O\left(T'+\sqrt{\frac{L_g+L_h}{\mu}}\log\frac{C_L^2(\psi_0 + S)}{\vareps^2} \left(\log C_\kappa + \log\log\frac{C_L^2(\psi_0 + S)}{\vareps^2}\right) \right). \label{eq:comp-crit-scvx}
\vspace{-0.1cm}
\end{eqnarray}
In the above, $\kappa$ is defined in Lemma~\ref{lem:prod-alpha}, $\psi_0$ is in Theorem~\ref{thm:main_corollary}, $C_L$ is defined in Lemma~\ref{lem:bd-dist-grad}, $T'$ and $C_\kappa$ are given in \eqref{eq:def-T0-Ckappa}, and the big-Os hide universal constants depending only on $\gamma_{\mathrm{dec}}$, $\gamma_{\mathrm{inc}}$ and $c$.
\end{theorem}

\begin{proof}
It follows from the definition of $\psi_k$, \eqref{ineq:stationary0} and \eqref{eq:rate-psi-k-1} that
\vspace{-0.1cm}
\begin{equation}\label{eq:bd-stationary-noise}
\dist\big(\vzero, \partial F(\widetilde\vx^{(k+1)})\big) \leqslant \textstyle C_L \sqrt{2\prod_{j=0}^k (1-c\alpha_j)\left(\psi_0 + S\right)}, ~\forall\,k\geqslant 0. 
\vspace{-0.1cm}
\end{equation}
Let $K_2 = \min_k \left\{k: \prod_{j=0}^{k-1} (1-c\alpha_j) \left(\psi_0 + S\right)\leqslant \frac{\vareps^2}{2C_L^2}\right\}$. Then $\widetilde\vx^{(K_2)}$ is an $\vareps$-stationary point. %if
%\begin{equation}\label{eq:def-K-crit}
%K_2 = \min_k \left\{k: \prod_{j=0}^{k-1} (1-c\alpha_j) \left(\psi_0 + S\right)\leqslant \frac{\vareps^2}{2C_L^2}\right\}.
%\end{equation}
Also, let $K_2'= \left\lceil \frac{\log\frac{2(\psi_0 + S)C_L^2}{\vareps^2}} {\log 1/(1-c/\sqrt\kappa)}\right\rceil$. Since $\alpha_j \geqslant  1/\sqrt\kappa$ by \eqref{eq:rate-alpha-mu}, we have 
$
\prod_{j=0}^{K_2'-1} (1-c\alpha_j)\left(\psi_0 + S\right)\leqslant  (1-c/\sqrt\kappa)^{K_2'}\left(\psi_0 + S\right)\leqslant \frac{\vareps^2}{2C_L^2},
$ 
which means $K_2 \leqslant K_2'$. 

By Lemma~\ref{lem:boundsearchsteps}, Algorithm~\ref{alg:accellinesearch} will stop after at most $\lceil \log_{\gamma_{\mathrm{dec}}}\frac{\underline{L}}{L_g}\rceil$ iterations with  $(g, \nabla g)$ queried only twice in each iteration, and %By Lemma~\ref{lem:boundsearchsteps} 
also, the total number of queres to $(g, \nabla g)$ by 
Algorithm~\ref{alg:seekstationary} is at most  $O(K_2)$. Hence, until an $\vareps$-stationary solution is found, the total number of queries to  $(g, \nabla g)$ by Algorithm~\ref{alg:iAPG} is at most $2\lceil \log_{\gamma_{\mathrm{dec}}}\frac{\underline{L}}{L_g}  \rceil K_2'+O(K_2')  $, which implies \eqref{eq:up-bd-K-crit} when $\underline{L} = \Theta(L_g)$.

In addition, similar to \eqref{eq:argu-T-opt}, we have
\vspace{-0.2cm}
\begin{equation}\label{eq:argu-T-crit-1}
\begin{aligned}
T_{\mathrm{crit}}^{\mathrm{sc}} =& ~  \textstyle O\left(\left\lceil \log_{\gamma_{\mathrm{dec}}}\frac{\underline{L}}{L_g}\right\rceil \sqrt{1+\frac{L_h}{\underline{L}}}\log\frac{\sqrt{L_g+L_h+L_h^2/\underline{L}}\sqrt{\frac{1}{2\underline L} } \dist\big(\vzero, \partial F(\vx^{(0)})\big)}{\vareps_0}\right) \\
&~ \textstyle  + \sum_{k=1}^{K_2-1} O\left(\left\lceil \log_{\gamma_{\mathrm{dec}}}\frac{\underline{L}}{L_g}\right\rceil \sqrt{1+\frac{L_h}{\underline{L}}}\log\frac{\sqrt{L_g+L_h+L_h^2/\underline{L}}\sqrt{\frac{1}{\underline L} } \left(\vareps_{k-1} + \frac{3L_g(\delta_k +\delta_{k-1}) }{ \gamma_{\mathrm{dec}}\sqrt c} \right)}{\vareps_k}\right)\end{aligned}
\vspace{-0.2cm}
\end{equation}
By \eqref{eq:bd-log-frac} and $K_2 \leqslant K_2'$, we obtain the result in \eqref{eq:comp-crit-scvx} and complete the proof.
\end{proof}

\subsection{Oracle complexity in the convex case}\label{sec:rate-compl-cvx}
In this subsection, we consider the convex case, i.e., $\mu=0$. Though the complexity result will not be used in later sections, the result has its own merit. The following technical lemma is needed in our convergence analysis. It is obtained by applying inequality $\sqrt{a+b}\leqslant \sqrt{a}+\sqrt{b}$ for $a,b\geqslant 0$ to the conclusion of Lemma 1 in \cite{schmidt2011convergence}.
\begin{lemma}\label{lem:bd-u}
Let $\{u_k\}_{k\geqslant 1}$ be a sequence of nonnegative numbers. Suppose 
%\begin{equation*}%\label{eq:cond-u}
$u_k^2 \leqslant C + \sum_{i=1}^k \lambda_i u_i, \forall\, k\geqslant 1,$
%\end{equation*}
where $C\geqslant 0$ is a constant and $\lambda_i \geqslant 0$ for all $i\geqslant 1$. Then $u_k\leqslant \sum_{i=1}^k\lambda_i + \sqrt{C}, \forall\, k\geqslant 1$.
\end{lemma}

\iffalse
\begin{proof}
By the quadratic formula, we have $u_k \leqslant \frac{\lambda_k}{2}+\sqrt{C+\lambda_k^2/4+\sum_{i=1}^{k-1}\lambda_i u_i}$ for all $k\geqslant 1$. Using the fact $\sqrt{a+b}\leqslant \sqrt{a} + \sqrt{b}$ for any nonnegative numbers $a$ and $b$, we then have 
\begin{equation}\label{eq:bd-uk}
u_k \leqslant \lambda_k + \sqrt{C+\sum_{i=1}^{k-1}\lambda_i u_i}, \forall\, k\geqslant 1.
\end{equation}

Below we prove the desired result by induction. First, we have from \eqref{eq:bd-uk} that $u_1\leqslant \lambda_1+\sqrt{C}$, namely, the desired result holds for $k=1$. Second, assume $u_t\leqslant \sum_{i=1}^t\lambda_i + \sqrt{C}, \forall\, 1\leqslant t\leqslant k-1$. Then $u_i \leqslant \sum_{i=1}^{k-1}\lambda_i + \sqrt{C}$ for all $i\leqslant k-1$, due to the nonnegativity of $\{\lambda_i\}$. Hence, the inequality in \eqref{eq:bd-uk} indicates
$$u_k\leqslant \lambda_k+ \sqrt{C+\sum_{i=1}^{k-1}\lambda_i \left(\sum_{i=1}^{k-1}\lambda_i + \sqrt{C}\right)} \leqslant \sqrt{C+\sum_{i=1}^{k-1}\lambda_i \left(\sum_{i=1}^{k-1}\lambda_i + 2\sqrt{C}\right)} = \sum_{i=1}^k\lambda_i + \sqrt{C},$$
which completes the induction.
\end{proof}
\fi

Like Theorem~\ref{thm:main_corollary}, the convergence of Algorithm~\ref{alg:iAPG} in the convex case can be derived from \eqref{eq:main}.

\begin{theorem}\label{thm:main_cvx_rate} 
Algorithm~\ref{alg:iAPG} guarantees that, for $k \geqslant 0$, 
\vspace{-0.1cm}
\begin{equation}\label{eq:dist-zk-term}
\textstyle 	\|\vx^*-\vz^{(k+1)}\| \leqslant 2\sum_{t=0}^{k} \frac{\eta_t\vareps_t}{\alpha_t} + \sqrt{\frac{2\eta_0(1-\alpha_0)}{\alpha_0^2}\phi_0}, \text{ and } 
\vspace{-0.2cm}
\end{equation}
\vspace{-0.2cm}
\begin{equation}\label{eq:rate-phiK}
\textstyle \phi_{k+1}\leqslant \frac{4\tau^2}{(k+2)^2}\left[\frac{1-\alpha_0}{\alpha_0^2} + \sqrt\tau \sum_{t=0}^{k}{\vareps_t}(t+\frac{1}{\alpha_0}) \left(\textstyle \frac{2\sqrt\tau}{\underline{L}}\sum_{j=0}^{t} \vareps_j(j+\frac{1}{\alpha_0}) + \sqrt{\frac{2\eta_0(1-\alpha_0)}{\alpha_0^2}\phi_0}\right)\right],
\vspace{-0.1cm}
\end{equation}
where 
%\begin{equation}\label{eq:def-phi}
	$\phi_k: = F(\vx^{(k)})-F^*+\frac{\gamma_{k}}{2}\|\vx^*-\vz^{(k)}\|^2, \forall\, k\geqslant 0$ 
%\end{equation}
and $\tau$ is defined in Lemma~\ref{lem:prod-alpha}.
\end{theorem}

\begin{proof}
By the definition of $\phi_k$, we can rewrite \eqref{eq:main} as 
$\phi_{k+1} \leqslant (1-\alpha_k)\phi_k + \vareps_k \alpha_k\|\vx^*-\vz^{(k+1)}\|.$
Applying this inequality recursively gives
\vspace{-0.1cm}
\begin{eqnarray}\label{eq:main1-cvx}
	\phi_{k+1} &\leqslant & \textstyle \prod_{t=0}^{k}(1-\alpha_t)\phi_0 + \sum_{t=0}^{k}\Big(\prod_{j=t+1}^{k}(1-\alpha_j)\Big) \vareps_t\alpha_t\|\vx^*-\vz^{(t+1)}\|\cr
	&=& \textstyle \gamma_{k+1}\frac{\eta_0(1-\alpha_0)}{\alpha_0^2} + \gamma_{k+1}\sum_{t=0}^{k} \frac{\eta_t\vareps_t}{\alpha_t}\|\vx^*-\vz^{(t+1)}\|,
	\vspace{-0.1cm}
\end{eqnarray}
where the equality follows from $1-\alpha_k = \frac{\alpha_{k}^2}{\alpha_{k-1}^2}\frac{\eta_{k-1}}{\eta_k}, \forall\, k\geqslant 1$ and the updating equation of $\gamma_{k+1}$. The equation in \eqref{eq:main1-cvx} together with the definition of $\phi_k$ %in Theorem~\ref{thm:main_cvx_rate} 
implies
\vspace{-0.1cm}
\begin{equation}\label{eq:dist-zk}
\textstyle 	\|\vx^*-\vz^{(k+1)}\|^2 \leqslant \frac{2\eta_0(1-\alpha_0)}{\alpha_0^2}\phi_0 + 2\sum_{t=0}^{k} \frac{\eta_t\vareps_t}{\alpha_t}\|\vx^*-\vz^{(t+1)}\|, \forall\, k\geqslant 0.
\vspace{-0.1cm}
\end{equation}

Applying Lemma~\ref{lem:bd-u} to \eqref{eq:dist-zk} gives \eqref{eq:dist-zk-term}, which, together with \eqref{eq:main1-cvx}, further implies
\vspace{-0.1cm}
\begin{equation}\label{eq:main2-cvx}
\textstyle \phi_{k+1}\leqslant \gamma_{k+1}\left[\frac{\eta_0(1-\alpha_0)}{\alpha_0^2} + \sum_{t=0}^{k}\frac{\eta_t\vareps_t}{\alpha_t} \left(\textstyle 2\sum_{j=0}^{t} \frac{\eta_j\vareps_j}{\alpha_j} + \sqrt{\frac{2\eta_0(1-\alpha_0)}{\alpha_0^2}\phi_0}\right)\right].
\vspace{-0.1cm}
\end{equation}
Notice $\gamma_{k+1}=\frac{\alpha_{k}^2}{\eta_{k}}$ and $\frac{\eta_j}{\eta_k}<{\tau}$ for all $k$ and $j\geqslant 0$. We have from \eqref{eq:rate-alpha-mu} and \eqref{eq:main2-cvx} that 
\vspace{-0.1cm}
\begin{equation}\label{eq:main3-cvx}
\textstyle \phi_{k+1}\leqslant \frac{4\tau^2}{(k+2)^2}\left[\frac{1-\alpha_0}{\alpha_0^2} + \sum_{t=0}^{k}\frac{\vareps_t}{\alpha_t} \left(\textstyle 2\sum_{j=0}^{t} \frac{\eta_j\vareps_j}{\alpha_j} + \sqrt{\frac{2\eta_0(1-\alpha_0)}{\alpha_0^2}\phi_0}\right)\right].
\vspace{-0.1cm}
\end{equation}
Plug into \eqref{eq:main3-cvx} the lower bound of $\alpha_{k}$ in \eqref{eq:rate-alpha-mu} and the upper bound of $\eta_i$ in \eqref{eq:cond-alpha-gamma}. We have the assertion of the theorem.
\end{proof}

Below we specify the choice of the inexactness $\{\vareps_k\}_{k\geqslant 0}$ and establish the oracle complexity results. To do so, let $\vareps_0>0$  and $\delta>0$ be any constants and define the following quantities
\vspace{-0.1cm}
\begin{eqnarray}\label{eq:eps-cvx}
\vareps_k &=& \textstyle \frac{\vareps_0}{(k+1)^{2+\delta}}, \forall\, k\geqslant 0, \quad %\\\label{eq:def-C-delta}
S_\delta = \textstyle \sum_{k=0}^\infty \vareps_k(k+\frac{1}{\alpha_0}) = \sum_{k=0}^\infty \frac{\vareps_0}{(k+1)^{2+\delta}} (k+\frac{1}{\alpha_0})< \infty.
\vspace{-0.1cm}
\end{eqnarray}
We will use Proposition~\ref{lem:inner-comp} to bound the number of queries to $(h, \nabla h)$. Similar to the previous subsection, we first  bound $\Phi(\vx^{(k)}; \vy^{(k)},\eta_k) - \Phi(\vx^{(k+1)}_*; \vy^{(k)},\eta_k)$. Let 
%\vspace{-0.1cm}
%\begin{equation}\label{eq:def-X-cvx}
$\textstyle X=\left\{\vx: \|\vx -\vx^*\| \leqslant \frac{2\sqrt\tau}{\underline{L}}S_\delta + \sqrt{\frac{2\eta_0(1-\alpha_0)}{\alpha_0^2}\phi_0}\right\},$ 
%\vspace{-0.1cm}
%\end{equation} 
and $D_X:=\max_{\vx_1, \vx_2\in X}\|\vx_1-\vx_2\|$ be the diameter of $X$. Also, define
\vspace{-0.1cm}
\begin{equation}\label{eq:def-B-Phi}
B_\Phi=\max\left\{\big(\vareps_0 + {\textstyle L_g D_X(1+ \frac{2}{\gamma_{\mathrm{dec}}} }) \big)^2,\  \dist\big(\vzero, \partial F(\vx^{(0)})\big)^2 \right\}.
\vspace{-0.1cm}
\end{equation}

% in a similar way as we did for the strongly-convex case. 

\begin{lemma}\label{lem:bd-Phi-diff-cvx}
Suppose $\{\vareps_k\}_{k\geqslant 0}$ in Algorithm~\ref{alg:iAPG} are given in \eqref{eq:eps-cvx}. Let $\vx^{(k+1)}_*$ and $\Phi$ be defined by~\eqref{eq:def-Phi}. Algorithm~\ref{alg:iAPG} guarantees that (i) $\vx^{(k)}$, $\vy^{(k)}$, $\vz^{(k)}\in X$ for all $k\geqslant 0$ and that (ii) $\Phi(\vx^{(k)}; \vy^{(k)},\eta_k) - \Phi(\vx^{(k+1)}_*; \vy^{(k)},\eta_k)\leqslant \frac{B_\Phi}{2\underline{L}} , \forall\, k\geqslant 0$, where $B_\Phi$ is defined in \eqref{eq:def-B-Phi}.
%\begin{enumerate}
%\item[(i)]  $\vx^{(k)}\in X$, $\vy^{(k)}\in X$, and $\vz^{(k)}\in X$ for all $k\geqslant 0$ and that
%\item[(ii)] $\Phi(\vx^{(k)}; \vy^{(k)},\eta_k) - \Phi(\vx^{(k+1)}_*; \vy^{(k)},\eta_k)\leqslant \frac{B_\Phi}{2\underline{L}} , \forall\, k\geqslant 0$, where $B_\Phi$ is defined in \eqref{eq:def-B-Phi}. %with $X$ in \eqref{eq:def-X-cvx}.
%\end{enumerate}
\end{lemma}

\begin{proof}
It follows that $\vx^{(0)}=\vz^{(0)}\in X$ by the definition of $\phi_0$ in Theorem~\ref{thm:main_cvx_rate} and the relation $\frac{\alpha_0^2}{\eta_0} = (1-\alpha_0)\gamma_0$. From \eqref{eq:dist-zk-term}, \eqref{eq:eps-cvx}, \eqref{eq:rate-alpha-mu}, and \eqref{eq:cond-alpha-gamma}, we have 
\vspace{-0.1cm}
$$
\textstyle \|\vx^*-\vz^{(k+1)}\| \leqslant 2\sum_{t=0}^{k} \frac{\sqrt{\tau}(t+1/\alpha_0)\vareps_t}{\underline{L}} + \sqrt{\frac{2\eta_0(1-\alpha_0)}{\alpha_0^2}\phi_0}
\leqslant  \frac{2\sqrt\tau}{\underline{L}}S_\delta + \sqrt{\frac{2\eta_0(1-\alpha_0)}{\alpha_0^2}\phi_0}, 
\vspace{-0.1cm}
$$ 
so $\vz^{(k+1)}\in X,\forall\, k\geqslant 0$. Since $\vx^{(k+1)} = \alpha_k \vz^{(k+1)} + (1-\alpha_k)\vx^{(k)}$, by induction, we can show $\vx^{(k+1)}\in X,\forall\, k\geqslant 0$. From the updating equation of $\vy^{(k)}$, we also have $\vy^{(k)}\in X,\forall\, k\geqslant 0$. This proves assertion (i).

%Next, we prove assertion (ii). %Since $\dist\big(\vzero, \nabla g(\vy^{(k-1)}) + \frac{1}{\eta_{k-1}}(\vx^{(k)}-\vy^{(k-1)}) + \partial H(\vx^{(k)})\big) \leqslant \vareps_{k-1}$ for $k\geqslant 1$, we have  \eqref{eq:low-bd-dist-Phi-scvx-pre} by the definition of $\Phi$ in \eqref{eq:def-Phi}. 
From Assertion (i) and \eqref{eq:low-bd-dist-Phi-scvx-pre}, we have %imply
$ \dist\big(\vzero, \partial \Phi(\vx^{(k)}; \vy^{(k)},\eta_k)\big) \leqslant \textstyle \vareps_0 + L_g D_X(1+\frac{2}{\gamma_{\mathrm{dec}}}),\forall\, k\geqslant 1,$
which, together with \eqref{eq:low-bd-Phi} and \eqref{eq:cond-alpha-gamma}, gives 
\vspace{-0.2cm}
\begin{align*}
\textstyle  \Phi(\vx^{(k)}; \vy^{(k)},\eta_k) - \Phi(\vx^{(k+1)}_*; \vy^{(k)},\eta_k) 
\leqslant &~{\textstyle \frac{1}{2\underline{L}} }\left( \textstyle \vareps_0 + L_g D_X(1+\frac{2}{\gamma_{\mathrm{dec}}})\right)^2, \forall\, k\geqslant 1. 
\vspace{-0.2cm}
\end{align*}
Hence, assertion (ii) holds for $k\geqslant 1$. For the case of $k=0$, we use \eqref{eq:low-bd-Phi-case-k=0} and complete the proof.
\end{proof}

Now we are ready to show the oracle complexity result to obtain an $\vareps$-optimal solution.
\begin{theorem}[Oracle complexity to obtain an $\vareps$-optimal solution] 
Suppose $\{\vareps_k\}_{k\geqslant 0}$ in Algorithm~\ref{alg:iAPG} are given in \eqref{eq:choice-eps-k}. Also, suppose $\vx^{(k+1)}$ is computed by applying Algorithm~\ref{alg:iAPG} to \eqref{eq:def-Phi} with the inputs given in \eqref{eq:calliAPG} and the optional steps enabled. Then for any $\vareps>0$, Algorithm~\ref{alg:iAPG} with $\underline{L} = \Theta(L_g)$ can produce an $\vareps$-optimal solution of \eqref{eq:cvx-composite} by $K_{\mathrm{opt}}^{\mathrm{cvx}}$ queries to $(g, \nabla g)$ and $T_{\mathrm{opt}}^{\mathrm{cvx}}$ queries to $(h, \nabla h)$, where
\vspace{-0.1cm}
\begin{eqnarray}
&&\textstyle K_{\mathrm{opt}}^{\mathrm{cvx}} \leqslant 2\left\lceil \log_{\gamma_{\mathrm{dec}}}\frac{\underline{L}}{L_g}\right\rceil\left(\left\lceil \frac{2\tau \sqrt{C_0}}{\sqrt\vareps} \right\rceil - 1\right)\label{eq:num-q-cvx-g},\\
&&\textstyle T_{\mathrm{opt}}^{\mathrm{cvx}} = O\left(\left\lceil \frac{2\tau \sqrt{C_0}}{\sqrt\vareps} \right\rceil \sqrt{1+\frac{L_h}{L_g}} \Big(\log\frac{\left(1+\frac{L_h}{L_g}\right)\sqrt{B_\Phi}}{\vareps_0}+(2+\delta)\log\frac{2\tau \sqrt{C_0}}{\sqrt\vareps}\Big)\right).\label{eq:num-q-cvx-h}
\vspace{-0.1cm}
\end{eqnarray}
In the inequality and equality above, $C_0=\frac{1-\alpha_0}{\alpha_0^2} + \sqrt{\tau}S_\delta \left(\textstyle \frac{2\sqrt\tau}{\underline{L}}S_\delta + \sqrt{\frac{2\eta_0(1-\alpha_0)}{\alpha_0^2}\phi_0}\right)$ and the big-Os hide universal constants depending only on $\gamma_{\mathrm{dec}}$ and $\gamma_{\mathrm{inc}}$.
\end{theorem}

\begin{proof}
From \eqref{eq:rate-phiK}, the definition of $\phi_k$ in Theorem~\ref{thm:main_cvx_rate}, the choice of $\{\vareps_k\}_{k\geqslant 0}$ in \eqref{eq:eps-cvx}, and the definition of $S_\delta$ in \eqref{eq:eps-cvx}, it follows that
\vspace{-0.2cm}
\begin{equation}\label{eq:complexity-phiK}
\textstyle F(\vx^{(k+1)})-F^*\leqslant \phi_{k+1}\leqslant \frac{4\tau^2}{(k+2)^2}\left[\frac{1-\alpha_0}{\alpha_0^2} + \sqrt{\tau} S_\delta \left(\textstyle \frac{2\sqrt\tau}{\underline{L}}S_\delta + \sqrt{\frac{2\eta_0(1-\alpha_0)}{\alpha_0^2}\phi_0}\right)\right]=  \frac{4\tau^2C_0}{(k+2)^2}.
\vspace{-0.2cm}
\end{equation}
Hence, in order to produce an $\vareps$-optimal solution, Algorithm~\ref{alg:iAPG} needs no more than $K_3=\left\lceil \frac{2\tau}{\sqrt\vareps} \sqrt{C_0} \right\rceil - 1$ 
%\begin{eqnarray}
%	\label{eq:bd-K-cvx}
%	K_3=\left\lceil \frac{2\tau}{\sqrt\vareps} \sqrt{C_0} \right\rceil - 1
%\end{eqnarray}
iterations. By Lemma~\ref{lem:boundsearchsteps}, Algorithm~\ref{alg:accellinesearch} will stop after at most $\lceil \log_{\gamma_{\mathrm{dec}}}\frac{\underline{L}}{L_g}\rceil$ iterations with $(g, \nabla g)$ queried twice in each iteration. Hence, $K_{\mathrm{opt}}^{\mathrm{cvx}}$ satisfies \eqref{eq:num-q-cvx-g}. 

Moreover, by Proposition~\ref{lem:inner-comp}, Lemma~\ref{lem:bd-Phi-diff-cvx} and the choice of $\vareps_k$ in \eqref{eq:eps-cvx}, we have that the total number of the queries to $(h, \nabla h )$ is  
\vspace{-0.3cm}
\begin{eqnarray*}
T_{\mathrm{opt}}^{\mathrm{cvx}} &= & \textstyle \sum_{k=0}^{K_3-1} O\left(\left\lceil\log_{\gamma_{\mathrm{dec}}}\frac{\underline{L}}{L_g}\right\rceil\sqrt{1+\frac{L_h}{\underline{L}}} \log\frac{\sqrt{L_g+L_h+L_h^2/\underline{L}}\sqrt{\frac{B_\Phi}{2\underline L}}(k+1)^{2+\delta}}{\vareps_0}\right),
\vspace{-0.2cm}
\end{eqnarray*}
which, together with the facts that $K_3=\left\lceil \frac{2\tau}{\sqrt\vareps} \sqrt{C_0} \right\rceil - 1$ and that $\underline{L} = \Theta(L_g)$, gives the result in \eqref{eq:num-q-cvx-h}.
\end{proof}

\section{Inexact proximal augmented Lagrangian method}
In this section, we consider the affine-constrained composite problem
\vspace{-0.2cm}
\begin{equation}\label{eq:aff-prob}
\min_\vx \left\{G(\vx):=f(\vx) + r(\vx)\right\}, \st \vA_E\vx = \vb_E,\ \vA_I\vx \leqslant \vb_I, 
\vspace{-0.2cm}
\end{equation}
where $f$ is $L_f$-smooth and $\mu$-strongly convex with $\mu\geqslant 0$, and $r$ is closed convex and admits an easy proximal mapping. We assume that $(f, \nabla f)$ is significantly more expensive than $(\vA(\cdot), \vA^\top(\cdot))$ to evaluate, where $\vA=[\vA_E; \vA_I]$. We denote the Lagrangian multipliers to the affine constraints by $\vlam = [\vlam_E; \vlam_I]$ with $\vlam_E$ and $\vlam_I$ associated to the equality and inequality constraints, respectively. We also assume  \eqref{eq:aff-prob} has a finite optimal solution $\vx^*$ and has the corresponding Lagrange multiplier $\vlam^*= [\vlam_E^*; \vlam_I^*]$ satisfying 
\vspace{-0.2cm}
\begin{equation}\label{eq:kkt-cond}
	\vzero\in \partial G(\vx^*) + \vA^\top\vlam^*; \quad \vA_E\vx^*- \vb_E=\vzero, \ \vA_I\vx^*- \vb_I \leqslant \vzero; \quad  \vlam_I^* \geqslant \vzero,~\langle \vlam_I^*, \vA_I\vx^*- \vb_I \rangle = 0.
	\vspace{-0.2cm}
\end{equation}
Our goal is to find a solution that satisfies \eqref{eq:kkt-cond} with $\vareps$-precision, which is defined formally below.

\begin{definition}[$\vareps$-KKT solution]\label{def:eps-kkt}
For a given $\vareps\geqslant 0$, a point $\bar\vx\in \dom(G)$ is called an $\vareps$-KKT solution of \eqref{eq:aff-prob}, if there is a multiplier $\bar\vlam= [\bar\vlam_E; \bar\vlam_I]$ such that
\vspace{-0.1cm}
\begin{equation}\label{eq:eps-kkt-cond}
\dist\big(\vzero, \partial G(\bar\vx) + \vA^\top\bar\vlam\big) \leqslant \vareps; \sqrt{\|\vA_E\bar\vx- \vb_E\|^2 +  \|[\vA_I\bar\vx- \vb_I]_+\| ^2} \leqslant \vareps; \bar\vlam_I\geqslant\vzero,\,\|\bar\vlam_I \odot (\vA_I\bar\vx- \vb_I)\| \leqslant \vareps.
\vspace{-0.1cm}
\end{equation}
\end{definition}
\vspace{-0.3cm}

We consider an inexact proximal augmented Lagrangian method (iPALM) presented in Algorithm~\ref{alg:ipalm} for finding an $\vareps$-KKT solution for \eqref{eq:aff-prob}. At iteration $k$, based on the current solution $\vx^{(k)}$ and the Lagrangian multiplier $\vlam^{(k)}$, the iPALM generates the next solution $\vx^{(k+1)}$  by approximately solving  the following proximal augmented Lagrangian subproblem 
\vspace{-0.1cm}
\begin{equation}\label{eq:Psik}
\vx^{(k+1)}\approx\argmin\limits_{\vx}\big\{\textstyle \Psi_k(\vx):=\cL_{\beta_k}(\vx, \vlam^{(k)})+\frac{\rho_k}{2}\|\vx-\vx^{(k)}\|^2\big\}.
\vspace{-0.1cm}
\end{equation}
Here, $\cL_\beta$ is the classic augmented Lagrangian function of \eqref{eq:aff-prob} and has the following form: 
\vspace{-0.1cm}
$$
\textstyle \cL_\beta(\vx,\vlam) = G(\vx) + \langle\vlam_E, \vA_E\vx-\vb_E\rangle + \frac{\beta}{2}\|\vA_E\vx-\vb_E\|^2 + \frac{1}{2\beta}\left(\big\|[\beta(\vA_I\vx - \vb_I) + \vlam_I]_+\big\|^2 - \|\vlam_I\|^2\right).
\vspace{-0.1cm}
$$ 
In particular, the iPLAM finds $\vx^{(k+1)}$ as an $\bar\vareps_k$-stationary point of $\Psi_k$, i.e., \eqref{eq:find-x-k-cond} holds. We can apply the iAPG method in Algorithm~\ref{alg:iAPG} to find $\vx^{(k+1)}$. 
We will show that, compared to existing results, the oracle complexity of the iPALM that uses  the iAPG as a subroutine can significantly reduce the number of queries to $(f, \nabla f)$ while just slightly increasing the number of queries to $(\vA(\cdot), \vA^\top(\cdot))$.  

%Suppose  $\beta_k>0$ is the penalty parameter and $\rho_k>0$ is the proximal parameter for $k\geqslant 0$. 

\begin{algorithm}[h]
{\small
\caption{Inexact proximal augmented Lagrangian method (iPALM)}
\label{alg:ipalm}
\DontPrintSemicolon
\textbf{Initialization:} $\vx^{(0)}\in\dom(G)$ and $\vlam^{(0)}$; set $k=0$\;
\While{Condition \eqref{eq:eps-kkt-cond} with $(\bar\vx, \bar\vlam) =(\vx^{(k)}, \vlam^{(k)})$ does not hold}{
Find $\vx^{(k+1)}\approx\argmin\limits_{\vx}\Psi_k(\vx)$ such that 
\vspace{-0.2cm}
\begin{equation}\label{eq:find-x-k-cond}
\dist\big(\vzero, \partial \Psi_k(\vx^{(k+1)})\big)\leqslant \bar\vareps_k,
\vspace{-0.2cm}
\end{equation}
where $\Psi_k$ is defined in \eqref{eq:Psik}.\;
Let $\vlam_E^{(k+1)}= \vlam_E^{(k)} + \beta_k (\vA_E\vx^{(k+1)} - \vb_E),\ \vlam_I^{(k+1)}= [\vlam_I^{(k)} + \beta_k (\vA_I\vx^{(k+1)} - \vb_I)]_+$, and
set $k\gets k+1$.
}	
}
\end{algorithm}

Before giving the details, %on solving \eqref{eq:Psik}, 
we first present the following lemma that characterizes the relationship between %the solutions in 
two consecutive iterates of Algorithm~\ref{alg:ipalm}.

\begin{lemma}
Let $\vx^{(k)}$ and $\vlam^{(k)}$ be generated from Algorithm~\ref{alg:ipalm}. It holds for any $k\geqslant 0$ that
\vspace{-0.1cm}
\begin{eqnarray}\label{eq:ineq-update-x-2}
\bar\vareps_k\|\vx^{(k+1)}- \vx^*\| &\geqslant  &{\mu}\|\vx^{(k+1)}- \vx^*\|^2 \textstyle + \frac{1}{2\beta_k}\big(\|\vlam^{(k+1)} - \vlam^*\|^2 + \|\vlam^{(k+1)} - \vlam^{(k)}\|^2 - \|\vlam^{(k)} - \vlam^*\|^2\big) \cr
&& \textstyle + \frac{\rho_k}{2}\big(\|\vx^{(k+1)} - \vx^{(k)}\|^2 + \|\vx^{(k+1)} - \vx^{*}\|^2 - \|\vx^{(k)} - \vx^{*}\|^2\big).
\vspace{-0.1cm}
\end{eqnarray}
\end{lemma}
\vspace{-0.2cm}

\begin{proof}
From \eqref{eq:find-x-k-cond}, there exists $\vv^{(k)}\in \partial_\vx \cL_{\beta_k}(\vx^{(k+1)}, \vlam^{(k)}) + \rho_k(\vx^{(k+1)} - \vx^{(k)})$ such that $\|\vv^{(k)}\|\leqslant \bar\vareps_k$, and thus by the $\mu$-strong convexity of $G$, we have
\vspace{-0.1cm}
\begin{eqnarray}\label{eq:ineq-update-x}
\langle \vv^{(k)}, \vx^{(k+1)}- \vx^*\rangle &\geqslant & \textstyle G(\vx^{(k+1)}) - G(\vx^*) + \frac{\mu}{2}\|\vx^{(k+1)}- \vx^*\|^2 + \langle \vA_E^\top\vlam_E^{(k)}, \vx^{(k+1)}- \vx^*\rangle\cr
&& \hspace{-1cm}+ \big\langle \beta_k\vA_E^\top(\vA_E\vx^{(k+1)} - \vb_E), \vx^{(k+1)}- \vx^* \big\rangle \\
&& \hspace{-1cm}+ \big\langle \vA_I^\top[\beta_k(\vA_I\vx^{(k+1)} - \vb_I) + \vlam_I^{(k)}]_+, \vx^{(k+1)}- \vx^* \big\rangle +  \langle \rho_k(\vx^{(k+1)} - \vx^{(k)}), \vx^{(k+1)}- \vx^*\rangle. \nonumber
\vspace{-0.1cm}
\end{eqnarray}
By the Cauchy-Schwarz inequality, it holds $\langle \vv^{(k)}, \vx^{(k+1)}- \vx^*\rangle\leqslant \|\vv^{(k)}\|\cdot \|\vx^{(k+1)}- \vx^*\|\leqslant \bar\vareps_k\|\vx^{(k+1)}- \vx^*\|$. Hence, by the update of $\vlam^{(k+1)}$ and the facts $\vA_E\vx^* =\vb_E$ and  $\vA_I\vx^* \leqslant \vb_I$, we obtain from \eqref{eq:ineq-update-x} that
\vspace{-0.1cm}
\begin{eqnarray}\nonumber
\bar\vareps_k\|\vx^{(k+1)}- \vx^*\| &\geqslant  & \textstyle G(\vx^{(k+1)}) - G(\vx^*) + \frac{\mu}{2}\|\vx^{(k+1)}- \vx^*\|^2 + \langle \vlam_E^{(k+1)}, \vA_E\vx^{(k+1)} - \vb_E\rangle\\\nonumber
&& \textstyle + \langle \vlam_I^{(k+1)}, \vA_I\vx^{(k+1)} - \vb_I\rangle+ \langle \rho_k(\vx^{(k+1)} - \vx^{(k)}), \vx^{(k+1)}- \vx^*\rangle\\\label{eq:ineq-update-x-1}
&=  & \textstyle G(\vx^{(k+1)}) - G(\vx^*) + \frac{\mu}{2}\|\vx^{(k+1)}- \vx^*\|^2 + \langle \vlam_E^{(k+1)}, \vA_E\vx^{(k+1)} - \vb_E\rangle\\\nonumber
&& \textstyle + \langle \vlam_I^{(k+1)}, \vA_I\vx^{(k+1)} - \vb_I\rangle+ \frac{\rho_k}{2}\big(\|\vx^{(k+1)} - \vx^{(k)}\|^2 + \|\vx^{(k+1)} - \vx^{*}\|^2 - \|\vx^{(k)} - \vx^{*}\|^2\big).
\vspace{-0.1cm}
\end{eqnarray}
Using the updating equation of  $\vlam^{(k+1)}$ again, we have 
\vspace{-0.1cm}
\begin{eqnarray}\label{eq:lam-term}
\langle \vlam_E^{(k+1)} - \vlam_E^*, \vA_E\vx^{(k+1)} - \vb_E\rangle  & = &  \textstyle \langle \vlam_E^{(k+1)} - \vlam_E^*, \frac{1}{\beta_k}(\vlam_E^{(k+1)} - \vlam_E^{(k)}) \rangle\cr
& = & \textstyle\frac{1}{2\beta_k}\left(\|\vlam_E^{(k+1)} - \vlam_E^*\|^2 + \|\vlam_E^{(k+1)} - \vlam_E^{(k)}\|^2 - \|\vlam_E^{(k)} - \vlam_E^*\|^2\right),
\vspace{-0.1cm}
\end{eqnarray}
and, by \cite[Lemma~4]{xu2021iteration}, it holds that
\vspace{-0.1cm}
\begin{eqnarray}\label{eq:lam-I-term}
\langle \vlam_I^{(k+1)} - \vlam_I^* , \vA_I\vx^{(k+1)} - \vb_I\rangle & \geqslant &\big\langle \vlam_I^{(k+1)} - \vlam_I^*, \max\big\{\textstyle -\frac{\vlam_I^{(k)}}{\beta_k},\ \vA_I\vx^{(k+1)} - \vb_I\big\}\big\rangle\cr
& = & \textstyle \langle \vlam_I^{(k+1)} - \vlam_I^*, \frac{1}{\beta_k}(\vlam_I^{(k+1)} - \vlam_I^{(k)}) \rangle\cr
& = & \textstyle\frac{1}{2\beta_k}\left(\|\vlam_I^{(k+1)} - \vlam_I^*\|^2 + \|\vlam_I^{(k+1)} - \vlam_I^{(k)}\|^2 - \|\vlam_I^{(k)} - \vlam_I^*\|^2\right).
\vspace{-0.1cm}
\end{eqnarray}

Adding \eqref{eq:lam-term} and \eqref{eq:lam-I-term} to \eqref{eq:ineq-update-x-1} gives
\vspace{-0.1cm}
\begin{eqnarray}\label{eq:all-term-I-E}
\bar\vareps_k\|\vx^{(k+1)}- \vx^*\| &\geqslant  & \textstyle G(\vx^{(k+1)}) - G(\vx^*) + \frac{\mu}{2}\|\vx^{(k+1)}- \vx^*\|^2 + \langle \vlam^*, \vA\vx^{(k+1)} - \vb\rangle\\
&& \textstyle + \frac{\rho_k}{2}\big(\|\vx^{(k+1)} - \vx^{(k)}\|^2 + \|\vx^{(k+1)} - \vx^{*}\|^2 - \|\vx^{(k)} - \vx^{*}\|^2\big)\cr
&& \textstyle + \frac{1}{2\beta_k}\big(\|\vlam^{(k+1)} - \vlam^*\|^2 + \|\vlam^{(k+1)} - \vlam^{(k)}\|^2 - \|\vlam^{(k)} - \vlam^*\|^2\big). \nonumber
\vspace{-0.1cm}
\end{eqnarray}
By the KKT conditions $\vzero\in \partial G(\vx^*) + \vA^\top \vlam^*$ and $\langle \vlam_I^*, \vA_I\vx^* - \vb_I\rangle =0$, it follows that 
\vspace{-0.1cm}
\begin{eqnarray*}
G(\vx^{(k+1)}) - G(\vx^*) + \langle \vlam^*, \vA\vx^{(k+1)} - \vb\rangle & =  & \textstyle G(\vx^{(k+1)}) - G(\vx^*) + \langle \vA^\top\vlam^*, \vx^{(k+1)} - \vx^*\rangle 
 \geqslant   \frac{\mu}{2}\|\vx^{(k+1)}- \vx^*\|^2,
 \vspace{-0.1cm}
\end{eqnarray*}
where the inequality holds from the $\mu$-strong convexity of $G$. Applying this inequality to \eqref{eq:all-term-I-E} gives the desired result.
\end{proof}

\subsection{Outer-iteration complexity}
In this subsection, we assume that \eqref{eq:find-x-k-cond} can be guaranteed. %Then 
We specify the choices of $\{\beta_k\}_{k\geqslant 0}$, $\{\rho_k\}_{k\geqslant 0}$ and $\{\bar\vareps_k\}_{k\geqslant 0}$ and establish the outer-iteration complexity of Algorithm~\ref{alg:ipalm}. To do so, we first show the uniform boundedness of the primal-dual iterates below. %in the following lemma.

\begin{lemma}[Boundedness of primal-dual iterates]\label{lem:bound-pd-iter}
Suppose $\beta_k = \beta_0\sigma^k$ and $\rho_k = \rho_0\sigma^{-k}, \forall\, k\geqslant 0$ for some $\beta_0>0, \rho_0>  0$ and $\sigma >1$ in Algorithm~\ref{alg:ipalm}. It holds for any $k\geqslant 0$ that
\vspace{-0.1cm}
\begin{equation}\label{eq:bd-x-lam}
\textstyle \sqrt{ \beta_0\rho_0\|\vx^{(k+1)} - \vx^*\|^2 + \|\vlam^{(k+1)} - \vlam^*\|^2 } \leqslant \sum_{i=0}^k \frac{2\beta_i\bar\vareps_i}{\sqrt{\beta_0\rho_0}} + \sqrt{\beta_0\rho_0\|\vx^{(0)} - \vx^{*}\|^2 + \|\vlam^{(0)} - \vlam^*\|^2}.
\vspace{-0.1cm}
\end{equation}
%In addition, if $\mu>0$, it holds for any $k\geqslant 0$ that
%\vspace{-0.1cm}
%\begin{equation}\label{eq:ineq-update-x-4}
% \textstyle (\mu\beta_k+\beta_0\rho_0) \|\vx^{(k+1)} - \vx^{*}\|^2 + \|\vlam^{(k+1)} - \vlam^*\|^2 
%\leqslant  \textstyle \sum_{i=0}^k \frac{\beta_i\bar\vareps_i^2}{\mu}  + \beta_0\rho_0 \|\vx^{(0)} - \vx^{*}\|^2 + \|\vlam^{(0)} - \vlam^*\|^2.
%\vspace{-0.1cm}
%\end{equation}
\end{lemma}

\begin{proof}
Multiplying $2\beta_k$ to both sides of \eqref{eq:ineq-update-x-2} gives
\vspace{-0.1cm}
\begin{eqnarray}\label{eq:ineq-update-x-2-1}
2\beta_k\bar\vareps_k\|\vx^{(k+1)}- \vx^*\| &\geqslant  &2\mu\beta_k\|\vx^{(k+1)}- \vx^*\|^2 \textstyle + \big(\|\vlam^{(k+1)} - \vlam^*\|^2 + \|\vlam^{(k+1)} - \vlam^{(k)}\|^2 - \|\vlam^{(k)} - \vlam^*\|^2\big) \cr
&& \textstyle + {\beta_0\rho_0}\big(\|\vx^{(k+1)} - \vx^{(k)}\|^2 + \|\vx^{(k+1)} - \vx^{*}\|^2 - \|\vx^{(k)} - \vx^{*}\|^2\big).
\vspace{-0.1cm}
\end{eqnarray}
Sum up \eqref{eq:ineq-update-x-2-1} to have
\vspace{-0.1cm}
$$\textstyle \beta_0\rho_0\|\vx^{(k+1)} - \vx^*\|^2 + \|\vlam^{(k+1)} - \vlam^*\|^2 \leqslant \sum_{i=0}^k 2\beta_i\bar\vareps_i\|\vx^{(i+1)}- \vx^*\| +  \beta_0\rho_0\|\vx^{(0)} - \vx^{*}\|^2 + \|\vlam^{(0)} - \vlam^*\|^2.
\vspace{-0.1cm}$$
We obtain \eqref{eq:bd-x-lam} by the inequality above and Lemma~\ref{lem:bd-u} with $u_k=\sqrt{\beta_0\rho_0\|\vx^{(k)} - \vx^*\|^2 + \|\vlam^{(k)} - \vlam^*\|^2}$, $\lambda_i=\frac{2\beta_{i-1}\bar\vareps_{i-1}}{\sqrt{\beta_0\rho_0}}$, and $C= \beta_0\rho_0\|\vx^{(0)} - \vx^{*}\|^2 + \|\vlam^{(0)} - \vlam^*\|^2$.
%
%When $\mu>0$, we have $2\beta_k\bar\vareps_k\|\vx^{(k+1)}- \vx^*\|- \mu\beta_k\|\vx^{(k+1)}- \vx^*\|^2 \leqslant \frac{\beta_k\bar\vareps_k^2}{\mu}$. Hence, summing up \eqref{eq:ineq-update-x-2-1} gives
%\vspace{-0.1cm}
%\begin{equation*}\label{eq:ineq-update-x-3}
%\begin{aligned}
%&~ \textstyle \sum_{i=0}^k\mu\beta_k\|\vx^{(i+1)}- \vx^*\|^2 + \|\vlam^{(k+1)} - \vlam^*\|^2 + \beta_0\rho_0 \|\vx^{(k+1)} - \vx^{*}\|^2 \\
%\leqslant &~ \textstyle \sum_{i=0}^k \frac{\beta_i\bar\vareps_i^2}{\mu}  + \beta_0\rho_0 \|\vx^{(0)} - \vx^{*}\|^2 + \|\vlam^{(0)} - \vlam^*\|^2,
%\end{aligned}
%\vspace{-0.1cm}
%\end{equation*}
%which implies the result in \eqref{eq:ineq-update-x-4}.
\end{proof}

\begin{lemma}\label{lem:ineq-a-b}
Let  $\sigma>1$ and $a\in (0,1)$. If $b \geqslant \frac{64}{a^2(\log\sigma)^4} \geqslant 1$,  then $(\log_\sigma b)^2 \leqslant a\cdot b$.
\end{lemma}

\begin{proof}
Let $\theta(x)= \frac{1}{2}(\log x)^2 - x$. Then $\theta'(x) = \frac{1}{x}\log x - 1$.
By the fact $\log x < x$ for all $x > 0$, we have $\theta'(x) < 0, \forall\, x >0$ and thus $\theta$ is a decreasing function. Hence, $\theta(x) \leqslant \theta(1) <0, \forall\, x\geqslant 1$, which implies $(\log_\sigma x^2)^2 \leqslant \frac{8x}{(\log \sigma)^2}, \forall\, x \geqslant 1$. Take $x = \sqrt{b}$. We have $(\log_\sigma b)^2 \leqslant \frac{8\sqrt{b}}{(\log \sigma)^2} \leqslant a\cdot b$, where the second inequality is equivalent to the asssumption that $b \geqslant \frac{64}{a^2(\log\sigma)^4}$.
\end{proof}

By Lemmas~\ref{lem:bound-pd-iter} and \ref{lem:ineq-a-b}, we show below that Algorithm~\ref{alg:ipalm} can produce an $\vareps$-KKT point if $\{\bar\vareps_k\}_{k\geqslant 0}$ are chosen appropriately.

\begin{theorem}\label{thm:outer-iter}
Let $\beta_k$ and $\rho_k$ be defined as in Lemma~\ref{lem:bound-pd-iter}. Let $\bar\vareps=\frac{\vareps(\sigma-1)}{8(\sigma+1)}\min\{1, \sqrt{\beta_0\rho_0}\}$ and choose $\bar\vareps_k = \min\{\bar\vareps, \sqrt{\frac{\rho_0}{20\sigma}}\sigma^{-k}\}, \forall\, k\geqslant 0$ in Algorithm~\ref{alg:ipalm}. Then $\vx^{(K)}$ generated by Algorithm~\ref{alg:ipalm} is an $\vareps$-KKT point of \eqref{eq:aff-prob} with  $\vlam^{(K)}$ being the corresponding multiplier, where 
\vspace{-0.1cm}
\begin{equation}\label{eq:bd-out-iter}
\textstyle K=\max\left\{\left\lceil\log_\sigma\frac{4D_0\sqrt{\rho_0}}{\sqrt{\beta_0}\vareps}\right\rceil,\ \left\lceil\log_\sigma\frac{4D_0}{\beta_0\vareps}\right\rceil,\ \left\lceil\log_\sigma \frac{5(D_0 + \|\vlam^*\|)^2}{\beta_0\vareps}\right\rceil,\ \left\lceil2\log_\sigma \frac{8}{\vareps(\log\sigma)^2} \right\rceil - 1\right\}+1, \text{ and }
\vspace{-0.3cm}
\end{equation}
\vspace{-0.3cm}
\begin{equation}\label{eq:def-D0}
	D_0= \sqrt{\beta_0\rho_0\|\vx^{(0)} - \vx^{*}\|^2 + \|\vlam^{(0)} - \vlam^*\|^2}.
	\vspace{-0.3cm}
\end{equation} 
\end{theorem}

\begin{proof}
Since $\bar\vareps_i \leqslant \bar\vareps$ for $i\geqslant 0$, we have from \eqref{eq:bd-x-lam} that 
\vspace{-0.1cm}
\begin{equation}\label{eq:bd-x-lam-spec}
\textstyle \|\vx^{(k)}-\vx^*\| \leqslant \frac{2\bar\vareps(\sigma^k-1)}{\rho_0(\sigma-1)}+\frac{D_0}{\sqrt{\beta_0\rho_0}}, \quad \|\vlam^{(k)}-\vlam^*\| \leqslant \frac{2\bar\vareps \sqrt{\beta_0}(\sigma^k-1)}{\sqrt{\rho_0}(\sigma-1)}+D_0, ~\forall\,k\geqslant 0.
\vspace{-0.1cm}
\end{equation}
with $D_0$ defined in \eqref{eq:def-D0}. Hence, by the triangle inequality and \eqref{eq:bd-x-lam-spec}, it holds that
\vspace{-0.1cm}
\begin{equation*}%\label{eq:bd-x-lam-spec-2-0}
\textstyle \|\vx^{(k+1)}-\vx^{(k)}\| \leqslant \frac{2\bar\vareps(\sigma^{k+1}+\sigma^k-2)}{\rho_0(\sigma-1)}+\frac{2D_0}{\sqrt{\beta_0\rho_0}},\ \|\vlam^{(k+1)}-\vlam^{(k)}\| \leqslant \frac{2\bar\vareps \sqrt{\beta_0}(\sigma^{k+1}+\sigma^k-2)}{\sqrt{\rho_0}(\sigma-1)}+2D_0,
\vspace{-0.1cm}
\end{equation*}
and thus
\vspace{-0.1cm}
\begin{equation}\label{eq:bd-x-lam-spec-2}
\textstyle \rho_{K-1}\|\vx^{(K)}-\vx^{(K-1)}\| \leqslant \frac{2\bar\vareps(\sigma+1)}{\sigma-1}+\frac{2D_0\sqrt{\rho_0}\sigma^{1-K}}{\sqrt{\beta_0}},\ \frac{1}{\beta_{K-1}}\|\vlam^{(K)}-\vlam^{(K-1)}\| \leqslant \frac{2\bar\vareps (\sigma+1)}{\sqrt{\beta_0\rho_0}(\sigma-1)}+\frac{2D_0}{\beta_0}\sigma^{1-K}.
\vspace{-0.1cm}
\end{equation}
By the choice of $\bar\vareps$ and the definition of $K$ in \eqref{eq:bd-out-iter}, we have from \eqref{eq:bd-x-lam-spec-2} that
\vspace{-0.1cm}
\begin{equation}\label{eq:bd-x-lam-spec-3}
\textstyle \rho_{K-1}\|\vx^{(K)}-\vx^{(K-1)}\| \leqslant \frac{3\vareps}{4}, \quad \frac{1}{\beta_{K-1}}\|\vlam^{(K)}-\vlam^{(K-1)}\| \leqslant \frac{3\vareps}{4}.
\vspace{-0.1cm}
\end{equation}

Additionally,  since $\bar\vareps_i \leqslant \sqrt{\frac{\rho_0}{20\sigma}}\sigma^{-i}$ for $i\geqslant 0$, it is implied by \eqref{eq:bd-x-lam} that
$\|\vlam^{(k)}-\vlam^*\| \leqslant k\sqrt{\frac{\beta_0}{5\sigma}} + D_0$ and thus $\|\vlam^{(k)}\|  \leqslant k\sqrt{\frac{\beta_0}{5\sigma}} + D_0 + \|\vlam^*\|$, which further implies 
$ \|\vlam^{(k)}\|^2 \leqslant \frac{2\beta_0 k^2}{5\sigma} + 2(D_0 + \|\vlam^*\|)^2, \forall\, k\geqslant 0. $
Hence, 
\vspace{-0.1cm}
\begin{equation}\label{eq:bd-lam-sqr-2}
\textstyle \frac{1}{\beta_{K-1}}(\|\vlam^{(K)}\|^2 + \frac{1}{4}\|\vlam^{(K-1)}\|^2) \leqslant \frac{1}{\beta_{0}\sigma^{K-1}}(\frac{\beta_0 K^2}{2\sigma} + \frac{5}{2}(D_0 + \|\vlam^*\|)^2).
\vspace{-0.1cm}
\end{equation}
Since $K-1 \geqslant \log_\sigma \frac{5(D_0 + \|\vlam^*\|)^2}{\beta_0\vareps}$, it holds that $\frac{5}{2\beta_{0}\sigma^{K-1}} (D_0 + \|\vlam^*\|)^2 \leqslant \frac{\vareps}{2}$. Also, since $K\geqslant\left\lceil2\log_\sigma \frac{8}{\vareps(\log\sigma)^2} \right\rceil$, it holds that $\sigma^K \geqslant \frac{64}{\vareps^2(\log\sigma)^4}$, which implies $\frac{K^2}{\sigma^K} \leqslant \vareps$ according to Lemma~\ref{lem:ineq-a-b} with $a=\vareps$ and $b=\sigma^K$. 
Therefore, the right-hand side of \eqref{eq:bd-lam-sqr-2} is no more than $\vareps$, meaning that
\vspace{-0.1cm}
\begin{equation}\label{eq:bd-lam-sqr-3}
\textstyle \frac{1}{\beta_{K-1}}(\|\vlam^{(K)}\|^2 + \frac{1}{4}\|\vlam^{(K-1)}\|^2) \leqslant  \vareps.
\vspace{-0.1cm}
\end{equation}

Now from the updating equations of $\vx^{(k+1)}$ and $\vlam^{(k+1)}$, we have for any $k\geqslant 1$,
\begin{subequations}\label{eq:dist-grad}
\vspace{-0.1cm}
\begin{align}
\textstyle \dist\big(\vzero, \partial G(\vx^{(k)}) + \vA^\top\vlam^{(k)}\big) &\leqslant \bar\vareps_{k-1} + \rho_{k-1}\|\vx^{(k)}-\vx^{(k-1)}\|, \label{eq:dual-feas}\\
\textstyle \sqrt{\|\vA_E\vx^{(k)}-\vb_E\|^2 + \big\|[\vA_I\vx^{(k)}-\vb_I]_+\big\|^2} &\leqslant \textstyle \frac{1}{\beta_{k-1}}\|\vlam^{(k)}-\vlam^{(k-1)}\|, \label{eq:primal-feas}
\vspace{-0.1cm}
\end{align}
and, by \cite[Eqn.(3.18)]{li2021augmented}, we have
\vspace{-0.1cm}
\begin{equation}
\textstyle \|\vlam_I^{(k)} \odot (\vA_I\vx^{(k)}-\vb_I) \| \leqslant \frac{1}{\beta_{k-1}}\left(\|\vlam_I^{(k)}\|^2 + \frac{1}{4}\|\vlam_I^{(k-1)}\|^2\right).
\vspace{-0.1cm}
\end{equation}
\end{subequations}
Note that $\vlam_I^{(k)}\geqslant \vzero$  holds for any $k\geq1$ by the updating equation of $\vlam_I^{(k)}$. Moreover, because of \eqref{eq:bd-x-lam-spec-3} and \eqref{eq:bd-lam-sqr-3} and the fact that $\bar\vareps_k \leqslant \frac{\vareps}{4}$, the three inequalities in \eqref{eq:dist-grad} imply that $(\vx^{(K)}, \vlam^{(K)})$ satisfies the $\vareps$-KKT conditions of \eqref{eq:aff-prob}. 
\end{proof}

\subsection{Overall oracle complexity}
In this subsection, we discuss the details on how to ensure \eqref{eq:find-x-k-cond} and then characterize the total oracle complexity of Algorithm~\ref{alg:ipalm} to produce an $\vareps$-KKT point of \eqref{eq:aff-prob}. Define 
\vspace{-0.1cm}
\begin{eqnarray}\label{eq:def-g-k}
g_k(\vx) &=& \textstyle f(\vx) + \frac{\rho_k}{2}\|\vx-\vx^{(k)}\|^2\\\label{eq:def-h-k}
h_k(\vx) &=& \textstyle \langle\vlam_E^{(k)}, \vA_E\vx-\vb_E\rangle + \frac{\beta_k}{2}\|\vA_E\vx-\vb_E\|^2 + \frac{1}{2\beta_k}\left(\big\|[\beta_k(\vA_I\vx - \vb_I) + \vlam_I^{(k)}]_+\big\|^2 - \|\vlam_I^{(k)}\|^2\right).
\vspace{-0.1cm}
\end{eqnarray}
Then the iPALM subproblem \eqref{eq:Psik} can be written as 
\vspace{-0.1cm}
\begin{equation}\label{eq:prob-psi}
	\min_\vx\left\{\Psi_k(\vx)=g_k(\vx)+h_k(\vx)+r(\vx)\right\},
\vspace{-0.1cm}	
\end{equation} 
which is an instance of \eqref{eq:cvx-composite} with $g=g_k$ and $h=h_k$. This means that \eqref{eq:find-x-k-cond} can be ensured by approximately solving the iPALM subproblem \eqref{eq:prob-psi} using Algorithm~\ref{alg:iAPG}. This way, we can apply the complexity result in Theorem~\ref{thm:oracle-crit-scvx} to establish the oracle complexity for each outer iteration of Algorithm~\ref{alg:ipalm}. 

We adopt the following settings on solving each iPALM subproblem.

\begin{setting}[How to solve iPALM subproblems]\label{set:how-subprob}
%Let $K$ be given in \eqref{eq:bd-out-iter}.  $0 \leqslant k < K$
In iteration $k$ of Algorithm~\ref{alg:ipalm}, Algorithm~\ref{alg:iAPG} is applied to find $\vx^{(k+1)}$ satisfying \eqref{eq:find-x-k-cond}. More precisely,  we compute $\vx^{(k+1)}$ by calling  the \emph{iAPG} method as
\vspace{-0.2cm}
\begin{eqnarray}\label{eq:calliAPG-iPALM}
	\vx^{(k+1)}=\mathrm{iAPG}\left(g_k, h_k, r, \vx^{(k)}, \eta_0, \gamma_0, \mu+\rho_k,\underline{L}, \{\vareps_t\}_{t\geqslant 0}\right),
\vspace{-0.2cm}	
\end{eqnarray}
where $\vareps_t$ is defined as in \eqref{eq:choice-eps-k} for $t\geqslant 1$ and\footnote{Again we take $\underline{L}=\Theta(L_f)$ in order to simplify our results, but our algorithm does not need to know $L_f$.} $\underline{L}=\Theta(L_f)$.
\end{setting}

For simplicity of our analysis, in the setting above,  the values of $\eta_0$, $\gamma_0$, $\gamma_{\mathrm{dec}}$, $\gamma_{\mathrm{inc}}$, $\underline{L} $, and $\vareps_0$ stay the same  across the calls of the iAPG method by different iterations of the iPALM. Also, we use the previous iPALM iterate $\vx^{(k)}$ as the starting point for solving the $k$-th iPALM subproblem.

\begin{setting}[Choice of parameters]\label{set:paras}
Given an $\vareps\in (0,1)$, we choose $\{\beta_k\}$, $\{\rho_k\}$, and $\{\bar\vareps_k\}$ in Algorithm~\ref{alg:ipalm} as the same as those in  
 Theorem~\ref{thm:outer-iter}.
\end{setting}

\vspace{0.2cm}
\noindent\textbf{Notation and some uniform bounds.}~~Under Settings~\ref{set:how-subprob} and \ref{set:paras}, to facilitate our analysis, we first give some notations that are used in this subsection.  Let $K$ be given in \eqref{eq:bd-out-iter}. We define
\begin{equation}\label{eq:def-bar-rho-beta}
\textstyle \underline{\rho}=\rho_{K-1},\ \overline\beta=\beta_{K-1}, \ B_\vx = \frac{2\bar\vareps(\sigma^K-1)}{\rho_0(\sigma-1)}+\frac{D_0}{\sqrt{\beta_0\rho_0}},\ B_\vlam = \frac{2\bar\vareps \sqrt{\beta_0}(\sigma^K-1)}{\sqrt{\rho_0}(\sigma-1)}+D_0, \ \underline{\vareps} = \min\{\bar\vareps, \sqrt{\frac{\rho_0}{20\sigma}}\sigma^{-K}\}
\end{equation}
where $D_0$ is given in \eqref{eq:def-D0}.

In order to apply Theorem~\ref{thm:oracle-crit-scvx} to the iPALM subproblem~\eqref{eq:prob-psi}, we define
\begin{align}\label{eq:def-kappa-D-k}
&\textstyle \kappa^{(k)}=\frac{L_f+ \rho_k}{\gamma_{\mathrm{dec}}(\mu+ \rho_k)}, \quad S^{(k)}=\frac{\sqrt{\kappa^{(k)}}}{2(1-c)^2\underline{L}}\sum_{t=0}^{\infty}\frac{\vareps_0^2}{(t+1)^{2}} < \infty,~\forall\, k < K, \\
%\end{equation}
%\begin{equation}
& \label{eq:def-psi-0-k}
\textstyle \psi_0^{(k)}= \Psi_k(\vx^{(k)}) - \Psi_k^* + (1-(1-c)\alpha_0)\frac{\gamma_0}{2}\|\vx_*^{(k+1)}-\vx^{(k)}\|^2, ~\forall\, k < K.
\end{align}
with
$\vx_*^{(k+1)}=\argmin_\vx \Psi_k(\vx)$ and $\Psi_k^* = \min_\vx \Psi_k(\vx).$
Moreover, define
\begin{align}
&\textstyle L_{\Psi_k} = L_f+\rho_k+\beta_k\|\vA\|^2, \ C_L^{(k)} = \frac{L_{\Psi_k}}{\sqrt{\underline{L}}} + \sqrt{\frac{L_{\Psi_k}}{\gamma_{\mathrm{dec}}}} 
\label{eq:def-T0-Ckappa-ialm-CL}\\
&\textstyle T_0^{(k)}=\sqrt{1+\frac{\rho_k+\beta_k\|\vA\|^2}{L_f}}\log\frac{\frac{L_{\Psi_k}}{L_f} \dist\big(\vzero, \partial \Psi_k(\vx^{(k)})\big)}{\vareps_0}, \label{eq:def-T0-Ckappa-ialm-T}\\
&\textstyle C_\kappa^{(k)} = \frac{\sqrt{\kappa^{(k)} }L_{\Psi_k}(2-c)}{\vareps_0\gamma_{\mathrm{dec}}\sqrt c (1-c)}  \sqrt{\frac{2(\psi_0^{(k)} + S^{(k)})}{\mu+\rho_k}}, \label{eq:def-T0-Ckappa-ialm-C}
\end{align}
where $c\in(0,1)$ is the same universal constant as in \eqref{eq:choice-eps-k}. 
Because $\underline\rho \leqslant \rho_k$ and $\overline\beta \geqslant \beta_k$ for all $0\leqslant k < K$, the quantities defined below are respectively upper bounds of $\kappa^{(k)}$, $S^{(k)}$, $L_{\Psi_k}$, $C_L^{(k)}$, and $C_\kappa^{(k)}$:
\begin{align}
&\textstyle \bar\kappa=\frac{L_f+\underline{\rho}}{\gamma_{\mathrm{dec}}(\mu+\underline{\rho})}, \quad \overline S=\frac{\sqrt{\bar\kappa}}{2(1-c)^2\underline{L}}\sum_{t=0}^{\infty}\frac{\vareps_0^2}{(t+1)^{2}} < \infty, \label{eq:def-kappa-D-bar}\\
&\textstyle 
\overline L_\Psi = L_f+\rho_0+\overline\beta\|\vA\|^2,\ \overline C_L = \frac{\overline L_\Psi}{\sqrt{\underline{L}}} + \sqrt{\frac{\overline L_\Psi}{\gamma_{\mathrm{dec}}}}, \ \overline C_\kappa = \frac{\sqrt{\overline \kappa}(2-c)\overline L_\Psi}{\vareps_0\gamma_{\mathrm{dec}}\sqrt c (1-c)}  \sqrt{\frac{2(\overline\psi_0 + \overline S)}{\mu+\underline \rho} }.  \label{eq:def-C-L-bar}
\end{align}

%Similar to the definition of $X$ in \eqref{eq:def-X-opt}, we define
%\begin{equation}\label{eq:X-k-aff}
%X_k = \Big\{\vx: \frac{\mu+\rho_k}{2}\|\vx - \vx^{(k+1)}_*\|^2 \leqslant \psi_0^{(k)} + S^{(k)}\Big\}
%\end{equation}
%and let $D_{X_k}$ be the diameter of $X_k$. Also, similar to $B_\Phi$ in \eqref{eq:def-B-Phi}, we define
%\begin{equation}\label{eq:def-B-Phi-k}
%B_{\Phi_k}=\max\left\{\big(\vareps_0 + {\textstyle L_g D_{X_k}(1+ \frac{2}{\gamma_{\mathrm{dec}}} }) \big)^2,\  { \textstyle\frac{1}{2\underline{L}} }\dist\big(\vzero, \partial \Psi_k(\vx^{(k)})\big)^2 \right\},
%\end{equation}
%where $\Psi_k$ is given in \eqref{eq:prob-psi}.

By the above notations, we can show the following two lemmas.
\begin{lemma}\label{lem:bd-rho-beta-x-lam}
Suppose Setting~\ref{set:paras} is adopted. It holds that $\rho_k \geqslant \underline{\rho}  $ and $  \beta_k \leqslant \overline\beta$ for all $0\leqslant k <K$. In addition, $\|\vx^{(k)} - \vx^*\|\leqslant B_\vx$ and $\|\vlam^{(k)}-\vlam^*\|\leqslant B_\vlam$ hold for all $0\leqslant k \leqslant K$.  Moreover, $\|\vx^{(k+1)}_*-\vx^*\|\leqslant B_\vx$ for all $0\leqslant k < K$.
\end{lemma}

\begin{proof}
It is trivial to show that $\rho_k \geqslant \underline{\rho}  $ and $  \beta_k \leqslant \overline\beta, \forall\, 0\leqslant k <K$. From \eqref{eq:bd-x-lam-spec} and the definition of $B_\vx$ and $B_\vlam$ in \eqref{eq:def-bar-rho-beta}, we have $\|\vx^{(k)} - \vx^*\|\leqslant B_\vx$ and $\|\vlam^{(k)}-\vlam^*\|\leqslant B_\vlam, \forall\, 0\leqslant k \leqslant K$. Moreover, notice that the first inequality in \eqref{eq:bd-x-lam-spec} also applies to $\vx_*^{(k+1)}$. Hence, we have
\vspace{-0.1cm}
\begin{equation}\label{eq:bd-x-star-k+1}
\textstyle \|\vx^{(k+1)}_*-\vx^*\| \leqslant \frac{2\bar\vareps(\sigma^{k+1}-1)}{\rho_0(\sigma-1)}+\frac{D_0}{\sqrt{\beta_0\rho_0}} \leqslant B_\vx, \forall \, k < K.
\vspace{-0.1cm}
\end{equation}
This completes the proof.
\end{proof}

\begin{lemma}\label{lem:bd-psi-0-k}
Let $\psi_0^{(k)}$ be defined in \eqref{eq:def-psi-0-k}. Then for any $1\leqslant k < K$,
\vspace{-0.1cm}
$$\textstyle \psi_0^{(k)} \leqslant 2B_\vx \big(1 + 2\rho_0 B_\vx + \|\vA\|(2\sigma B_\vlam + B_\vlam + \|\vlam^*\|)\big) + (1-(1-c)\alpha_0)\frac{\gamma_0B_\vx^2}{2}.
\vspace{-0.1cm}$$
\end{lemma}
\vspace{-0.2cm}

\begin{proof}
From \eqref{eq:dual-feas}, %the first inequality in \eqref{eq:dist-grad} and 
the definition of $\Psi_k$,  the fact that $\vlam_I^{(k)}\geqslant \vzero$, and the fact that $\|[\vx+\vy]_+ - \vy\|^2 \le \|[\vx]_+\|^2+\|\vy\|^2, \forall\, \vy\geqslant\vzero$, it follows that
\vspace{-0.1cm}
\begin{eqnarray}
&&\dist\big(\vzero, \partial \Psi_k(\vx^{(k)})\big) \cr
&\leqslant & \bar\vareps_{k-1} + \rho_{k-1}\|\vx^{(k)}-\vx^{(k-1)}\| +\|\vA\|\sqrt{\beta_k^2\|\vA_E \vx^{(k)}-\vb_E\|^2 + \beta_k^2\|[\vA_I \vx^{(k)}-\vb_I]_+\|^2 + \|\vlam_I^{(k)}\|^2} \label{eq:dist-p-Psi-k} \\
&\leqslant & \bar\vareps_{k-1} + 2\rho_0 B_\vx + \|\vA\|(2\sigma B_\vlam + B_\vlam + \|\vlam^*\|), \label{eq:dist-p-Psi-k-2}
\vspace{-0.1cm}
\end{eqnarray}
where in the second inequality, we have used Lemma~\ref{lem:bd-rho-beta-x-lam}, \eqref{eq:primal-feas}, and the fact that $\sqrt{a+b}\leqslant \sqrt{a} + \sqrt{b}, \forall\, a, b\geqslant 0$.
The inequality in \eqref{eq:dist-p-Psi-k-2}, together with the convexity of $\Psi_k$, Lemma~\ref{lem:bd-rho-beta-x-lam}, and \eqref{eq:bd-x-star-k+1}, gives
\vspace{-0.1cm}
\begin{equation*}
\Psi_k(\vx^{(k)}) - \Psi_k^* \leqslant 2B_\vx \big(\bar\vareps_{k-1} + 2\rho_0 B_\vx + \|\vA\|(2\sigma B_\vlam + B_\vlam + \|\vlam^*\|)\big), \forall\, 1\leqslant k < K.
\vspace{-0.1cm}
\end{equation*}
%Hence, by Lemma~\ref{lem:bd-rho-beta-x-lam}, the second equation in \eqref{eq:dist-grad}, and \eqref{eq:bd-x-star-k+1}, we obtain 
%\vspace{-0.1cm}
%\begin{equation*}
%\Psi_k(\vx^{(k)}) - \Psi_k^* \leqslant 2B_\vx \big(\bar\vareps_{k-1} + 2\rho_0 B_\vx + 2\sigma B_\vlam\|\vA\|\big), \forall\, 1\leqslant k < K.
%\vspace{-0.1cm}
%\end{equation*}
Now the desired result follows from that facts that $\bar\vareps_{k-1} \leqslant 1$ and that $\|\vx_*^{(k+1)}-\vx^{(k)}\|^2 \leqslant B_\vx^2$.
\end{proof}

By Lemma~\ref{lem:bd-psi-0-k}, we can bound $\psi_0^{(k)}$ uniformly for $0\leqslant k < K$ by the quantity  
\vspace{-0.1cm}
\begin{equation}\label{eq:def-bar-phi0}
\overline\psi_0 := \max\Big\{\textstyle \psi_0^{(0)}, 2B_\vx \big(1 + 2\rho_0 B_\vx + \|\vA\|(2\sigma B_\vlam + B_\vlam + \|\vlam^*\|)\big) + (1-(1-c)\alpha_0)\frac{\gamma_0B_\vx^2}{2}\Big\}.
\vspace{-0.1cm}
\end{equation}

Now we are ready to show the overall oracle complexity of Algorithm~\ref{alg:ipalm}.

\begin{theorem}[Overall oracle complexity to produce an $\vareps$-KKT point]\label{thm:complexity-kkt}
Suppose Settings~\ref{set:how-subprob} and \ref{set:paras} are adopted. Let $K$ be given in \eqref{eq:bd-out-iter}. In order to produce an $\vareps$-KKT point of \eqref{eq:aff-prob},  Algorithm~\ref{alg:ipalm} needs to make $Q_f$ queries to $(f, \nabla f)$ and $Q_{\vA}$ queries to $(\vA(\cdot), \vA^\top(\cdot))$. For the convex case and the strongly-convex case, the quantities $Q_f$ and $Q_{\vA}$ are given as follows:
\begin{enumerate}[leftmargin=*]
\item[(i)] when $\mu=0$,
\vspace{-0.1cm}
\begin{align}
&\textstyle Q_f= O\left( \Big(K+\sqrt{\frac{L_f}{\rho_0}} \frac{\sigma^{K/2}-1}{\sqrt\sigma-1}\Big)\log\frac{\overline C_L^2(\overline\psi_0 + \overline S)}{\underline\vareps^2}\right),\label{eq:Q-g-mu=0}\\
&Q_{\vA} = O\Bigg( \textstyle \Big(K + \sqrt{\frac{\beta_0\|\vA\|}{L_f}} \frac{\sigma^{K/2}-1}{\sqrt\sigma-1} \Big) \log\frac{\frac{\overline L_\Psi}{ L_f}\overline V_\Psi}{\vareps_0} \label{eq:Q-A-mu=0}\\
&~\hspace{1.5cm}\textstyle + \Big(K+ \sqrt{\frac{L_f}{\rho_0}} \frac{\sigma^{K/2}-1}{\sqrt\sigma-1}+ \frac{\|\vA\|\sqrt{\beta_0}}{\sqrt{\rho_0}} \frac{\sigma^{K}-1}{\sigma-1}\Big) \log\frac{(\overline\psi_0 + \overline S)\overline C_L^2}{\underline\vareps^2} \left(\log \overline C_\kappa + \log\log\frac{\overline C_L^2(\overline \psi_0 + \overline S)}{\underline\vareps^2}\right)\Bigg);\nonumber
\vspace{-0.1cm}
\end{align}
\item[(ii)] when $\mu>0$,
\vspace{-0.1cm}
\begin{align}
&\textstyle Q_f= O\left(K\sqrt{\frac{L_f}{\mu}} \log\frac{\overline C_L^2(\overline\psi_0 + \overline S)}{\underline\vareps^2}\right), \label{eq:Q-g-mu>0}\\
&Q_{\vA} = O\Bigg( \textstyle \Big(K + \sqrt{\frac{\beta_0\|\vA\|}{L_f}} \frac{\sigma^{K/2}-1}{\sqrt\sigma-1} \Big) \log\frac{{\frac{\overline L_\Psi}{L_f}}\overline V_\Psi}{\vareps_0} \label{eq:Q-A-mu>0} \\
&~ \hspace{1.5cm} + \textstyle \Big(K\sqrt{\frac{L_f}{\mu}}+\frac{\|\vA\|\sqrt{\beta_0}}{\sqrt{\mu}} \frac{\sigma^{K/2}-1}{\sqrt\sigma-1}\Big) \log\frac{(\overline\psi_0 + \overline S)\overline C_L^2}{\underline\vareps^2} \left(\log \overline C_\kappa + \log\log\frac{\overline C_L^2(\overline \psi_0 + \overline S)}{\underline\vareps^2}\right)\Bigg). \nonumber
\vspace{-0.1cm}
\end{align}
\end{enumerate} 
The above big-Os hide universal constants depending only on $\gamma_{\mathrm{dec}}$.
\end{theorem}

\begin{proof}
By Theorem~\ref{thm:outer-iter}, we only need to bound the overall number of queries that are made to produce $\vx^{(K)}$.  
From Theorem~\ref{thm:oracle-crit-scvx}, we can find an $\bar\vareps_k$-stationary point of $\Psi_k$ in \eqref{eq:prob-psi} by   Algorithm~\ref{alg:iAPG} with  $Q_f^{(k)}$ queries to $(f, \nabla f)$ and $Q_{\vA}^{(k)}$ queries to $(\vA(\cdot), \vA^\top(\cdot))$, where
\vspace{-0.1cm}
\begin{align}
&\textstyle Q_f^{(k)} = O\left( \left\lceil \log_{\gamma_{\mathrm{dec}}}\frac{\underline{L}}{L_f+\rho_k}\right\rceil \cdot \left\lceil \frac{\log\frac{(C_L^{(k)})^2(\psi_0^{(k)} + S^{(k)})}{\bar\vareps^2_k}} {\log 1/(1-c/\sqrt{\kappa^{(k)}})}\right\rceil\right) = O\left(\sqrt{\kappa^{(k)}} \log\frac{(C_L^{(k)})^2(\psi_0^{(k)} + S^{(k)})}{\underline\vareps^2}\right), \label{eq:Q-g-k}\\
&\textstyle Q_{\vA}^{(k)}=O\left(T_0^{(k)}+\sqrt{\frac{L_f+\rho_k+\beta_k\|\vA\|^2}{\mu+\rho_k}}\log\frac{(C_L^{(k)})^2(\psi_0^{(k)} + S^{(k)})}{\underline\vareps^2} \left(\log C_\kappa^{(k)} + \log\log\frac{(C_L^{(k)})^2(\psi_0^{(k)} + S^{(k)})}{\underline\vareps^2}\right) \right). \label{eq:Q-A-k}
\vspace{-0.1cm}
\end{align}
In the two inequalities above, we have used $\bar\vareps_k \geqslant \underline\vareps$.

When $\mu=0$, we have from \eqref{eq:Q-g-k}, $\psi_0^{(k)} + S^{(k)} \leqslant \overline\psi_0 + \overline S$, and $C_L^{(k)} \leqslant \overline C_L$ that
\vspace{-0.1cm}
\begin{align*}
\textstyle Q_f = &~ \textstyle  \sum_{k=0}^{K-1} Q_f^{(k)} = \sum_{k=0}^{K-1} O\left(\sqrt{\kappa^{(k)}} \log\frac{\overline C_L^2(\overline\psi_0 + \overline S)}{\underline\vareps^2}\right)\overset{\eqref{eq:def-kappa-D-k}} = \sum_{k=0}^{K-1} O\left(\sqrt{1+\frac{L_f\sigma^k}{\rho_0}} \log\frac{\overline C_L^2(\overline\psi_0 + \overline S)}{\underline\vareps^2}\right),%\\
%= & ~ O\left( \Big(K+\sqrt{\frac{L_f}{\rho_0}} \frac{\sigma^{K/2}-1}{\sqrt\sigma-1}\Big)\log\frac{\overline C_L^2(\overline\psi_0 + \overline S)}{\bar\vareps^2}\right).
\vspace{-0.1cm}
\end{align*}
which gives \eqref{eq:Q-g-mu=0} by $\sum_{k=0}^{K-1}\sqrt{1+\frac{L_f\sigma^k}{\rho_0}} \leqslant K+\sqrt{\frac{L_f}{\rho_0}} \frac{\sigma^{K/2}-1}{\sqrt\sigma-1}$.
 Also, it follows from \eqref{eq:dist-p-Psi-k-2} %, \eqref{eq:primal-feas}, and Lemma~\ref{lem:bd-rho-beta-x-lam} 
 that 
 \vspace{-0.1cm}
$$\dist\big(\vzero, \partial \Psi_k(\vx^{(k)})\big)\leqslant \max\left\{\dist\big(\vzero, \partial \cL_{\beta_0}(\vx^{(0)}, \vlam^{(0)})\big),\ \bar\vareps+2\rho_0 B_\vx + \|\vA\|(B_\vlam(2\sigma +1) + \|\vlam^*\|)\right\} =: \overline V_\Psi, \forall\, k\geqslant 0.
\vspace{-0.1cm}$$
 Thus, we have from \eqref{eq:Q-A-k} and by $C_L^{(k)} \leqslant \overline C_L$, $\kappa^{(k)}\leqslant \bar\kappa $, and $L_f + \rho_k+\beta_k\|\vA\|^2 \leqslant \overline L_\Psi$ that
 \vspace{-0.1cm}
\begin{align*}
&~Q_{\vA} =  \textstyle \sum_{k=0}^{K-1} Q_{\vA}^{(k)} \\
= & ~ \sum_{k=0}^{K-1} O\Big( \textstyle \sqrt{1+\frac{\rho_k+\beta_k\|\vA\|^2}{L_f}}\log\frac{{\frac{\overline L_\Psi}{L_f}}\overline V_\Psi}{\vareps_0} 
+\sqrt{\frac{L_f+\rho_k+\beta_k\|\vA\|^2}{\rho_k}}\log\frac{\overline C_L^2(\overline \psi_0 + \overline S)}{\underline\vareps^2} \big(\log \overline C_\kappa + \log\log\frac{\overline C_L^2(\overline \psi_0 + \overline S)}{\underline\vareps^2}\big)\Big)\\
= &~ O\Bigg( \textstyle \Big(K + \sqrt{\frac{\beta_0\|\vA\|}{L_f}} \frac{\sigma^{K/2}-1}{\sqrt\sigma-1} \Big) \log\frac{{\frac{\overline L_\Psi}{ L_f}}\overline V_\Psi}{\vareps_0} \\
&~\hspace{0.5cm}\textstyle + \Big(K+ \sqrt{\frac{L_f}{\rho_0}} \frac{\sigma^{K/2}-1}{\sqrt\sigma-1}+ \frac{\|\vA\|\sqrt{\beta_0}}{\sqrt{\rho_0}} \frac{\sigma^{K}-1}{\sigma-1}\Big) \log\frac{(\overline\psi_0 + \overline S)\overline C_L^2}{\underline\vareps^2} \left(\log \overline C_\kappa + \log\log\frac{\overline C_L^2(\overline \psi_0 + \overline S)}{\underline\vareps^2}\right)\Bigg).
\vspace{-0.1cm}
\end{align*}
%where we have used $\sum_{k=1}^{K-1} k\sigma^k = \frac{1}{\sigma-1}\left((K-1)\sigma^K-\frac{\sigma^K-\sigma}{\sigma-1}\right)$.
This proves the case of $\mu=0$.

When $\mu>0$, we have
\vspace{-0.1cm}
\begin{align*}
\textstyle Q_f = &~ \textstyle \sum_{k=0}^{K-1} Q_f^{(k)} = \sum_{k=0}^{K-1} \textstyle O\left(\sqrt{\kappa^{(k)}} \log\frac{\overline C_L^2(\overline\psi_0 + \overline S)}{\underline\vareps^2}\right) = O\left(K\sqrt{\bar\kappa} \log\frac{\overline C_L^2(\overline\psi_0 + \overline S)}{\underline\vareps^2}\right) = O\left(K\sqrt{\frac{L_f}{\mu}} \log\frac{\overline C_L^2(\overline\psi_0 + \overline S)}{\underline\vareps^2}\right),
\vspace{-0.1cm}
\end{align*}
which gives \eqref{eq:Q-g-mu>0}. Also, we have 
\vspace{-0.1cm}
\begin{align*}
 &~ Q_{\vA} = \textstyle \sum_{k=0}^{K-1} Q_{\vA}^{(k)} \\
  = &~ \sum_{k=0}^{K-1} O\Big( \textstyle \sqrt{1+\frac{\rho_k+\beta_k\|\vA\|^2}{L_f}}\log\frac{{\frac{\overline L_\Psi}{L_f}}\overline V_\Psi}{\vareps_0} 
+\sqrt{\frac{L_f+\rho_k+\beta_k\|\vA\|^2}{\mu+\rho_k}}\log\frac{\overline C_L^2(\overline \psi_0 + \overline S)}{\underline\vareps^2} \big(\log \overline C_\kappa + \log\log\frac{\overline C_L^2(\overline \psi_0 + \overline S)}{\underline\vareps^2}\big)\Big)\\
= &~ O\Bigg( \textstyle \Big(K + \sqrt{\frac{\beta_0\|\vA\|}{L_f}} \frac{\sigma^{K/2}-1}{\sqrt\sigma-1} \Big) \log\frac{{\frac{\overline L_\Psi}{L_f}}\overline V_\Psi}{\vareps_0}  \\
&~ \hspace{0.8cm} + \textstyle \Big(K\sqrt{\frac{L_f}{\mu}}+\frac{\|\vA\|\sqrt{\beta_0}}{\sqrt{\mu}} \frac{\sigma^{K/2}-1}{\sqrt\sigma-1}\Big) \log\frac{(\overline\psi_0 + \overline S)\overline C_L^2}{\underline\vareps^2} \left(\log \overline C_\kappa + \log\log\frac{\overline C_L^2(\overline \psi_0 + \overline S)}{\underline\vareps^2}\right)\Bigg).
\vspace{-0.1cm}
\end{align*}
This proves the case of $\mu>0$ and completes the proof.
\end{proof}

\begin{remark}\label{rm:affine-cstr}
Notice that $K=O(\log\frac{1}{\vareps})$ by \eqref{eq:bd-out-iter} and that $\sigma^K=O(\frac{1}{\vareps})$ and $\underline\vareps = \Theta(\vareps)$. Hence, from Theorem~\ref{thm:complexity-kkt}, we have that $Q_f=O\left(\sqrt{\frac{L_f}{\vareps}} \log\frac{1}{\vareps}\right)$ and $Q_{\vA}=O\left( \big(\sqrt{\frac{L_f}{\vareps}} + \frac{\|\vA\|}{\vareps}\big) \big(\log\frac{1}{\vareps}\big)^2\right)$ for the case of $\mu=0$, and $Q_f=O\left(\sqrt{\frac{L_f}{\mu}} \big(\log\frac{1}{\vareps}\big)^2\right)$ and $Q_{\vA}=O\left(\big(\log\frac{1}{\vareps}\sqrt{\frac{L_f}{\mu}} + \frac{\|\vA\|}{\sqrt{\mu\vareps}}\big)\log\frac{1}{\vareps}\right)$ for the case of $\mu>0$. If $\rho_0=O(\vareps)$ and $\beta_0=O(\frac{1}{\vareps})$, then $K=O(1)$. For this setting, the factor $\big(\log\frac{1}{\vareps}\big)^2$ will reduce to $\log\frac{1}{\vareps}$ for $Q_{\vA}$ in the case of $\mu=0$ and for $Q_f$ in the case of $\mu>0$.
\end{remark}

\section{Smoothed bilinear saddle-point structured optimization}
In this section, we consider the bilinear saddle-point structured optimization problem
\vspace{-0.2cm}
\begin{equation}\label{eq:sp-prob}
p^*=\min_{\vx\in\RR^n} \left\{ p(\vx) := f(\vx) + r(\vx) + \max_{\vy\in \RR^m} \big\{\langle \vy, \vA\vx\rangle - \phi(\vy) \big\} \right\},
\vspace{-0.2cm}
\end{equation}
where $\vA\in\RR^{m\times n}$, $f$ is a $L_f$-smooth and convex function, and $r$ and $\phi$ are closed convex functions that admit easy proximal mappings. We assume that $(f, \nabla f)$ is significantly more expensive than $(\vA(\cdot), \vA^\top(\cdot))$ to evaluate. We adopt the following notation in this section
\begin{subequations}\label{eq:def-Hx}
\vspace{-0.1cm}
\begin{align}
&G(\vx) := f(\vx) + r(\vx), \quad \bar h(\vx):=\max_{\vy\in \RR^m} \big\{\langle \vy, \vA\vx\rangle - \phi(\vy) \big\}, \\
&\varphi(\vy):=\min_{\vx\in\RR^n} \big\{ G(\vx) + \langle \vy, \vA\vx\rangle\big\}, \quad d(\vy) :=  \varphi(\vy) - \phi(\vy).
\vspace{-0.1cm}
\end{align}
\end{subequations}
We call $p(\vx)-d(\vy)$ the \emph{duality gap} at $(\vx,\vy)$ which is always non-negative by the definition of $p$ and $d$. A pair of points $(\vx^*, \vy^*)$ that satisfies $p(\vx^*) = d(\vy^*)$, or equivalently, 
$
\vzero\in \partial G(\vx^*) + \vA^\top\vy^*, \vzero\in \vA\vx^* - \partial \phi(\vy^*).
$
is called a \emph{saddle point} of \eqref{eq:sp-prob}. We make the following assumption on \eqref{eq:sp-prob}.
\begin{assumption}\label{assump:sp-prob}
	Function  $f$ is $L_f$-smooth and $\mu$-strongly convex with $\mu>0$; $\dom(\phi)$ is bounded, i.e., $D_\phi := \max_{\vy_1, \vy_2\in \dom(\phi)}\|\vy_1 -\vy_2\| < \infty$; \eqref{eq:sp-prob} has a saddle point $(\vx^*, \vy^*)$.
\end{assumption}

Our goal in this section is to find a point $(\bar\vx, \bar\vy)$ that satisfies the condition above with $\vareps$-precision. We call it an $\vareps$-saddle point of \eqref{eq:sp-prob} defined formally below.
\begin{definition}
For any $\vareps\geqslant 0$, a point $(\bar\vx, \bar\vy)$ is called an $\vareps$-saddle point of \eqref{eq:sp-prob} if
\begin{equation}\label{eq:eps-sp-cond}
\dist\big(\vzero, \partial G(\bar\vx) + \vA^\top\bar\vy\big)\leqslant \vareps, \quad \dist\big(\vzero,  \vA\bar\vx - \partial \phi(\bar\vy)\big)\leqslant \vareps.
\end{equation}
\end{definition}
We consider finding  an $\vareps$-saddle point of \eqref{eq:sp-prob} by applying the iAPG method to an smooth approximation of \eqref{eq:sp-prob} using Nesterov's smoothing technique. 

The following result shows the duality gap of an $\vareps$-saddle point of \eqref{eq:sp-prob}.
\begin{theorem}\label{thm:sp-pd}
Suppose Assumption~\ref{assump:sp-prob} holds. If $(\bar\vx, \bar\vy)$ is an $\vareps$-saddle point of \eqref{eq:sp-prob}, then $p(\bar\vx) - d(\bar\vy) \leqslant 2\vareps D_\phi  + \frac{3\vareps^2}{2\mu}$.
\end{theorem}

\begin{proof}
Since $(\bar\vx, \bar\vy)$ is an $\vareps$-saddle point, there exist $\bar\vu\in \partial G(\bar\vx) + \vA^\top\bar\vy$ and $\bar\vv\in \vA\bar\vx - \partial \phi(\bar\vy)$ such that $\|\bar\vu\|\leqslant \vareps$ and $\|\bar\vv\|\leqslant \vareps$. By the $\mu$-strong convexity of $G$, it follows that
\vspace{-0.1cm}
\begin{align}\label{eq:bd-G-bar-x}
G(\bar\vx) \leqslant & ~ \textstyle G(\vx^*) + \langle \bar\vu - \vA^\top\bar\vy, \bar\vx - \vx^*\rangle - \frac{\mu}{2}\|\bar\vx - \vx^*\|^2 \cr
= & ~ \textstyle G(\vx^*) + \langle \bar\vu, \bar\vx - \vx^*\rangle - \langle \bar\vy, \vA\bar\vx - \vA\vx^*\rangle - \frac{\mu}{2}\|\bar\vx - \vx^*\|^2\cr
\leqslant & ~ \textstyle G(\vx^*) - \langle \bar\vy, \vA\bar\vx - \vA\vx^*\rangle + \frac{1}{2\mu}\|\bar\vu\|^2,
\vspace{-0.1cm}
\end{align}
where we have used the Young's inequality in the last inequality. In addition, by the convexity of $\phi$ and the definition of $\bar h$ in \eqref{eq:def-Hx}, we have
$ \bar h(\bar\vx) + \langle \bar\vv, \bar\vy - \widehat\vy\rangle \leqslant \langle \bar\vy, \vA\bar\vx \rangle - \phi(\bar\vy)$, 
where $\widehat\vy \in \argmax_\vy \big\{\langle \vy, \vA\bar\vx\rangle - \phi(\vy) \big\}$. Adding this inequality to \eqref{eq:bd-G-bar-x} gives
\begin{align}\label{eq:bd-f-bar-x}
p(\bar\vx) + \langle \bar\vv, \bar\vy - \widehat\vy\rangle \leqslant & ~ \textstyle G(\vx^*) + \langle \bar\vy, \vA\vx^*\rangle - \phi(\bar\vy) + \frac{1}{2\mu}\|\bar\vu\|^2\cr
= & ~ \textstyle p(\vx^*) + \langle \bar\vy - \vy^*, \vA\vx^*\rangle + \phi(\vy^*) - \phi(\bar\vy) + \frac{1}{2\mu}\|\bar\vu\|^2,
\end{align}
where the equality holds because $(\vx^*,\vy^*)$ is a saddle point of \eqref{eq:sp-prob}. Now from the convexity of $\phi$ and the fact $\vA\vx^*\in \partial \phi(\vy^*)$, it follows that $\langle \bar\vy - \vy^*, \vA\vx^*\rangle + \phi(\vy^*) - \phi(\bar\vy) \leqslant 0$. Hence, we have from \eqref{eq:bd-f-bar-x} that 
\begin{align}\label{eq:bd-f-bar-x-2}
p(\bar\vx) \leqslant \textstyle p(\vx^*) - \langle \bar\vv, \bar\vy - \widehat\vy\rangle + \frac{1}{2\mu}\|\bar\vu\|^2 \leqslant p(\vx^*) + \vareps D_\phi  + \frac{\vareps^2}{2\mu}.
\end{align}

Similarly, from the convexity of $\phi$ and $\bar\vv\in \vA\bar\vx - \partial \phi(\bar\vy)$, it follows that
\begin{align}\label{eq:bd-phi-bar-y}
-\phi(\bar\vy) \geqslant -\phi(\vy^*) + \langle \bar\vv - \vA\bar\vx, \bar\vy - \vy^*\rangle.
\end{align}
In addition, by the definition of $\varphi$ in \eqref{eq:def-Hx} and the fact $\bar\vu\in \partial G(\bar\vx) + \vA^\top\bar\vy$, we have
$\varphi(\bar\vy) + \langle \bar\vu, \bar\vx - \widehat\vx\rangle \geqslant G(\bar\vx) + \langle \bar\vy, \vA\bar\vx\rangle$, 
where $\widehat\vx=\argmin_\vx \big\{ G(\vx) + \langle \bar\vy, \vA\vx\rangle\big\}$. Adding this inequality to \eqref{eq:bd-phi-bar-y} yields
\begin{align}\label{eq:bd-d-bar-y}
d(\bar\vy) + \langle \bar\vu, \bar\vx - \widehat\vx\rangle \geqslant & ~ G(\bar\vx) + \langle \vy^*, \vA\bar\vx\rangle -\phi(\vy^*) + \langle \bar\vv, \bar\vy - \vy^*\rangle\cr
= & ~ p^* - G(\vx^*) + G(\bar\vx) +  \langle \vy^*, \vA\bar\vx - \vA\bar\vx^*\rangle  + \langle \bar\vv, \bar\vy - \vy^*\rangle.
\end{align}
Notice that $-\vA^\top \vy^* \in \partial G(\vx^*)$. We have from the convexity of $G$ that $- G(\vx^*) + G(\bar\vx) +  \langle \vy^*, \vA\bar\vx - \vA\bar\vx^*\rangle \geqslant 0$. Hence, \eqref{eq:bd-d-bar-y} implies
\vspace{-0.2cm}
\begin{align}\label{eq:bd-d-bar-y-2}
d(\bar\vy) + \langle \bar\vu, \bar\vx - \widehat\vx \rangle \geqslant p^* + \langle \bar\vv, \bar\vy - \vy^*\rangle \geqslant p^* - \vareps D_y.
\vspace{-0.2cm}
\end{align}
Moreover, from $\bar\vu\in \partial G(\bar\vx) + \vA^\top\bar\vy$ and $\vzero \in \partial G(\widehat\vx) + \vA^\top\bar\vy$ together with the $\mu$-strong convexity of $G$, it holds $\langle \bar\vu, \bar\vx - \widehat\vx \rangle \geqslant \mu \|\bar\vx - \widehat\vx\|^2$. Hence, by the Cauchy-Schwarz ineuqality, we have $\|\bar\vx - \widehat\vx\| \leqslant \frac{\|\bar\vu\|}{\mu}$ and $\langle \bar\vu, \bar\vx - \widehat\vx \rangle \leqslant \frac{\|\bar\vu\|^2}{\mu} \leqslant \frac{\vareps^2}{\mu}$, which together with \eqref{eq:bd-d-bar-y-2} gives $d(\bar\vy) \geqslant p^* - \vareps D_\phi - \frac{\vareps^2}{\mu}$. Therefore, from \eqref{eq:bd-f-bar-x-2}, we conclude that $p(\bar\vx) - d(\bar\vy) \leqslant 2\vareps D_\phi  + \frac{3\vareps^2}{2\mu}$. This completes the proof.
\end{proof}

\vspace{-0.2cm}
\begin{remark}
By Theorem~\ref{thm:sp-pd}, in order to produce a primal-dual solution of  \eqref{eq:sp-prob} that achieves a duality gap at most $\vareps>0$, it suffices to find an $\bar\vareps$-saddle point, where $\bar\vareps = \min\big\{\frac{\vareps}{4 D_\phi}, \sqrt{\frac{4\mu\vareps}{3}}\big\}$. The advantage of targeting a near-saddle point is the direct verifiability of the conditions in \eqref{eq:eps-sp-cond} while the duality gap cannot be directly computed.
\end{remark}

%When $\phi$ is strongly convex, $H$ will be smooth. However, 
Notice that when $\phi$ is merely convex (as compared to strongly convex), $\bar h$ is in general non-smooth. In this case, \cite{nesterov2005smooth} introduces a smoothing technique and solves an approximation of \eqref{eq:sp-prob} as follows: 
\vspace{-0.2cm}
\begin{equation}\label{eq:sp-prob-smooth}
p^*_\rho=\min_{\vx\in\RR^n} \left\{p_\rho(\vx):= f(\vx) + r(\vx) + h_\rho(\vx)\right\},
\vspace{-0.2cm}
\end{equation}
where $\rho>0$ is the smoothing parameter, and $h_\rho$ is defined by
\vspace{-0.2cm}
\begin{equation}\label{eq:def-Hx-rho}
h_\rho(\vx) = \max_{\vy\in \RR^m} \big\{\textstyle \langle \vy, \vA\vx\rangle - \phi(\vy) - \frac{\rho}{2}\|\vy-\vy^{(0)}\|^2 \big\} 
\vspace{-0.2cm}
\end{equation}
with any $\vy^{(0)}\in \dom(\phi)$. The result below is from \cite[Thm.~1]{nesterov2005smooth}. 

\begin{lemma}\label{eq:smooth-H-rho}
$h_\rho$ defined in \eqref{eq:def-Hx-rho} is $\frac{\|\vA\|^2}{\rho}$-smooth and $\nabla h_\rho(\vx) = \vA^\top\vy(\vx)$, where for any $\vx$,
\vspace{-0.2cm}
\begin{equation}\label{eq:def-yofx-rho}
\vy(\vx) = \argmax_{\vy\in \RR^m} \big\{\textstyle \langle \vy, \vA\vx\rangle - \phi(\vy) - \frac{\rho}{2}\|\vy-\vy^{(0)}\|^2 \big\} = \prox_{\phi/\rho}\left(\vy^{(0)} + \frac{1}{\rho}\vA\vx\right). 
\vspace{-0.2cm}
\end{equation}
\end{lemma}
Lemma~\ref{eq:smooth-H-rho} implies that \eqref{eq:sp-prob-smooth} is an instance of \eqref{eq:cvx-composite} with 
$g=f$ and $h=h_\rho$. This means that \eqref{eq:sp-prob-smooth} can be approximately solved by Algorithm~\ref{alg:iAPG}. More precisely,  we compute an $\vareps$-stationary point of \eqref{eq:sp-prob-smooth} by calling the iAPG method as
\vspace{-0.2cm}
\begin{eqnarray}\label{eq:calliAPG-smooth}
\bar\vx=\mathrm{iAPG}\left(f, h_\rho, r, \vx^{(0)}, \eta_0, \gamma_0, \mu,\underline{L}, \{\vareps_k\}_{k\geqslant 0}\right),
\vspace{-0.3cm}
\end{eqnarray}
where $\vareps_k$ is defined as in \eqref{eq:choice-eps-k} for $k\geqslant 0$, $\vx^{(0)}$ is any point in $\dom(r)$, and $\underline{L}=\Theta(L_f)$.

By Lemma~\ref{eq:smooth-H-rho} and Theorem~\ref{thm:oracle-crit-scvx}, we obtain the following complexity of finding an $\vareps$-saddle point of \eqref{eq:sp-prob}.
\begin{theorem}[Overall oracle complexity to produce an $\vareps$-saddle point]\label{thm:complexity-eps-sp}
Suppose Assumption~\ref{assump:sp-prob} holds. Given any $\vareps>0$, let $\rho = \frac{\vareps}{D_\phi}$ in \eqref{eq:sp-prob-smooth} and choose any $\vy^{(0)}\in \dom(\phi)$ in \eqref{eq:def-Hx-rho}. Suppose $\bar\vx$ is an $\vareps$-stationary point of \eqref{eq:sp-prob-smooth} found by applying Algorithm~\ref{alg:iAPG} to \eqref{eq:sp-prob-smooth} with the inputs given in \eqref{eq:calliAPG-smooth} and the optional steps enabled. In addition, let $\bar\vy=\vy(\bar\vx)$, where $\vy(\cdot)$ is defined in \eqref{eq:def-yofx-rho}. Then $(\bar\vx, \bar\vy)$ is an $\vareps$-saddle point of \eqref{eq:sp-prob}. To produce $(\bar\vx, \bar\vy)$, at most $K_{\mathrm{sp}}$ queries to $(f, \nabla f)$ and $T_{\mathrm{sp}}$ queries to $\big(\vA(\cdot), \vA^\top(\cdot)\big)$ are needed, where
\vspace{-0.2cm}
\begin{eqnarray}
&&\textstyle K_{\mathrm{sp}} = O\left(\left\lceil \log_{\gamma_{\mathrm{dec}}}\frac{\underline{L}}{L_f}\right\rceil \cdot \left\lceil \frac{\log\frac{C_L^2(\psi_0 + S_f)}{\vareps^2}} {\log 1/(1-c/\sqrt\kappa_f)}\right\rceil\right) = O\left(\sqrt\kappa_f\log\frac{C_L^2(\psi_0 + S_f)}{\vareps^2}\right), \label{eq:up-bd-K-sp}\\
&&\textstyle T_{\mathrm{sp}}=O\left(T_0+\sqrt{\frac{L_f+\frac{D_\phi\|\vA\|^2}{\vareps}}{\mu}}\log\frac{C_L^2(\psi_0 + S_f)}{\vareps^2} \left(\log C_{\kappa_f} + \log\log\frac{C_L^2(\psi_0 + S_f)}{\vareps^2}\right) \right). \label{eq:comp-crit-sp}
\vspace{-0.3cm}
\end{eqnarray}
Here, $c\in(0,1)$, $\kappa_f=\frac{L_f}{\gamma_\mathrm{dec}\mu}$, $S_f$ is defined the same as $S$ in \eqref{eq:D-delta} except that $\kappa$ is replaced by $\kappa_f$, 
\vspace{-0.2cm}
$$\textstyle \psi_0 = p_\rho(\vx^{(0)})-p_\rho^*+(1-(1-c)\alpha_{0})\frac{\gamma_{0}}{2}\|\vx_\rho^*-\vx^{(0)}\|^2,\quad \textstyle C_L= \frac{L_f+\frac{D_\phi \|\vA\|^2}{\vareps}}{\sqrt{\underline{L}}} + \sqrt{\frac{L_f+\frac{D_\phi \|\vA\|^2}{\vareps}}{\gamma_{\mathrm{dec}}}} \vspace{-0.2cm}$$ 
with $\vx_\rho^*=\argmin_\vx p_\rho(\vx)$, and 
\vspace{-0.2cm}
\begin{equation}\label{eq:def-T0-Ckappa-sp}
\textstyle T_0=\sqrt{1+\frac{D_\phi\|\vA\|^2}{L_f \vareps}}\log\frac{\big(1+\frac{D_\phi \|\vA\|^2}{L_f \vareps}\big) \dist\big(\vzero, \partial p_\rho(\vx^{(0)})\big)}{\vareps_0},\ \textstyle C_{\kappa_f} = \frac{\sqrt{\kappa_f }(2-c)(L_f+\frac{D_\phi \|\vA\|^2}{\vareps})}{\vareps_0\gamma_{\mathrm{dec}}\sqrt c (1-c)}  \sqrt{\frac{2(\psi_0 + S_f)}{\mu}}.
\vspace{-0.2cm}
\end{equation}
\end{theorem}

\begin{proof}
First, suppose that $\bar\vx$ is an $\vareps$-stationary point of $p_\rho$, i.e., $\dist\big(\vzero, \partial p_\rho(\bar\vx) \big) \leqslant \vareps$. Let $\bar\vy=\vy(\bar\vx)$. Then by Lemma~\ref{eq:smooth-H-rho}, we have $\dist\big(\vzero, \nabla g(\bar\vx) + \partial r(\bar\vx) + \vA^\top\bar\vy\big) \leqslant \vareps$. In addition, notice $\vzero\in \vA\bar\vx - \partial \phi(\bar\vy) - \rho(\bar\vy - \vy^{(0)})$, and thus $\dist\big(\vzero, \vA\bar\vx - \partial \phi(\bar\vy)\big) \leqslant \rho\|\bar\vy - \vy^{(0)}\|\leqslant \rho D_\phi = \vareps$, where the equality follows from $\rho = \frac{\vareps}{D_\phi}$. Therefore, $(\bar\vx, \bar\vy)$ is an $\vareps$-saddle point of \eqref{eq:sp-prob}.

Second, by Lemma~\ref{eq:smooth-H-rho}, querying $\nabla h_\rho$ once needs one call to $ \big(\vA(\cdot), \vA^\top(\cdot), \prox_{\phi/\rho}\big)$. Now noticing that the smoothness constant of $h_\rho$ is $\frac{\|\vA\|^2}{\rho}= \frac{D_\phi\|\vA\|^2}{\vareps}$, we obtain the bounds on $K_{\mathrm{sp}}$ and $T_{\mathrm{sp}}$ directly from Theorem~\ref{thm:oracle-crit-scvx}. This completes the proof.
\end{proof}

\section{Experimental results}\label{sec:numerical}
In this section, we conduct numerical experiments to demonstrate the practical performance of the proposed algorithms. All the tests were conducted with MATLAB 2021a on a Windows machine with 10 CPU cores and 128 GB memory.

\subsection{Multitask learning}
In this subsection, we test the proposed iAPG in Alg.~\ref{alg:iAPG} on the multitask learning \cite{evgeniou2005learning} and compare it to the exact counterpart. Given $m$ binary-class datasets $\cD_l=\{(\vx_{l,i}, y_{l,i})\}_{i=1}^{N_l}, l=1,\ldots,m$ with $\vx_{l, i}\in\RR^n$ and the corresponding label $y_{l,i}\in\{+1,-1\}$ for each $l$ and $i$, we solve the multitask logistic regression \cite{gu2014multitask} and use a regularizer given in \cite[Eqn. (23)]{evgeniou2005learning} together with an $\ell_1$ term:  
\vspace{-0.2cm}
\begin{equation}\label{eq:mult-learn}
\min_\vW F(\vW):=  \underbrace{\sum_{l=1}^m \frac{1}{N_l}  \sum_{i=1}^{N_l} \log\big( 1+\exp( -y_{l,i}\vw_l^\top \vx_{l,i} ) \big) +  \frac{\mu}{2} \|\vW\|_F^2}_{g(\vW)} + \underbrace{\frac{\lambda_1}{2} \| {\textstyle \vW-\frac{1}{m}  \vW\vone\vone^\top }\|_F^2}_{h(\vW)}  + \underbrace{\lambda_2\|\vW\|_1}_{r(\vW)},
\vspace{-0.2cm}
\end{equation}
where $\|\vW\|_1:=\sum_{i,j}|w_{i,j}|$, and $\vw_l$, as the $l$-th column of $\vW$, is the classifier parameter for the $l$-th task. 

In the experiments, we fixed $\lambda_2=10^{-3}$ and chose $\mu\in\{0.01, 0.1\}$ and $\lambda_1\in \{1, 10, 100\}$. 
Notice that a larger value of $\lambda_1$ leads to a stronger correlation between the $m$ classifiers and a larger smoothness constant of $h$.
We randomly generated $m=4$ binary-class datasets as in \cite{xu2016proximal}. For each task $l=1,\ldots,m$, every positive sample follows the Gaussian distribution $\cN(\vmu_{l}, \vSigma)$ and negative sample following $\cN(-\vmu_l, \vSigma)$ with
\vspace{-0.1cm} 
$$\vSigma=\left[\begin{array}{cc}
\rho \vone_{s\times s} + (1-\rho)\vI _{s\times s} & \vzero_{s\times (n-s)} \\
\vzero_{(n-s) \times s} & \vI_{(n-s)\times (n-s)}
\end{array}\right],
\quad
\vmu_{l} = \left[\begin{array}{c}
\vone_s \\
\vzero_{n-s}
\end{array}\right] + \vd_l
\vspace{-0.1cm}
$$
where the entries of $\vd_l$ follow the uniform distribution on $[\frac{1}{2}, 1]$. We set $n=200, N_l=500,\forall\, l$ or $n=2000, N_l=5000, \forall\, l$. For each combination of $(\mu, \lambda_1, n, N_l)$, we conducted 10 independent trials. Since the smoothness constants of $g$ and $h$ can be computed explicitly, we also tested the methods without line search. We terminated the tested method once it produced an $\vareps$-stationary point $\overline\vW$, i.e., $\dist(\vzero, \partial F(\overline\vW))\leqslant \vareps$, and $\vareps=10^{-6}$ was set. For both iAPG and APG, we set $\gamma_{\mathrm{inc}}=2$ and $\gamma_{\mathrm{dec}}=\frac{1}{2}$ as in Algorithm~\ref{alg:accellinesearch} if line search is adopted. In addition, for iAPG, the initial inexactness $\vareps_0=10^{-3}$ was set. The results are shown in Table~\ref{table:mult-learn}. Here, $\#g$ represents the number\footnote{We increase $\#g$ by one if $g$ or $\nabla g$ or $(g,\nabla g)$ is called. The same rule is adopted for $\#h$.} of calls to $g$ or $\nabla g$, $\#h$ is the number of calls to $h$ or $\nabla h$, \verb|stat.viol.| denotes $\dist(\vzero, \partial F(\overline\vW))$, and the time is in seconds. From the results, we see that the proposed iAPG requires smaller $\#g$ than the exact APG in all cases. Though iAPG has larger $\#h$ than APG, the former takes shorter time and thus is more efficient. The advantage of iAPG over APG becomes more significant as the problem becomes more difficult, i.e., when $\mu$ is smaller and/or $\lambda_1$ is bigger. These verify our theoretical results.  In addition, even without knowing the smoothness constants, the iAPG by line search has a similar performance to that using the smoothness constants.

\setlength{\tabcolsep}{1pt}

\begin{table}[h]\caption{Results by the proposed iAPG method (i.e., Algorithm~\ref{alg:iAPG}) and its exact counterpart APG on solving 10 independent random instances of the regularized multitask logistic regression \eqref{eq:mult-learn} with different sizes and model parameters. The numbers in the parentheses are the standard deviations.}\label{table:mult-learn}
\vspace{-0.5cm}
\begin{center}
\resizebox{\textwidth}{!}{
\begin{tabular}{|c||cccc|cccc|ccc|ccc|}
\hline
 & \multicolumn{4}{|c|}{iAPG no line search} & \multicolumn{4}{|c|}{iAPG with line search} & \multicolumn{3}{|c|}{APG no line search} & \multicolumn{3}{|c|}{APG with line search} \\\hline\hline  
 $(\mu, \lambda_1)$ & \#$g$ & \#$h$ & stat. viol. & time & \#$g$ & \#$h$ & stat. viol. & time & \#$(g, h)$ & stat. viol. & time & \#$(g, h)$ & stat. viol. & time \\\hline
\multicolumn{15}{|c|}{Problem size: $n = 200, N_l = 500$ for each $l=1,\ldots, 4$} \\\hline
$(0.1,1)$ & 37(0.0) & 546(4.1) & 7.4e-7(7.9e-8) & 0.03 &46(4.0) & 850(66.6) & 7.0e-7(2.1e-7) & 0.04 &103(0.0) & 8.0e-7(3.0e-8) & 0.04 &158(4.2) & 7.5e-7(1.7e-7) & 0.05 \\
$(0.1,10)$ & 37(0.0) & 1815(7.4) & 7.3e-7(7.7e-8) & 0.03 &47(2.6) & 2209(104.9) & 7.0e-7(2.4e-7) & 0.04 &322(1.0) & 9.5e-7(3.1e-8) & 0.09 &604(4.4) & 8.6e-7(9.5e-8) & 0.15 \\
$(0.1,100)$ & 37(0.0) & 5946(37.0) & 7.7e-7(6.7e-8) & 0.06 &48(2.1) & 5226(298.4) & 5.3e-7(2.6e-7) & 0.06 &1038(4.1) & 9.8e-7(9.0e-9) & 0.27 &1584(6.0) & 9.7e-7(1.3e-8) & 0.38 \\\hline
$(0.01,1)$ & 106(1.1) & 1806(13.7) & 8.9e-7(7.4e-8) & 0.05 &106(0.9) & 2313(24.7) & 8.8e-7(8.7e-8) & 0.06 &288(1.0) & 9.6e-7(2.3e-8) & 0.08 &404(1.2) & 9.2e-7(6.0e-8) & 0.10 \\
$(0.01,10)$ & 106(1.0) & 6023(76.8) & 8.6e-7(6.5e-8) & 0.08 &106(0.9) & 5727(211.3) & 8.9e-7(7.6e-8) & 0.08 &874(4.2) & 9.8e-7(1.1e-8) & 0.22 &1643(10.0) & 9.6e-7(3.1e-8) & 0.39 \\
$(0.01,100)$ & 107(0.8) & 19666(189.9) & 8.6e-7(5.5e-8) & 0.16 &107(1.4) & 13381(430.9) & 8.6e-7(1.1e-7) & 0.13 &2775(13.4) & 1.0e-6(3.2e-9) & 0.71 &4248(22.9) & 9.9e-7(8.6e-9) & 1.02 \\\hline\hline
\multicolumn{15}{|c|}{Problem size: $n = 2000, N_l = 5000$ for each $l=1,\ldots, 4$} \\\hline
$(0.1,1)$ & 31(0.0) & 561(0.6) & 5.3e-7(2.1e-8) & 4.5 &38(4.9) & 869(113.3) & 3.4e-7(2.8e-7) & 4.7 &105(0.0) & 8.5e-7(1.9e-8) & 6.4 &165(1.7) & 7.5e-7(9.5e-8) & 7.9 \\
$(0.1,10)$ & 31(0.0) & 1870(5.6) & 5.5e-7(2.0e-8) & 4.5 &41(4.9) & 2149(245.9) & 6.8e-7(3.4e-7) & 4.8 &341(0.6) & 9.6e-7(1.6e-8) & 12.4 &647(0.0) & 8.2e-7(2.1e-8) & 20.0 \\
$(0.1,100)$ & 31(0.0) & 6102(17.2) & 5.6e-7(1.6e-8) & 4.9 &41(6.1) & 5103(854.0) & 4.5e-7(3.5e-7) & 5.0 &1107(2.1) & 9.8e-7(1.0e-8) & 32.0 &1728(8.2) & 9.5e-7(3.7e-8) & 47.2 \\\hline
$(0.01,1)$ & 91(0.6) & 2131(12.5) & 7.9e-7(6.4e-8) & 6.1 &88(0.0) & 2612(7.3) & 7.0e-7(1.8e-8) & 6.0 &319(0.8) & 9.7e-7(2.8e-8) & 11.8 &496(2.2) & 9.2e-7(5.7e-8) & 16.2 \\
$(0.01,10)$ & 91(0.0) & 7099(17.0) & 7.6e-7(2.1e-8) & 6.4 &88(0.0) & 7096(183.2) & 7.0e-7(2.2e-8) & 6.3 &999(0.8) & 9.9e-7(1.0e-8) & 29.2 &1903(4.6) & 9.7e-7(1.9e-8) & 51.7 \\
$(0.01,100)$ & 91(0.0) & 23013(41.8) & 7.6e-7(1.3e-8) & 7.3 &88(0.0) & 18142(281.0) & 7.0e-7(1.5e-8) & 6.9 &3183(4.0) & 9.9e-7(1.6e-9) & 85.1 &4975(14.5) & 9.9e-7(1.2e-8) & 129.1 \\\hline
\end{tabular}
}
\end{center}
\vspace{-0.2cm}
\end{table}

\subsection{Zero-sum constrained LASSO}\label{sec:lasso}
In this subsection, we test the iPALM in Algorithm~\ref{alg:ipalm}, which uses the iAPG in Algorithm~\ref{alg:iAPG} as a subroutine, on solving the zero-sum constrained LASSO \cite{gaines2018algorithms, james2019penalized}:
\vspace{-0.1cm}
\begin{equation}\label{eq:c-lasso}
\min_\vx \textstyle \frac{1}{2}\|\vA\vx - \vb\|^2 + \lambda \|\vx\|_1, \st \frac{1}{\sqrt{n}} \sum_{i=1}^n x_i = 0.
\vspace{-0.1cm}
\end{equation}
Here, $\vA\in\RR^{m\times n}$ and $\vb\in\RR^m$ are given, and we divide by $\sqrt{n}$ in the constraint to normalize the coefficient vector. We name the proposed method as iPALM\_iAPG and compare it to the accelerated primal-dual method, called APD, in \cite{hamedani2021primal}. To apply APD, we solve an equivalent min-max problem by using the ordinary Lagrangian function of \eqref{eq:c-lasso}. For iPALM\_iAPG, we set in Algorithm~\ref{alg:ipalm} $\beta_k=\beta_0\sigma^k, \rho_k = \rho_0\sigma^{-k}$ with $\beta_0=1, \rho_0=10^{-3}, \sigma = 3$, and $\vareps_0=10^{-5}$, $\gamma_{\mathrm{inc}}=3, \gamma_{\mathrm{dec}}=\frac{1}{2}$ in Algorithm~\ref{alg:accellinesearch} if line search is adopted. We set $\tau_0=1$ and $\gamma_0=10^{-3}$ for APD if line search is adopted; see Algorithm~2.3 in \cite{hamedani2021primal}.

In the tests, we set $m=2000, n= 5000$ and fixed $\lambda=10^{-3}$ in \eqref{eq:c-lasso}. Each row of $\vA$ took the form of $\frac{\va}{\|\va\|}$, where $\va$ was generated by the standard Gaussian distribution. We generated a zero-sum sparse vector $\vx^o$ with 200 nonzero components, whose locations were selected uniformly at random. Then we let $\vb = \vA\vx^o + 10^{-3}\frac{\vxi}{\|\vA\vx^o\|}$ with $\vxi$ generated from the standard Gaussian distribution. The stopping tolerance was set to $\vareps=10^{-6}$ to produce an $\vareps$-KKT point. We conducted 10 independent runs. The results are reported in Table~\ref{table:lasso}, where the methods without line search used explicitly-computed smoothness constants to set a constant stepsize. The quantity \#query\_obj denotes the number of queries to $(\vA,\vA^\top)$ and \#query\_cstr the number of times the constraint function in \eqref{eq:c-lasso} is evaluated. The quantities \verb|pres| and \verb|dres| respectively mean the violations of primal and dual feasibility in the KKT system. From the results, we see that the proposed method needs significantly shorter time than the APD method to produce comparable solutions. In addition, both methods with line search performed similarly as well as those without line search.

\setlength{\tabcolsep}{3pt}

\begin{table}[h]\caption{Results by the proposed method iPALM\_iAPG with and without line search and the accelerated primal-dual (APD) method in \cite{hamedani2021primal} with and without line search on solving 10 independent random instances of the constrained LASSO problem \eqref{eq:c-lasso} with $m=2000$ and $n=5000$. The numbers in the parentheses are the standard deviations.}\label{table:lasso}
\vspace{-0.2cm}
\begin{center}
\resizebox{0.75\textwidth}{!}{
\begin{tabular}{|c||c|c||c|c||c|}
\hline
\textbf{Method} & \textbf{\#query\_obj} & \textbf{\#query\_cstr} & \textbf{pres} & \textbf{dres} & \textbf{time}\\\hline\hline
iPALM\_iAPG no line search & 2521(286.3) & 21098(4723.5) & 3.0e-7(2.9e-7) & 6.2e-8(2.1e-10) & 18.2 \\\hline
iPALM\_iAPG with line search & 2962(347.0) & 9760(1200.6) & 3.0e-7(2.9e-7) & 5.2e-8(9.0e-9) & 17.0 \\\hline
APD no line search & \multicolumn{2}{|c||}{7929(606.7)} & 8.8e-10(1.1e-9) & 3.0e-7(2.8e-7) & 51.4 \\\hline
APD with line search & \multicolumn{2}{|c||}{4349(334.8)} & 1.8e-7(2.2e-7) & 2.9e-7(2.8e-7) & 55.3 \\\hline
\end{tabular}
}
\end{center}
\end{table}

\noindent\textbf{Effect by the initial penalty.}~~As the iPALM may benefit from a smaller initial penalty, we test iPALM\_iAPG with different values of $\beta_0$ and compare it to iPALM\_APG that uses the exact APG as the subroutine, in order to further demonstrate the advantage of the proposed method. To eliminate the possible effect by line search, we explicitly compute the smoothness constants and run both methods without line search. All the other settings are the same as those in the previous tests, except the value of $\beta_0$ varies from $\{0.1, 1, 10, 100\}$. The mean and standard deviation results of 10 independent runs are shown in Figure~\ref{fig:lasso}. From the results, we see that \#query\_obj for iPALM\_iAPG is almost not affected by $\beta_0$ while \#query\_cstr for iPALM\_iAPG and \#query\_oracle for iPALM\_APG both increase with $\beta_0$. In addition, the total running time for iPALM\_iAPG is almost not affected by $\beta_0$ either, and even when $\beta_0$ is small, iPALM\_iAPG can still be significantly more efficient than iPALM\_APG to produce a same-accurate KKT solution. 

\vspace{-0.2cm}
\begin{figure}[h]\caption{Mean and standard deviation results by the proposed iPALM\_iAPG and iPALM\_APG with different values of initial penalty parameter $\beta_0$ on solving 10 random instances of the constrained LASSO problem \eqref{eq:c-lasso} with $m=2000$ and $n=5000$.}\label{fig:lasso}
\begin{center}
\includegraphics[width=0.3\textwidth]{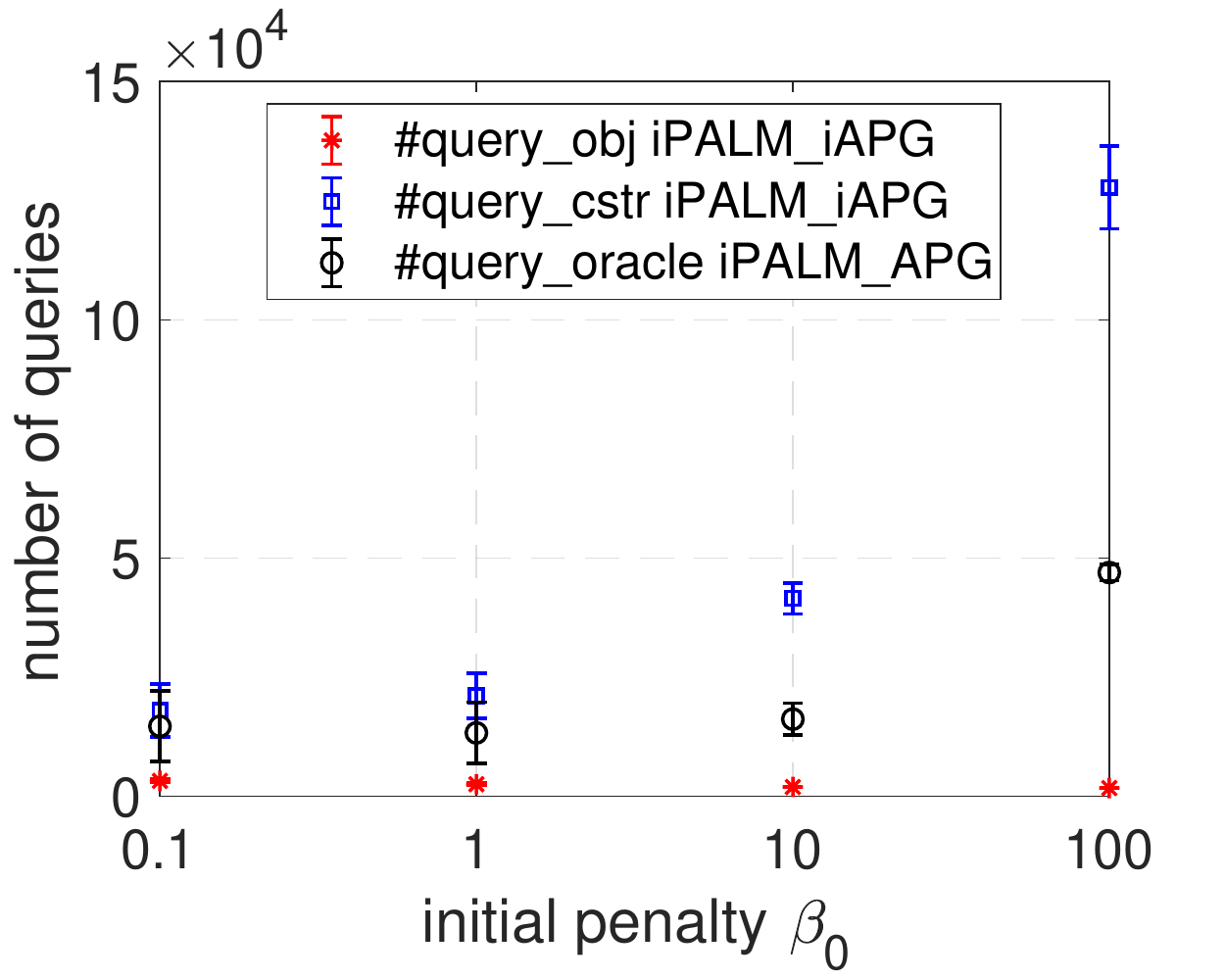}
\includegraphics[width=0.3\textwidth]{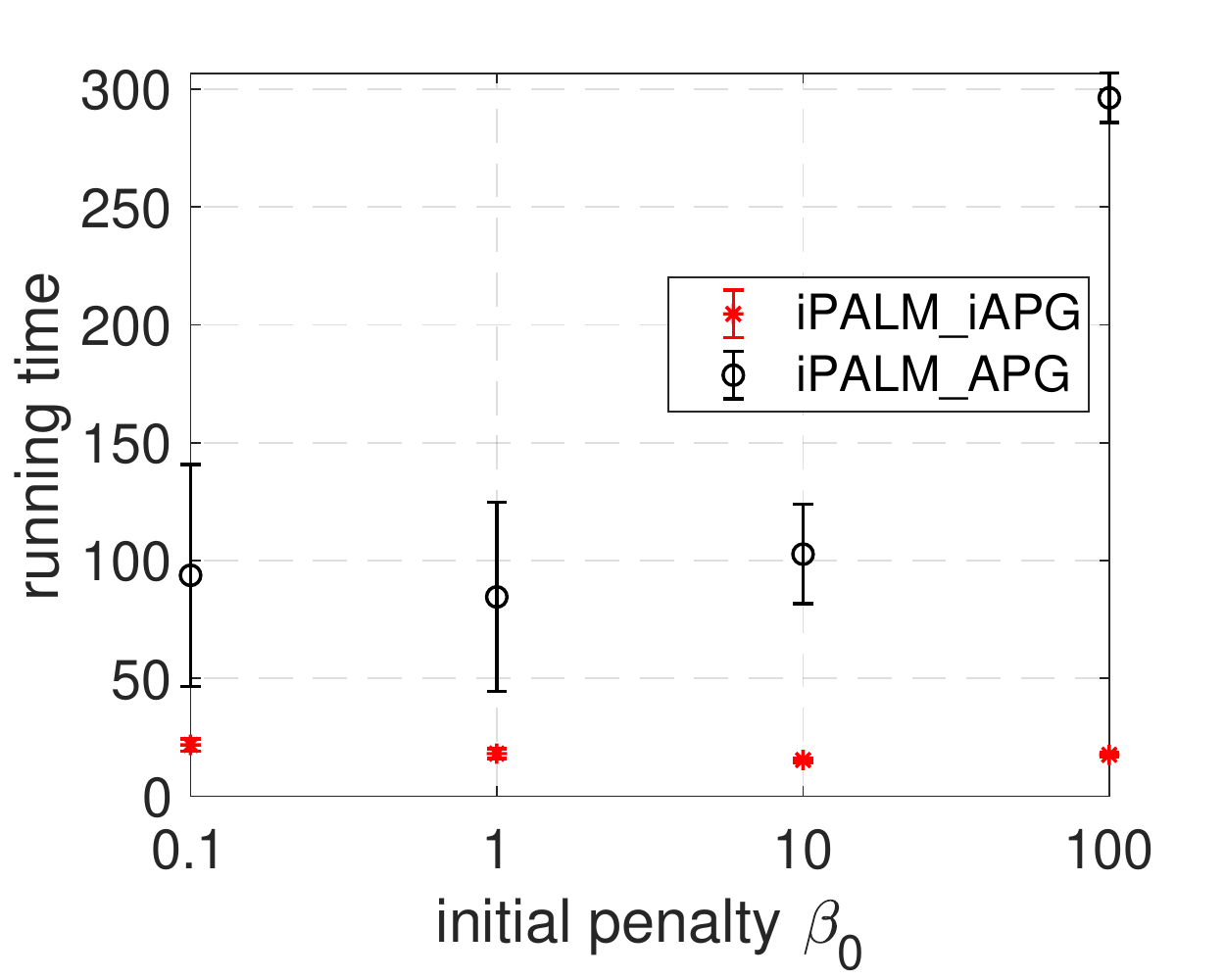}
\end{center}
\vspace{-0.3cm}
\end{figure}

\subsection{Portfolio optimization}
In this subsection, we test the proposed method iPALM\_iAPG on solving the portfolio optimization:
\vspace{-0.1cm}
\begin{equation}\label{eq:portf}
\min_\vx \textstyle \frac{1}{2}\vx^\top \vQ\vx, \st \vx \geqslant \vzero,\ \sum_{i=1}^n x_i \leqslant 1,\ \vxi^\top \vx \geqslant c,
\vspace{-0.1cm}
\end{equation}
where $\vxi$ is the vector of expected return rate of $n$ assets, $\vQ$ is the covariance matrix of the return rates, and $c$ is the imposed minimum total return. We solve instances of \eqref{eq:portf} with synthetic data and the real NASDAQ dataset that has been used in \cite{peng2016coordinate}.

For synthetic data, we set $c=0.02$ and $\vQ= \frac{\vH\vH^\top}{\|\vH\|^2} + \mu\vI$ with $\vH\in\RR^{n\times m}$ generated from the standard Gaussian distribution and $\mu\in \{0, 10^{-3}, 0.1\}$, and the entries of $\vxi$ independently follow uniform distribution on $[-1,2]$. The dimensions of $\vH$ were set to $n=2000$ and $m=1000$, and thus the objective of \eqref{eq:portf} with the generated data is $\mu$-strongly convex. For each value of $\mu$, we generated 10 independent instances. We compared to the APD method in \cite{hamedani2021primal} and the primal-dual sliding (PDS) method in \cite{lan2021graph}. The parameters of iPALM\_iAPG and APD were set to the same values as in section~\ref{sec:lasso}, and the parameters of PDS were set by following \cite[Theorem~2.2]{lan2021graph} with $R=1$. The results are reported in Table~\ref{table:portf_rand}, where \verb|cmpl| represents the amount of violation of complementarity condition in the KKT system, and all other quantities have the same meanings as those in Table~\ref{table:lasso}. From the results, we see that the proposed method iPALM\_iAPG can be significantly more efficient than APD and PDS in terms of the running time. APD with line search is less efficient than APD without line search in this test, while iPALM\_iAPG performed similarly well with and without line search. PDS required significantly fewer queries to the objective but much more queries to the constraint functions. This is because the inner loop of PDS needs to run to a theoretically-determined maximum number of iterations rather than to a computationally-checkable stopping condition.

For the NASDAQ dataset, $\vxi$ is the mean of 30-day return rates. The original covariance matrix $\vQ_0\in\RR^{2730\times 2730}$ is rank-deficient, and in \eqref{eq:portf}, we set  $\vQ=\vQ_0+\mu\vI$ with $\mu\in \{0, 10^{-3}, 0.1\}$. We also set $c=0.02$. These instances have worse condition numbers than the previous randomly generated ones. Hence, besides setting a stopping tolerance to $\vareps=10^{-6}$, we also set a maximum running time to one hour. We found that APD with line search did not work well for these instances, possibly because of the rounding error during the line search. Hence, we only reported the results of APD without line search by explicitly computing the smoothness constants and setting constant stepsizes. The results by all compared methods are shown in Table~\ref{table:portf_real}. Again, we see that the proposed method iPALM\_iAPG was significantly more efficient than APD and PDS in terms of running time. For the hardest case that corresponds to $\mu=0$, APD and PDS both failed to reach the desired accuracy within one hour. Similar to the instances with synthetic data, PDS required much more queries to the constraint functions, though its queries to the objective was significantly fewer than the proposed method.

\setlength{\tabcolsep}{2pt}

\vspace{-0.2cm}
\begin{table}[h]\caption{Results by the proposed method iPALM\_iAPG with and without line search, the accelerated primal-dual (APD) method in \cite{hamedani2021primal} with and without line search, and the primal-dual sliding (PDS) method in \cite{lan2021graph} on solving 10 instances of the portfolio optimization \eqref{eq:portf} with independently generated synthetic data. The numbers in the parentheses are the standard deviations.}\label{table:portf_rand}
\vspace{-0.5cm}
\begin{center}
\resizebox{\textwidth}{!}{
\begin{tabular}{|c||c||c|c||c|c|c||c|}
\hline
&\textbf{Method} & \textbf{\#query\_obj} & \textbf{\#query\_cstr} & \textbf{pres} & \textbf{dres} & \textbf{cmpl} & \textbf{time}\\\hline\hline
\multirow{5}{*}{\begin{sideways}$\mu= 0$\end{sideways}} 
&iPALM\_iAPG no line search & 2709(28.8) & 315480(19209.5) & 0.0e+00 & 8.2e-7(1.2e-7)  & 0.0e+00 & 7.2\\\cline{2-8}
&iPALM\_iAPG with line search & 3172(68.2) & 14984(403.7) & 0.0e+00 & 8.2e-7(1.2e-7)  & 0.0e+00 & 3.2\\\cline{2-8}
&APD no line search & \multicolumn{2}{|c||}{9229(772.6)} & 0.0e+00 & 8.2e-7(1.2e-7) & 0.0e+00 & 7.0\\\cline{2-8}
&APD with line search & \multicolumn{2}{|c||}{11736(1109.8)} & 0.0e+00 & 8.2e-7(1.2e-7) & 0.0e+00 & 10.9 \\\cline{2-8}
&PDS & 1578(514.6) & 18738280(12011557.5) & 3.5e-19(1.1e-18) & 8.1e-7(1.2e-7)  & 3.2e-27(1.0e-26) & 278.7 \\\hline\hline
\multirow{5}{*}{\begin{sideways}$\mu= 10^{-3}$\end{sideways}}
& iPALM\_iAPG no line search & 1451(20.4) & 183307(4012.1) & 4.3e-8(8.9e-9) & 6.8e-7(8.3e-8)  & 1.7e-15(6.8e-16) & 4.3\\\cline{2-8}
&iPALM\_iAPG with line search & 1782(25.5) & 50905(2211.4) & 4.2e-8(8.8e-9) & 7.0e-7(7.9e-8)  & 1.7e-15(6.7e-16) & 2.6\\\cline{2-8}
&APD no line search & \multicolumn{2}{|c||}{20931(1489.6)} & 3.4e-9(1.1e-9) & 6.9e-7(7.9e-8) & 1.6e-16(4.9e-17) & 15.1\\\cline{2-8}
&APD with line search & \multicolumn{2}{|c||}{26949(2055.4)} & 2.5e-9(9.1e-10) & 6.9e-7(7.9e-8) & 1.2e-16(4.0e-17) & 24.1\\\cline{2-8}
&PDS & 526(20.0) & 5008697(786935.0) & 1.7e-18(1.8e-18) & 6.8e-7(8.0e-8)  & 1.7e-25(6.4e-26) & 74.9\\\hline\hline
\multirow{5}{*}{\begin{sideways}$\mu= 0.1$\end{sideways}}
&iPALM\_iAPG no line search & 243(1.9) & 24951(497.0) & 1.4e-7(1.3e-8) & 6.4e-8(5.4e-9)  & 2.8e-13(2.6e-14) & 1.2\\\cline{2-8}
&iPALM\_iAPG with line search & 262(1.0) & 22420(703.6) & 1.4e-7(1.3e-8) & 6.9e-8(1.4e-9)  & 2.8e-13(2.6e-14) & 1.2\\\cline{2-8}
&APD no line search & \multicolumn{2}{|c||}{36869(2361.9)} & 5.8e-14(1.8e-14) & 1.4e-7(1.3e-8) & 1.1e-19(3.6e-20) & 26.3\\\cline{2-8}
&APD with line search & \multicolumn{2}{|c||}{59976(5195.3)} & 2.9e-14(1.2e-14) & 1.4e-7(1.3e-8) & 5.7e-20(2.4e-20) & 53.3\\\cline{2-8}
&PDS & 110(1.8) & 2216395(309914.2) & 6.9e-19(1.5e-18) & 1.3e-7(1.5e-8)  & 4.1e-24(3.5e-24) & 33.2\\\hline
\end{tabular}
}
\end{center}
\end{table}

\setlength{\tabcolsep}{3pt}

%\vspace{-0.2cm}
\begin{table}[h]\caption{Results by the proposed method iPALM\_iAPG with and without line search, the accelerated primal-dual (APD) method in \cite{hamedani2021primal} without line search, and the primal-dual sliding (PDS) method in \cite{lan2021graph} on solving instances of the portfolio optimization \eqref{eq:portf} with NASDAQ data.}\label{table:portf_real}
\vspace{-0.2cm}
\begin{center}
\resizebox{0.8\textwidth}{!}{\small
\begin{tabular}{|c||c||c|c||c|c|c||c|}
\hline
&\textbf{Method} & \textbf{\#query\_obj} & \textbf{\#query\_cstr} & \textbf{pres} & \textbf{dres} & \textbf{cmpl} & \textbf{time}\\\hline\hline
\multirow{4}{*}{\begin{sideways}$\mu= 0$\end{sideways}} 
&iPALM\_iAPG no line search & 112704  & 5530144  & 0.0e+00  & 4.2e-07   & 9.2e-19  & 350.6\\\cline{2-8}
&iPALM\_iAPG with line search & 37235  & 715328  & 0.0e+00  & 4.2e-07   & 0.0e+00  & 97.0\\\cline{2-8}
&APD no line search & \multicolumn{2}{|c||}{1118808}  & 0.0e+00  & 1.3e-06  & 2.2e-17  & 3603.8\\\cline{2-8}
&PDS & 54058  & 176318909  & 3.5e-18  & 1.1e-06   & 1.5e-25  & 3604.0 \\\hline\hline
\multirow{4}{*}{\begin{sideways}$\mu= 10^{-3}$\end{sideways}}
&iPALM\_iAPG no line search & 21314  & 375994  & 0.0e+00  & 2.3e-07   & 7.0e-14  & 54.1\\\cline{2-8}
&iPALM\_iAPG with line search & 48643  & 117194  & 0.0e+00  & 2.3e-07   & 7.1e-14  & 108.4\\\cline{2-8}
&APD no line search & \multicolumn{2}{|c||}{1119046}  & 0.0e+00  & 8.5e-07  & 4.9e-18  & 3603.6\\\cline{2-8}
&PDS & 6278  & 8927446  & 0.0e+00  & 2.2e-07   & 0.0e+00  & 195.5\\\hline\hline
\multirow{4}{*}{\begin{sideways}$\mu= 0.1$\end{sideways}}
&iPALM\_iAPG no line search & 3206  & 32178  & 4.4e-09  & 6.2e-08   & 4.8e-13  & 10.8\\\cline{2-8}
&iPALM\_iAPG with line search & 6601  & 16451  & 5.2e-09  & 6.2e-08   & 5.6e-13  & 17.8\\\cline{2-8}
&APD no line search & \multicolumn{2}{|c||}{1119360}  & 0.0e+00  & 9.0e-08  & 2.6e-21  & 3603.6\\\cline{2-8}
&PDS & 1404  & 29512311  & 0.0e+00  & 5.6e-08   & 0.0e+00  & 591.7\\\hline
\end{tabular}
}
\end{center}
\vspace{-0.2cm}
\end{table}

\section{Conclusions}\label{sec:conclusion}
We have presented an inexact accelerated proximal gradient (iAPG) method for solving structured composite convex optimization, which have two smooth components with significantly different computational costs. When the more costly component has a significantly smaller smoothness constant than the less costly one, the proposed iAPG can significantly reduce the overall complexity than its exact counterpart, by querying the more costly component less frequently than the less costly one. Using the iAPG method as a subroutine, we proposed gradient-based methods for solving affine-constrained composite convex optimization and for solving bilinear saddle-point structured nonsmooth convex optimization. Our methods can have significantly lower overall complexity than existing methods when the coefficient matrix (in the affine constraint or in the bilinear term) permits matrix-vector multiplications with low cost.

\bibliographystyle{abbrv}
\bibliography{ref}

\begin{thebibliography}{10}
\bibitem{xu2020first}
Y.~Xu.
\newblock First-order methods for problems with ${O} (1)$ functional
  constraints can have almost the same convergence rate as for unconstrained
  problems.
\newblock {\em arXiv preprint arXiv:2010.02282}, 2020.
\end{thebibliography}

\end{document}